\documentclass[11pt]{article}
\usepackage{dsfont}
\usepackage{amsmath, amsthm, amssymb, mathrsfs, eufrak}
\usepackage[abbrev,non-sorted-cites]{amsrefs}
\usepackage[all]{xy}
\usepackage{CJK}
\usepackage{grffile}
\usepackage{graphicx}
\newtheorem{thm}{Theorem}[section]
\newtheorem{cor}[thm]{Corollary}
\newtheorem{lem}[thm]{Lemma}
\newtheorem{ex}[thm]{Example}
\newtheorem{definition}[thm]{Definition}
\newtheorem{rmk}[thm]{Remark}

\newtheorem{prop}[thm]{Proposition}
\newtheorem{conj}[thm]{Conjecture}

\begin{document}

\title{Open Gromov-Witten Invariants on Elliptic K3 Surfaces and Wall-Crossing}
\author{Yu-Shen Lin}

\maketitle
\section{Introduction}
 
 Gromov \cite{G} invented the techniques of pseudo-holomorphic curves and soon became an important tool to study symplectic topology. In this paper, we will provide a version of open Gromov-Witten invariants $\tilde{\Omega}^{Floer}(\gamma;u)$ for elliptic K3 surfaces. Naively, $\tilde{\Omega}^{Floer}(\gamma;u)$ counts the holomorphic discs with boundary on a special Lagrangian $L_u$ (after hyperK\"ahler rotation). The open Gromov-Witten invariant $\tilde{\Omega}^{Floer}(\gamma;u)$ locally with respect to the boundary condition of holomorphic discs. The base is locally divided into chambers and the invariants may jump when the torus fibre move from one chamber to another. 
  
Despite the interest in symplectic geometry, the counting of holomorphic discs also play an important role in mirror symmetry. Strominger-Yau-Zaslow \cite{SYZ}\footnote{Later Kontsevich-Soibelman \cite{KS4} modified the original Strominger-Yau-Zaslow conjecture to the following statement: Calabi-Yau manifolds will converge in the sense of Gromov-Hausdorff to affine manifolds with singularities.} conjectured that Calabi-Yau manifolds will admit special Lagrangian fibration near the large complex limit and the mirror Calabi-Yau manifolds can be constructed via dual torus fibration. The conjecture achieves big success toward the understanding of mirror symmetry when the Calabi-Yau manifolds are trivial torus fibration, which called the semi-flat case\cite{L8}. When the special Lagrangian fibrations admit singular fibres, the semi-flat complex structures on the mirror do not extend over the singular fibres and thus need certain "correction". The corrected of the mirror complex structure is constructed by Kontsevich-Soibelman \cite{KS1} for K3 surface and Gross-Seibert \cite{GS1} for Calabi-Yau manifolds with toric degenerations. It is conjectured that the quantum correction of the mirror complex structures come from holomorphic discs \cite{F3}\cite{KS1}\cite{GS1}. The conjecture is verified for toric Calabi-Yau manifolds \cite{A}\cite{CLL}. Along the line of thought, Gaiotto-Moore-Neitzke \cite{GMN} proposed a recipe of explicit expression of hyperK\"ahler metric for hyperK\"ahler manifolds with abelian fibrations. The key ingredient of the recipe is an invariant which satisfies the Kontsevich-Soibelman wall-crossing formula. In this article, we construct an explicit example of the above wall-crossing phenomenon of holomorphic discs via gluing. Moreover, we prove the open Gromov-Witten invariants satisfy Kontsevich-Soibelman wall-crossing formula under the assumption of discs classes are primitive.
  \begin{thm}(=Theorem \ref{32}) Assume $\gamma_1,\gamma_2$ are primitive classes, and $\gamma=\gamma_1+\gamma_2$ is the unique splitting of holomorphic discs. Then the difference of the open Gromov-Witten invariant $\tilde{\Omega}^{Floer}(\gamma)$ on the different side of the wall is given by
     \begin{align*}
         \Delta \tilde{\Omega}^{Floer}(\gamma)=|\langle\partial\gamma_1,\partial\gamma_2\rangle| \tilde{\Omega}^{Floer}(\gamma_1)\tilde{\Omega}^{Floer}(\gamma_2).
     \end{align*}
  \end{thm}
 
 Tropical geometry also arises naturally from the point of view of Strominger-Yau-Zaslow conjecture. It is speculated that the holomorphic curves (with boundaries) project to amoebas under the fibration. As the Calabi-Yau manifolds approach the large complex limit points, the amoebas will converge to $1$-skeletons on the base, known as the tropical curves. One advantage of introducing the notation tropical geometry is that one can reduce an enumerative problem into a combinatoric problem. Mikhalkin \cite{M2} first establish such correspondence theorem between tropical curves and holomorphic curves on toric surfaces. Later it is generalized to toric manifolds by Nishinou-Seibert \cite{NS}. Nishinou \cite{N} also proved a correspondence of counting tropical discs and holomorphic discs on toric manifolds. In this article, we introduce the tropical discs on K3 surfaces and the associated weights. In particular, there exists a tropical discs counting invariants $\tilde{\Omega}^{trop}(\gamma;u)$ counts the tropical discs end on $u$ with the weight. One application of the open Gromov-Witten invariants is that we establish a weaker version of correspondence theorem: 
   \begin{thm} (=Theorem \ref{47})
      If $\tilde{\Omega}^{Floer}(\gamma;u)\neq 0$ for some relative class $\gamma\in H_2(X,L_u)$, then there exists a tropical disc represents $\gamma$. 
   \end{thm}
 Together with the primitive wall-crossing formula for open Gromov-Witten invariants on K3 surfaces, we prove in certain cases there exists a correspondence theorem for tropical discs counting/open Gromov-Witten invariants.

\section*{Outline of the Paper}
In Section \ref{1036}, we review the basic facts about hyperK\"ahler geometry. In section 3, we first talk about tropical geometry on the base elliptic K3. After that, we construct a scattering diagram on the base of elliptic K3 surface and use it to define tropical version of Donaldson-Thomas invariants $\Omega^{trop}$. We also discuss the relation between the invariants $\Omega^{trop}$ and counting of tropical discs. In section 4, we will study holomorphic discs inside K3 surfaces with special Lagrangian fibration. With the help of auxiliary $S^1$-family of complex structures coming from twistor family, we define the open Gromov-Witten invariants for elliptic K3 surfaces. We calculate a multiple cover formula for this new invariant using localization theorem. At the end, we derive a corresponding theorem between holomorphic discs and tropical discs.

\section*{Acknowledgements} The author would like to thank Shing-Tung Yau for constant support and encouragement. The author is also indebted to Kenji Fukaya, Ono, Andrew Neitzke, Wenxuan Lu for explaining their related work. The author would like to thank Chiu-Chu Melissa Liu for helpful discussion and teaching him the localization techniques. The author would also like to thank Yan Soibelman for pointing out many useful references.

\section{Review of HyperK\"ahler Manifolds}\label{1036}
\subsection{HyperK\"ahler Manifolds and HyperK\"ahler Rotation Trick}
\begin{definition}
 A complex manifold $X$ of dimension $2n$ is called a hyperK\"ahler manifold if its holonomy group falls in $Sp(n)$.
\end{definition}
The holonomy of a hyperK\"ahler manifold guarantee the existence of a covariant constant holomorphic symplectic $2$-form, i.e., a holomorphic non-degenerate closed $2$-form.
\begin{ex}
  Every compact complex K\"ahler manifold admits a holomorphic symplectic $2$-form and therefore is hyperK\"ahler \cite{Y1}. In particular, K3 surfaces and Hilbert scheme of points on K3 surfaces are hyperK\"ahler. 
\end{ex}
\begin{ex} There are also non-compact examples of hyperK\"ahler manifolds such as Hitchin moduli spaces and Ooguri-Vafa space ( Section \ref{1001}).
\end{ex}

Let $X$ be a hyperK\"ahler manifold, $\Omega$ be its holomorphic symplectic $2$-form and $\omega$ be a K\"ahler form such that $g$ is the corresponding
Ricci-flat metric from the definition. Write
  \begin{align*}
    \omega_1=\mbox{Re}\Omega,
    \hspace{5mm}\omega_2=\mbox{Im}\Omega,\hspace{5mm}
    \omega_3=\omega.
  \end{align*}
and
  \begin{align*}
     g(\cdot,\cdot)=\omega_i(\cdot,J_i \cdot), \mbox{ for }i=1,2,3.
  \end{align*}
Then it is straight-forward to check that $J_1$, $J_2$ and $J_3$ are
integrable complex structures satisfying quaternionic relations.
In particular, $\underline{X}$ admits a family of
complex structures parametrized by $\mathbb{P}^1$, called the twistor line. Explicitly, they are given
by
\begin{equation*}
   J_{\zeta}=\frac{i(-\zeta+\bar{\zeta})J_1-(\zeta+\bar{\zeta})J_2+(1-|\zeta|^2)J_3}{1+|\zeta|^2},
   \hspace{3mm}
   \zeta\in \mathbb{P}^1.
\end{equation*} We will denote $\underline{X}$ with the complex structure $J_{\zeta}$ by $X_{\zeta}$.
 The holomorphic
symplectic $2$-forms $\Omega_{\zeta}$ with respect to the compatible
complex structure $J_{\zeta}$ are given by
\begin{equation}  \label{38}
  \Omega_{\zeta}=-\frac{i}{2\zeta}\Omega+\omega_3-\frac{i}{2}\zeta\bar{\Omega}.
\end{equation} In particular, straightforward computation gives
\begin{prop} \label{49}
Assume $\zeta=e^{i\vartheta}$, then we have
  \begin{align*}
     \omega_{\vartheta}:=\omega_{\zeta}&=-\mbox{Im}(e^{-i\vartheta}\Omega),\\
     \Omega_{\vartheta}:=\Omega_\zeta&=\omega_3-iRe(e^{-i\vartheta}\Omega).
  \end{align*}
\end{prop}

The complex structures on $X_{\zeta}$, $\zeta\in \mathbb{P}^1$ actually glue together to a complex dimension $3$ manifold. The manifold is called twistor space and topologically is just $\underline{X}\times \mathbb{P}^1$. We will denote $\mathcal{X}$ to be the twistor space with $X_{0},X_{\infty}$ deleted for the later on purpose.

\begin{rmk} \label{300}
   Let $L$ be a holomorphic Lagrangian in $(X,\omega,\Omega)$, namely,
   $\Omega|_L=0$. Assume the
north and south pole of the twistor line are given by
$(\omega,\Omega)$ and $(-\omega,\bar{\Omega})$ respectively, making
$L$ an holomorphic Lagrangian. The hyperK\"ahler structures
corresponding to the equator $\{\zeta=e^{i\vartheta}:|\zeta|=1\}$
make $L$ a special Lagrangian in
$X_{\vartheta}=(\underline{X},\omega_{\vartheta},\Omega_{\vartheta})$.
In particular, if $(\underline{X},\omega,\Omega)$ admits holomorphic
Lagrangian fibration, then it induces an $S^1$-family of special
Lagrangian fibrations on $X_{\vartheta}$ for each $\vartheta\in S^1$.
This is the so-called hyperK\"ahler rotation trick.
\end{rmk}

\subsection{Twistorial Data} \label{1003}
Let $f:X \rightarrow B\cong \mathbb{P}^1$ be an elliptic K3 surface with a holomorphic symplectic $2$-form $\Omega$ and a prescribed K\"ahler class $[\omega]$. From Yau's theorem there exists a unique Ricci-flat K\"ahler form
$\omega$ satisfying the Monge-Amp\`ere equation
$2\omega^2=\Omega\wedge \bar{\Omega}$. The triple $(X, \omega,
\Omega)$ will induce a twistor family of K3 surfaces. Let $\Delta\subseteq B$ be the discriminant locus (also referred as singularity of the affine structures later on) of the fibration and $B_0=B\backslash \Delta$. We will denote the fibre over $u\in B$ by $L_u$.We have the follow long exact sequence of local system of lattices
\begin{equation} \label{1002}
     \bigcup_{u\in B_0}H_2(X) \rightarrow \Gamma:=\bigcup_{u\in B_0} H_2(X, L_u) \rightarrow \Gamma_g:= \bigcup_{u\in B_0} H_1(L_u) \rightarrow 0.
\end{equation}
  The natural non-degenerate symplectic pairing $\langle, \rangle$ on $\Gamma_g$ lifts to a degenerate symplectic pairing on $\Gamma$ with kernel $ \bigcup_{u\in B_0}H_2(X)$. 
\begin{definition}
  The central charge is a homomorphism from the local system of lattices $\Gamma$ to $\mathbb{C}$ defined by 
     \begin{align*}
        Z :&\Gamma\longrightarrow \mathbb{C} \\
        &    \gamma_u \mapsto \int_{\gamma_u}\Omega
     \end{align*} for each $\gamma_u\in H_2(X,L_u)$.
\end{definition}The integral is well-defined because
$\Omega|_L=0$. The following lemma is straight forward computation:

\begin{lem} \label{34}
    For any $v \in TB_0$, we have
   \begin{equation}
   dZ_{\gamma}(v)=\int_{\partial \gamma}\iota_{\tilde{v}}\Omega,
   \end{equation}
   where $\tilde{v} \in TX$ is any lifting of $v$.
\end{lem}
\begin{proof}Since $\Omega|_L=0$, we view $\Omega$ as the
   element $(\Omega,0)\in H^2(X,L)$. From the variational formula of
   relative pairing,
   \begin{align*}
      dZ_{\gamma}(v)&=\mathcal{L}_v \langle \gamma,(\Omega,0)\rangle \\
                    &=\langle \gamma, (\iota_{\tilde{v}}d\Omega,
                    \iota_{\tilde{v}}(0-\Omega))\rangle=\int_{\partial
                    \gamma}\iota_{\tilde{v}}\Omega
   \end{align*}
\end{proof}

Given a point $u_0\in B_0$ and an element $\gamma_{u_0}\in
\Gamma_{u_0}$, there exists a neighborhood $\mathcal{U}$ of $u_0$
and a neighborhood of $\tilde{\mathcal{U}}$ of $\gamma_{u_0}$ such
that $\mathcal{U}$ is homeomorphic to $\tilde{\mathcal{U}}$. Under
this identification, we have
\begin{cor}
  The central charge $Z$ is a
  holomorphic function on $\Gamma$.
\end{cor}
\begin{proof}
  Since any $(0,1)$-vector on $TB_0$ can be expressed in term of
  $v+iJv$ for some $v\in TB_0$,
     \begin{align*}
        (v+iJv)Z=\int_{\partial
        \gamma}\iota_{(\tilde{v}+iJ\tilde{v})}\Omega=0.
     \end{align*}The latter equality holds because $\Omega$ is a
     $(2,0)$-form and $\tilde{v}+iJ\tilde{v}$ is a $(0,1)$-vector.
\end{proof}

\begin{cor} \label{350}
   Let $\gamma\in \Gamma$, then $dZ_{\gamma}(u)\neq 0$ if $u\in B_0$. Moreover, if $\gamma$ is the Lefschetz thimble around an $I_1$-singularity, then both $Z_{\gamma}$ and $dZ_{\gamma}$ extends to $0$ over the singularity.
\end{cor}
\begin{proof}
  The first part follows immediately from Lemma \ref{34} and $\mbox{Re}(e^{i\vartheta}\Omega)$ is always a symplectic $2$-form. The second part is because $Z_{\gamma}$ has a removable singularity at the origin. 
\end{proof}

\begin{cor} \label{1012}
  $\langle dZ, dZ \rangle=0$ and $\langle dZ, d\bar{Z} \rangle >0$.
\end{cor}
\begin{proof}
   For elliptic K3 with a holomorphic section, there is a standard short exact sequence of sheaves
 \begin{equation}
     0 \rightarrow R^1 f_* \mathbb{Z} \rightarrow R^1 f_* \mathcal{O}_X \cong \omega_{\mathbb{P}^1} \rightarrow \mathcal{O}^{\#} \rightarrow 0,
 \end{equation}
where $\mathcal{O}^{\#}$ denotes the sheaf of holomorphic sections
of $f:X \rightarrow \mathbb{P}^1$. Here $ R^1 f_* \mathcal{O}_X$ is
identified as the normal bundle of the zero section and the last map
is the fibrewise exponential map. The total space of $ R^1 f_* \mathcal{O}_X$ can be identified as the holomorphic cotangent bundle of $\mathbb{P}^1$ and admits a natural holomorphic symplectic $2$-form $\Omega_{can}$ which descends to the quotient. On the other hand, any holomorphic
symplectic $2$-form is a multiple of $\Omega_{can}$. Therefore, the
holomorphic volume form of an elliptic K3 surface coincides with the
one descending from the canonical volume form of the cotangent
bundle of the base by Hartogs' extension theorem. The proposition follows from direct computations
and Lemma \ref{34}. For an elliptic K3 surface $X$ without a holomorphic section, the Jacobian $X^{\#}$ is another elliptic K3 surface with a holomorphic section. One can identify the underlying space of $X$ and $X^{\#}$ fibre. We will denote the holomorphic volume form of $X$ and $X^{\#}$ by $\Omega_{can}$ and $\Omega$.
Then locally there exists a flat section $\sigma$ of the torus fibration such that
$\Omega=T_{\sigma}^*\Omega_{can}$, where $T_{\sigma}$ is the translation $\sigma$ \cite{GW}. The corollary follows from the fact that the pairing $\langle, \rangle$ is preserved by the operation. 
\end{proof}

\begin{rmk}In particular, the non-vanishing holomorphic $2$-form of
an elliptic K3 surface receives no quantum correction and admits
local $S^1$-action near singularities and local $T^2$-action away
from singularities.
\end{rmk}

The twistorial data consists of an hyperK\"ahler manifold with holomorphic Lagrangian fibration, a short exact sequence of lattices as (\ref{1002}), a holomorphic function $Z$ called central charge satisfying Corollary \ref{1012} and $\{dZ_{\gamma}\}_{\gamma\in \Gamma}$ span $T^*B$ pointwisely.

By the Strominger-Yau-Zaslow conjecture, the semi-flat metric approximates the true Ricci-flat metric near the large complex limit point. However, the curvature of the semi-flat metric blows up near the singular fibres and thus cannot be extended to the whole K3 (or a general) hyperK\"ahler manifold $X$. To remedy this defect, one has to introduce the "quantum corrections". From the result of Hitchin \cite{HKLR}, the explicit expression of hyperK\"ahler metric of a hyperK\"ahler manifold can be achieved from holomorphic symplectic $2$-forms with respect to all complex structures parametrized by the twistor line. To achieve this goal, the idea is to glue pieces of flat space with standard holomorphic symplectic $2$-form via certain symplectormorphisms, which are determined by the generalized Donaldson-Thomas invariants. These invariants are locally some integer-valued functions depending on the charge (thus depending on the base). There are so-called walls of marginal stability which separate the base of abelian fibration into chambers locally. The invariants are constants inside the chamber while might jump when across the wall. The jump of the invariants cannot be arbitrary but governed by the Kontsevich-Soibelman wall-crossing formula \cite{KS2}, which suggests the compatibility of the gluing flat pieces with a global holomorphic symplectic $2$-form. The Kontsevich-Soibelman wall-crossing formula is then interpreted as the smoothness of the holomorphic symplectic $2$-form. 

Given the twistorial data, Gaiotto-Moore-Neitzke found out the correct notion of generalized Donaldson-Thomas invariants and carried out the above scheme for Hitchin moduli spaces \cite{GMN2} (but outside of singular fibres). It still remains open for the construction of hyperK\"ahler metric for general abelian fibred hyperK\"ahler manifolds, especially compact ones such as K3 surfaces. Follow the the recipe in \cite{GMN}, the key step is to find the corresponding generalized Donaldson-Thomas invariants satisfying various properties. Here, we introduce the reduced open Gromov-Witten invariants in Section 4, which are conjectured to satisfy the Kontsevich-Soibelman wall-crossing formula and serve this purpose. 

\section{Tropical Geometry and Generalized Donaldson-Thomas Invariants on Elliptic K3 Surfaces}
\subsection{Tropical Discs}
Tropical geometry raises naturally from the point of view of modified Strominger-Yau-Zaslow conjecture \cite{GW}\cite{KS4}\cite{SYZ}. Naively, the special Lagrangian fibration in a Calabi-Yau manifold collapses to the base affine manifold when the Calabi-Yau manifold goes to the large complex limit point. The projection of holomorphic curves are amoebas and degenerate to some $1$-skeletons at the limit, which are called "tropical curves" in modern terminology. This idea has been carried out for toric varieties by Mikhalkin \cite{M2} and Lagrangian bundles with no monodromy by Parker \cite{P}. Here we are going to discuss the case where special Lagrangian fibration comes from K3 surfaces, which admits singular fibres and we study holomorphic discs instead of holomorphic curves. 

Throughout the paper unless specified, we will assume $X$ is an elliptic fibred K3 surface with $24$ $I_1$-type singular fibres. Let $B$ be the base of an elliptic fibration, then we have an $S^1$-family of integral affine structures on $B_0$. Indeed, for any $\vartheta \in S^1$ and lifting $\gamma_i$ of the generators of $\Gamma_g$, the functions $f_i=-Re(e^{-i\vartheta}Z_{\gamma_i})$ give the local  affine coordinates with transition functions in $Sp(2,\mathbb{Z})\ltimes \mathbb{R}^2$. We will denote $B$ together with this integral affine structure by $B_{\vartheta}$. In particular, neither the choice of K\"ahler class $[\omega]$ of $X$ nor the real scaling of the holomorphic $(2,0)$-form $\Omega$ change the affine straight lines on the base affine manifold. One advantage of introducing the affine structure is allowing us to discuss tropical geometry on $B$.

\begin{rmk}
  In the context of mirror symmetry, the affine structure described above is usually called the complex affine structure of the special Lagrangian in $X_{\vartheta}$. If the phase $\vartheta$ in $f_i$ is replaced by $\vartheta-\pi/2$ then it gives the so-called symplectic affine structure of $X_{\vartheta}$.
\end{rmk}   
  
 In particular, the following map
   \begin{align}\label{1009}
               (&\Gamma_g)_u \longrightarrow \hspace{7mm} T^*_uB_{\vartheta} \\ \nonumber
                  &\partial \gamma \mapsto   \big(v \in T_uB \mapsto \int_{\partial\gamma}\iota_{\tilde{v}}\mbox{Im}(e^{-i\vartheta}\Omega) \big)
   \end{align}is injective and full rank. Therefore, the set $\{dZ_{\gamma}, \gamma\in \Gamma\}$ spans $T^*B$. Dually, the above map induces an integral structure on $T_u^*B_{\vartheta}$ and thus on also induces one on $T_uB_{\vartheta}$. One can identify $\Gamma_g\otimes \mathbb{Q}$ with rational points in $TB_{\vartheta}$ using the map (\ref{1009}) and the unimodular symplectic pairing $\langle,\rangle$. The following is a temporary definition of tropical discs modified from \cite{G3}\cite{N}. The main difference from the usual definition of tropical discs in the literature is that we allow an edge maps to a point. This is due to the presence of singularities of the affine structure and is necessary for the tropical counting defined later to match with the corresponding geometry.
     
\begin{definition} \label{322}
    Let $B$ be an affine manifold with at worst singularity of monodromy conjugate to $\bigl(
    \begin{smallmatrix}
      1 & 1\\
      0 & 1
    \end{smallmatrix} \bigr)$ and with integral structure on $TB$. Let $B_0$ be the complement of the singularities $\Delta$. A tropical curve (with stop) on $B_{\vartheta}$ is a $3$-tuple $(\phi,G,w)$ where $G$ is a rooted connected graph (with a root $x$).  We denote the set of vertices and edges by $G^{[0]}$ and $G^{[1]}$ respectively, with a weight function $w:G^{[1]}\rightarrow \mathbb{N}$. And $\phi:G\rightarrow B$ is a continuous map such that 
  \begin{enumerate}
    \item We allow $G$ to have unbounded edges only when $B$ is non-compact.
    \item For each $e\in G^{[1]}$, $\phi|_e$ is either an embedding of affine segment on $B_0$ or $\phi |_e$ is a constant map. In the later case, $e$ is associated with a non-zero tangent direction (up to sign).
    \item For the root $x$, $\phi(x)\in B_0$.
    \item For each $v\in G^{[0]}$, $v\neq x$ and $\mbox{val}(v)=1$, we have $\phi(v)\in \Delta$. Moreover, if $\phi(v)$ corresponds to an $I_1$-type singular fibre \footnote{Straight forward computation shows that the monodromy invariant direction of the affine structure is rational. There might be other constraints for other kind of singularities. However, the author does not know the appropriate local model at the time.}, then the image of edge adjacent to $v$ is in the monodromy invariant direction. 
    \item For each $v\in G^{[0]}$, $\mbox{val}(v)\geq 1$, we have the following assumption: 
    
      (balancing condition) Each outgoing tangent at $u$ along the image of each edge adjacent to $v$ is rational with respect to the above integral structure on $T_{\phi(v)}B$. Denote the outgoing primitive tangent vectors by $v_i$, then 
          \begin{align*}
           \sum_i w_i v_i=0.
          \end{align*}
  \end{enumerate}
\end{definition}
 The
balancing condition will make the following definition well-defined.
\begin{definition} \label{2000}
    Let $\phi: G \rightarrow B$ be a parametrized tropical curves (with stop $x$) with only trivalent vertices. The multiplicity at a vertex $v\in G^{[0]}$ (except at the stop $x$) is given by 
       \begin{equation} \label{1010}
            Mult_v(\phi)=w_1 w_2|m_1 \wedge m_2|,
       \end{equation}
    where $E_1, E_2$ are two of the edges containing $V$ and $w_i=w_{\Gamma}(E_i)$ and $m_i \in T_{\phi(v)}B$ are the primitive integral vectors in the direction of $\phi(E_i)$. The last term $m_1\wedge m_2$ in (\ref{1010}) falls in $\wedge^2 T_{\phi(v)}B\cong \mathbb{Z}$ because of the integral structure.
\end{definition}

\begin{definition}
 A tropical disc is a tropical curve with stop $(w,G,\phi)$ and $G$ is a tree.
\end{definition}

The definition of tropical curves (discs) does not involve the K\"ahler structure. However, we here define the central charge of the tropical discs on K3 surface for the purpose of defining generalized Donaldson-Thomas invariants in the later section. One can view the central charge as the volume of the tropical disc together with a phase.
\begin{definition} \label{303}
   (Central charge of tropical discs)Let $(X,\omega,\Omega)$ be an elliptic K3 surface with singular fibres all of $I_1$-type. Given a tropical discs $\phi:G \rightarrow B$ with stop at $u$ on an
   tropical affine manifold $B$ induced from the affine coordinate of special Lagrangian fibration on $X_{\vartheta}$, we will associate it with a central charge as
   follows by induction on the number of singularities of affine structure $\phi$
   hits: If the $\phi$ only hits only one singularity and has its stop at $u$, then let
   $[\phi]\in H_2(X,L_u)$ \footnote{The class $[\phi]$ can also be understood as follows: one associate a cyclinder by normal construction or each edge of tropical disc. At each trivalent vertex, we glue in a pair of pant as local model. However, it is not known the suitable local model for vertex of higher valency}be the relative class of Lefschetz thimble
   such that $\int_{[\phi]}\omega_{\vartheta}>0$ and its central charge $Z_{\phi}=\int_{[\phi]}\Omega$. Assume $p$
   is an internal vertex of $\phi$ and each let $\phi_1,\cdots
   \phi_s$ are the components of $\mbox{Im}(\phi)\backslash p$
   containing an ingoing edge of $p$. By induction we already define $[\phi_i]\in H_2(X,L_p)$ and the corresponding central
   charges $Z_{\phi_i}=\int_{[\phi_i]}\Omega$ for each
   $i=1,\cdots,s$. For any $p'$ on the outgoing
   edge of $p$, there is a natural tropical disc $\phi'$ with stop at $p'$
   induced from $\phi$, then we define $[\phi']\in H_2(X,L_{p'})$ the parallel
   transport of $\sum_{i=1}^s[\phi_i]$ along the outgoing edge from
   $p$ to $p'$. The central charge of $\phi'$ is given by
   $Z_{\phi'}=\int_{[\phi']}\Omega$.
\end{definition}

\subsection{Scattering Diagrams on Elliptic K3 Surfaces}\label{1032}
In the section, we will review the notion of scattering diagrams invented by Kontsevich-Soibelman \cite{KS1}. Scattering diagrams can be viewed as the mirror of tropical discs. Indeed, scattering diagram of $X_{\vartheta}$ consists of affine rays on the base $B_0$ with the symplectic affine structure of $X_{\vartheta}$ while the tropical discs of $X_{\vartheta+\pi/2}$ are affine $1$-skeletons on $B_0$ with the complex affine structure of $X_{\vartheta}+\pi/2$. Since the above two affine structures on $B_0$ are canonically identified, the two K3 surface $X_{\vartheta}$ and $X_{\vartheta+\pi/2}$ are mirror pairs in the sense that the scattering diagram of one is identified with the tropical discs of another.

We first introduce the Novikov ring with formal parameter $T$. Let
\begin{equation*}
    \Lambda_0=\bigg\{ \sum^{\infty}_{i=0}a_i T^{\lambda_i} \bigg|  a_i \in \mathbb{C}, \lambda \in \mathbb{R}_{\geq 0}, \lim_{i \to \infty} \lambda_i =\infty \bigg\}
\end{equation*}
and
\begin{equation*}
    \Lambda_+=\bigg\{ \sum^{\infty}_{i=0}a_i T^{\lambda_i} \in \Lambda_0 \bigg| \lambda_i >0 \bigg\}.
\end{equation*}
There is a natural filtration on $\Lambda_0$ given by
\begin{equation}\label{1066}
    F^{\lambda}\Lambda_0= T^{\lambda}\Lambda_0
\end{equation}
for each $\lambda \in \mathbb{R}_{\geq 0}$. The
localization of $\Lambda_0$ at its maximal ideal, which is a
generalization of puiseux series, is algebraically closed and
complete in $T$-adic topology.

Recall that $\Gamma_g=\cup_{u\in B_0}H_1(L_u;\mathbb{Z})$. Let
\begin{equation*}
   \mathbb{C}[\Gamma_g] \hat{\otimes}_{\mathbb{C}}\Lambda_0=\lim_{\leftarrow}\mathbb{C}[\Gamma_g]\otimes_{\mathbb{C}}\Lambda_0/ F^{\lambda}\Lambda_0
\end{equation*} and $\mathfrak{g}=\Lambda_+ ( \mathbb{C}[\Gamma_g]
\hat{\otimes}_{\mathbb{C}}\Lambda_0)\otimes \Gamma_g^*$. Given any $\xi \in
\mathfrak{g}$, we have an element
\begin{equation*}
   \exp(\xi)\in
   Aut_{\Lambda_0}(\mathbb{C}[\Gamma_g]\hat{\otimes}_{\mathbb{C}}\Lambda_0).
\end{equation*} Write $e_{\partial\gamma}=z^{\partial\gamma} \otimes w_{\partial\gamma}\partial\gamma^{\bot}$, where $\partial\gamma=\omega_{\partial\gamma}\partial\gamma_{prim}$, $\omega_{\partial\gamma}\in \mathbb{Z}$ and $\partial\gamma_{prim}$ is primitive.
Then from Baker-Cambell-Hausdorf formula, there is a Lie algebra structure on
$\mathfrak{g}$ given by
\begin{equation} \label{1040}
       [e_{\partial\gamma_1},e_{\partial\gamma_2}]=(-1)^{\langle \partial\gamma,\partial\gamma' \rangle}\langle \gamma,\gamma' \rangle e_{\partial\gamma+\partial\gamma'}.
    \end{equation}
 This is the Lie algebra structure on the twisted complexified symplectic torus $\mathfrak{T}=\mbox{Spec}(\mathbb{C}[\Gamma_g] \hat{\otimes}_{\mathbb{C}}\Lambda_0)$ which admits a standard holomorphic symplectic $2$-form
 \begin{align*}
   \varpi=\langle \gamma_1,\gamma_2\rangle \frac{dx^{\partial \gamma_1}}{x^{\partial \gamma_1}}\wedge\frac{dx^{\partial \gamma_2}}{x^{\partial \gamma_2}},
 \end{align*}
 where $\{\partial\gamma_1,\partial\gamma_2\}$ is a basis of $\Gamma_g$. In particular, the subspace
\begin{equation*}
   \mathfrak{h}=\bigoplus_{\partial\gamma \in \Gamma_g \backslash \{0\}} z^{\partial\gamma}(\Lambda_+ \otimes \partial\gamma^{\bot})
\end{equation*}
is closed under the above bracket and thus via exponential map
produces a  Lie group $\mathbb{V}_{trop}$. Straight forward computation shows that elements in $\mathbb{V}^{trop}$ induce symplectomorphism on the twisted symplectic torus.

\begin{definition}\label{305}
   Let $B$ be an affine manifold with singularities $\Delta$ and with integral structures on its tangent bundle. Assume the monodromy at each singularity is conjugate to $\bigl(
   \begin{smallmatrix}
     1 & 1\\
     0 & 1
   \end{smallmatrix} \bigr)$. A scattering diagram $\mathfrak{D}=\{(\mathfrak{d},f_{\mathfrak{d}})\}$ on an affine manifold $B$ with an integral structure on its tangent bundle $TB$ is a collection of $2$-tuples $(\mathfrak{d},f_{\mathfrak{d}})$ such that
    \begin{enumerate}
       \item
        $\mathfrak{d}$ is either of line $\mathfrak{d}=\mathbb{R}\partial\gamma_{\mathfrak{d}}$ with rational slope with respect to the affine
               structure or a ray
        $\mathfrak{d}=o_{\mathfrak{d}}+\mathbb{R}_{\geq 0}\partial\gamma_{\mathfrak{d}}$ emanating from $o_{\mathfrak{d}}$ with rational slope with respect to the affine
       structure, where $\partial \gamma_{\mathfrak{d}}$ is a primitive integral vector. If $o_{\mathfrak{d}}\in \Delta$, then $\partial \gamma_{\mathfrak{d}}$ is in the monodromy invariant direction.
       \item The slab function $f_{\mathfrak{d}}(u)\in
       \mathbb{C}[z^{\partial \gamma_{\mathfrak{d},u}}]\hat{\otimes}_{\mathbb{C}}\Lambda_+$
       for each $u\in \mathfrak{d}$, where $\partial \gamma_{\mathfrak{d},u}$ is the
       parallel transport of $\partial \gamma_{\mathfrak{d}}$ from $o_{\mathfrak{d}}$ to $u$ along
       $\mathfrak{d}_i$. Moreover, $f_{\mathfrak{d}}(u)$ satisfies 
         \begin{align}
           f_{\mathfrak{d}}(u)\equiv 1 \pmod{\Lambda_+} 
         \end{align}
       and each of its monomial of is of the form
       $cz^{l\partial \gamma_{\mathfrak{d},u}}T^{A(u)}$, where $l\in \mathbb{N}$, $c$ is a constant. If $\mathfrak{d}$ is a ray, then $A(u)$ is a strictly increasing function
       along $\mathfrak{d}_i$ and $A(o_{\mathfrak{d}})>0$. If $\mathfrak{d}$ is a line, then $A(u)$ is a strictly monotone positive function.
       \item For every point $u\in B$ and a given $\lambda>0$, there are only finitely
       many rays $(\mathfrak{d},f_{\mathfrak{d}})\in \mathfrak{D}$ such that
       $f_{\mathfrak{d}}(u)\not\equiv 0 \pmod{T^{\lambda}}$.
       \item The singularity of the scattering diagram
       $\mbox{Sing}(\mathfrak{D})^{>\lambda}$ is given by the set
          \begin{align*}
             \{u\in B| &\exists (\mathfrak{d}_i,f_i)\in
             \mathfrak{D},i=1,2 \mbox{ such that } u\in \mathfrak{d}_1\cap
             \mathfrak{d}_2 \\
              &\hspace{50mm}\mbox{ and } f_1(u)f_2(u)\not \equiv 0
             (\mbox{mod }T^{\lambda})\}
          \end{align*}
    \end{enumerate}
\end{definition}
Let $\mathfrak{D}$ is a scattering diagram on an integral affine
manifold $B$, $u\in B$ and $\lambda>0$. Consider an embedding
   \begin{align*}
       \phi:S^1\rightarrow B\backslash\mbox{Sing}(\mathfrak{D})^{>\lambda}
   \end{align*}
in a small enough neighborhood of $u$ such that it intersects every ray $\mathfrak{d}$ starting from $u$ (or passing through $u$)
 transversally and exactly once (or twice respectively) if $(u\in\mathfrak{d},f)\in \mathfrak{D}$
and $(f(u)\not \equiv 0\mbox{ mod }T^{\lambda})$. Assume the
intersection order is $\mathfrak{d}_1,\cdots \mathfrak{d}_s$, then
we form an ordered product as follows :
    \begin{align} \label{1011}
      \theta_{\phi,\mathfrak{D}}^{u,\lambda}=\theta_{\mathfrak{d}_1}(u)\circ
      \dots \circ \theta_{\mathfrak{d}_s}(u),
    \end{align} where each term on right hand side of (\ref{1011}), called Kontsevich-Soibelman transformation, is of the form
    \begin{align*}
      \theta_{\mathfrak{d}_i}(u)=\mbox{exp}(\log{(f_i(u))}\otimes n_i)\in \mathbb{V}_{trop},
    \end{align*}with $n_i\in (\Gamma_g)^*$ primitive, annihilates the tangent
    space to $\mathfrak{d}_i$ and such that
      \begin{align*}
        \langle n_i,\phi'(p_i)\rangle >0, \mbox{ for } p_i\in \mbox{Im}\phi\cap
        \mathfrak{d}_i.
      \end{align*}
Now given an elliptic K3 surface $(X,\omega,\Omega)$ with Ricci-flat
K\"ahler form $\omega$. After hyperK\"ahler rotation, it induces a special Lagrangian torus fibration on
$X_{\vartheta}$, for each $\vartheta\in S^1$ (see Remark \ref{300}).
Fix a phase parameter $\vartheta \in S^1$ and we have an affine
structure with singularities $B_{\vartheta}$. 

 We will consider the completion (in the sense of Theorem \ref{13}) of following scattering diagram: each singular point $v$ emanates two BPS rays called initial rays along both monodromy invariant direction
$\mathfrak{d}_{\pm}$ with the slab function
\begin{align} \label{1045}
f_{\mathfrak{d}_{\pm}}=1+z^{\partial\gamma_{\pm}}T^{|Z_{\gamma_{\pm}}(u)|},
\end{align}
where $\gamma_{\pm}\in\Gamma$ denotes the relative classes of Lefschetz thimbles around
the corresponding singularity such that 
  \begin{align}
    \int_{\mathfrak{d}_{\pm}}(\iota_{\partial \gamma_{\pm}}\omega) du >0.
  \end{align}

The following theorem is a modified version of statements in \cite{KS1}\cite{GS1}\cite{L3}.
\begin{thm} \label{13}
   Let $\mathfrak{D}_{\vartheta}$ be a scattering diagram given by the initial data above. Then for generic $\vartheta\in S^1$, there is a scattering diagram $S(\mathfrak{D}_{\vartheta})$ which contains $\mathfrak{D}_{\vartheta}$ such that for any $\lambda>0$,
   there are only finitely many rays with nontrivial slab
   functions modulo $T^{\lambda}$. Moreover, given $u\in B_0$, $\lambda>0$ and a closed loop $\phi$ near $u$,
   one has
   \begin{align} \label{1050}
   \theta^{u,\lambda}_{\phi, S(\mathfrak{D}_{\vartheta})}\equiv 0\pmod{T^{\lambda}}
   \end{align}
\end{thm}
\begin{proof}
       Notice that when
      $\lambda$ is small we can take $S(\mathfrak{D}_{\vartheta})=\mathfrak{D}_{\vartheta}$. Indeed,
      since there are $24$ isolated singularities of the affine structures, we can choose disjoint open neighborhood for theses singularities such that the slab function $f_{\mathfrak{d}_{\pm}}$ for any singularity satisfies 
        \begin{align*}
         f_{\mathfrak{d}_{\pm}}(p)\equiv 1 \pmod {T^{\lambda_0}},
        \end{align*}for some $\lambda_0$ if $p$ is in the complement of above chosen open neighborhoods. In particular, the equation (\ref{1050}) holds for $\lambda < \lambda_0$. Therefore we will set $(\mathfrak{D}_{\vartheta})_0=\mathfrak{D}_{\vartheta}$.
       Assume that we already construct a scattering diagram $(\mathfrak{D}_{\vartheta})_k\supseteq \mathfrak{D}_{\vartheta}$ with only finitely many rays such that (\ref{1050}) holds for all $\lambda\leq k\lambda_0$. We want to construct $(\mathfrak{D}_{\vartheta})_{k+1}$ with analogue properties. 
       We consider $\theta_{\phi}$ at each singularity
      $u_0$ of $(\mathfrak{D}_{\vartheta})_k$ and a small loop $\phi$ around $u_0$. Note that the exponent of $T$
in $\theta_{\phi,(\mathfrak{D}_{\vartheta})_k}^{u,(k+1)\lambda_0}(u_0)$ is discrete and there are only finitely many terms with power of $T$ less than $(k+1)\lambda_0$. Then we have
\begin{equation*}
   \theta_{{\phi}, (\mathfrak{D}_{\vartheta})_k}^{u,(k+1)\lambda_0}(u_0) \equiv \exp{(\sum^s_{i=1}c_iz^{\partial \gamma_i}T^{\lambda_i}\otimes n_i)}\pmod{T^{(k+1)\lambda_0}},
\end{equation*} where $c_i\in \mathbb{R}$ and $\partial \gamma_i\in \Gamma_g$ and $k\lambda_0<\lambda_i<(k+1)\lambda_0$. (Actually, there might be more than one $u_0$. Then we will just consider them all at the same time.) We set
\begin{equation} \label{1020}
   (\mathfrak{D}_{\vartheta}')_{k+1}=(\mathfrak{D}_{\vartheta})_k \cup \{u_0+\mathbb{R}_{\geq 0}\partial\gamma_i, 1 \pm c_i z^{\partial\gamma_i}T^{f_i(u)} | i=1, \cdots, s \} 
\end{equation}
The sign in front of $c_i$ is chosen such that each contribute $\exp{(-c_i
z^{\partial\gamma_i}T^{f_i(u)}\otimes n_i)}$ to
$\theta_{\phi,\mathfrak{D}_k} \pmod{T^{\lambda_{k+1}+\epsilon}}$,
 where the slab function is given by
  \begin{align} \label{1061}
   f_i(u)=\lambda_{i}+|\int_{u_0}^u(\int_{\partial \gamma_i}\iota_{\tilde{\partial\gamma_i}}e^{-i\vartheta}\Omega) du|.
  \end{align}
However, the scattering diagram $(\mathfrak{D}_{\vartheta}')_{k+1}$ may not serves the purpose. This is because the added rays in (\ref{1020}) may intersects the initial rays near the singularities and create new singularities for the scattering diagram. Say $p$ is the intersection of $\mathfrak{d}\in (\mathfrak{D}_{\vartheta}')_{k+1}\backslash (\mathfrak{D}_{\vartheta})_k$ and an initial ray $\mathfrak{d}'$. Similarly, we will have
   \begin{align}\label{2020}
      \theta_{{\phi}, (\mathfrak{D}_{\vartheta}')_k}^{u,(k+1)\lambda_0}(u_0) \equiv \exp{(\sum^{t}_{l=1}c_iz^{\partial \gamma_l}T^{\lambda_l}\otimes n_l)}\pmod{T^{(k+1)\lambda_0}}.
   \end{align} Moreover, $\lambda_l=mf_{\mathfrak{d}'}(p)+f_{\mathfrak{d}}(p)$ for some $m\in \mathbb{N}$ and thus the sum in (\ref{2020}) can only be a finite summation. Let 
    \begin{equation} 
       (\mathfrak{D}_{\vartheta})_{k+1}=(\mathfrak{D}_{\vartheta}')_{k+1} \cup \{p+\mathbb{R}_{\geq 0}\partial\gamma_l, 1 \pm c_l z^{\partial\gamma_l}T^{f_l(u)} | l=1, \cdots, t \},
    \end{equation} then we have 
      \begin{align*}
      \theta_{{\phi}, (\mathfrak{D}_{\vartheta})_{k+1}}^{u,(k+1)\lambda_0}(u_0) \equiv 0\pmod{T^{(k+1)\lambda_0}}
      \end{align*}due to the following two facts:
 \begin{enumerate}
    \item Any affine ray start from the initial ray will not be able to intersect 
         the initial ray again without going out a neighborhood of the singularity.
    \item Shrinking the neiborhood around each singularity if necessary, then any affine line from a neighborhood of one singularity to another has its affine length larger than $\lambda_0$.
 \end{enumerate}      
        By induction, it suffices to take
        $S(\mathfrak{D}_{\vartheta})=\cup_k (\mathfrak{D}_{\vartheta})_k$ and this will complete the proof if there is no BPS rays hit the singularities of the affine structure.
  \begin{definition}
   We will call the rays in $S(\mathfrak{D}_{\vartheta})$ BPS rays.
\end{definition}
\begin{prop}
  Each BPS ray $(\mathfrak{d},f_{\mathfrak{d}})$ in the scattering diagram $S(\mathfrak{D}_{\vartheta})$ corresponds to an image of tropical disc $\phi_{\mathfrak{d}}$. In particular, we will write $\gamma_{\mathfrak{d}}=[\phi_{\mathfrak{d}}]$.
\end{prop}
\begin{proof}
  Notice that the monomials of the form $cz^{\partial\gamma}T^{|Z_{\gamma}|}$, where $c$ is a constant and  $\gamma\in \Gamma$, are closed under multiplication. We will prove the proposition and all the monomials of $f_{\mathfrak{d}}$ are in the form $cz^{\partial \gamma}T^{|Z_{\gamma}|}$ by induction on $\lambda_k$, the exponent of $T$ in $f_{\mathfrak{d}}(o_{\mathfrak{d}})$. First of all, each term of the slab functions associate to initial BPS rays are in this form and the initial BPS rays correspond to Lefschetz thimbles. Assume the two statement hold or $\lambda<\lambda_k$ and there is a BPS ray $\mathfrak{d}$ with $f_{\mathfrak{d}}(o_{\mathfrak{d}})=\lambda_k$. Then from the construction of the scattering diagram, we have 
    \begin{align*}
      \lambda_{k+1}=\sum_{\mathfrak{d}':o_{\mathfrak{d}}\notin \mathfrak{d}'}a_{\mathfrak{d}'}f_{\mathfrak{d}'}(o_{\mathfrak{d}}), 
    \end{align*} for some $a_{\mathfrak{d}'}\in \mathbb{Z}_{\geq 0}$. 
 Together with (\ref{1061}) and Lemma \ref{34}, we have $f_{\mathfrak{d}}=Z_{\sum a_{\mathfrak{d}'}\gamma_{\mathfrak{d}'}}$ and $\gamma_{\mathfrak{d}}=\sum a_{\mathfrak{d}'}\gamma_{\mathfrak{d}'}$.
\end{proof}

  We will prove that all the rays in $S(\mathfrak{D})$ will not contain any singularities of the affine structure for a generic $\vartheta$. For this purpose, we need to study the perturbation of BPS rays. This will also help us to define the generalized Donaldson-Thomas invariants in the later section.

Given a charge $\gamma\in H_2(X,L)$, we can consider the gradient
flow lines associate to the function
  \begin{align*}
    F_{\gamma}:=|Z_{\gamma}|^2. 
  \end{align*}Since we are interested in the situation when $\gamma$ comes from a tropical disc (see Definition \ref{303}), we may assume that $Z_{\gamma}$ is nonzero locally. We will consider the gradient flow lines of $F_{\gamma}$, namely, $y(t)$ satisfies the equation
     \begin{align} \label{1013}
       \frac{dy}{dt}=\nabla F_{\gamma}= |Z_{\gamma}|^2\nabla \log{|Z_{\gamma}|^2}
     \end{align}
  
\begin{prop}
   The (multi-)function $F_{\gamma}$ has no critical points. In particular, it is an Morse function.
\end{prop}  
\begin{proof}
   Indeed, assume $p$ is a critical point of $F_{\gamma}$. From the equation (\ref{1013}) and Cauchy-Riemann equation of $Z_{\gamma}$ at $p$, we have $dZ_{\gamma}(p)=0$. Using the expression in Lemma \ref{34}, it contradicts to the fact that $\mbox{Re}(e^{i\vartheta}\Omega)$ is a non-degenerate $2$-form at every point. 
\end{proof}

\begin{prop}
The phase of the central charge $\mbox{Arg}Z_{\gamma}$ is
constant along the gradient flow lines $y(t)$ and each gradient flow
line is an affine line on $B_{\vartheta}$, where
$\vartheta=\mbox{Arg}Z_{\gamma}$. In particular, the gradient flow lines are determined completely by the central charges and independence of the choice of K\"ahler metric on $B$.
\end{prop}  
\begin{proof}
Because $\log{Z_{\gamma}}=\log{|Z_{\gamma}|}+i\mbox{Arg}Z_{\gamma}$
is holomorphic, we have
   \begin{align*}
      \frac{d}{dt} \mbox{Arg}Z_{\gamma}(y(t))=(\nabla F_{\gamma})
      \mbox{Arg}Z_{\gamma}=J\nabla F_{\gamma}(\log{|Z_{\gamma}|})=0.
   \end{align*}
The second equation is because the central charge $Z_{\gamma}$ is a holomorphic function and follows from Cauchy-Riemann equation. The third equation holds because for any function $f$, its gradient $\nabla f$ is
perpendicular to the level set of $f$. So $J\nabla f(f)=0$. 
\end{proof}

Now there are two cases when a BPS ray hit the singularity of the affine structure depending on the BPS ray hit the singularity from the monodromy invariant direction or not. The corresponding central charges near the singularity would either the form $Ay+B$ or $\frac{A}{2\pi i}(y\log{y}-y)+B$, where $A\in \mathbb{Z}\backslash \{0\}$, $B\in \mathbb{C}^*$ are constants and $y$ is the local coordinate around the singularity of the affine structure. Notice that in both cases the limits of central charge at the origin exist and are nonzero. If the gradient flow of generic phase can't avoid the origin, then the central charge would have a constant phase around the origin. This contradicts to the form of central charge above. Since the there are only finitely many BPS rays modulo $T^{\lambda}$, $\lambda>0$. Thus, all the BPS rays in $S(\mathfrak{D}_{\vartheta})$ can avoid the singularities of the affine structure for a choice of generic phase $\vartheta$ by the Baire category theorem.
\end{proof}

\begin{rmk}
Notice that we don't have the notion of degree as in
\cite{GS1}\cite{KS1} therefore we need to use energy filtration
instead and a static construction of the scattering diagram. At each
singularity of the scattering diagram, the degree filtration and energy filtration are
equivalent. 
\end{rmk}

\begin{rmk}
  Notice that it is quite different from Gaiotto-Moore-Neitzke theory in the sense of the following: In \cite{GMN}, the geodesics (compared to the affine lines here) live on a Riemann surface while the BPS rays live on the base of Hitchin moduli of that Riemann surface. Therefore, the geodesics and BPS rays live on different spaces and can't be compared. 
\end{rmk}

\begin{rmk}
In physics literature, the mass $M$ of any charge $\gamma$ obeys
\begin{equation}
    M \geq|Z_{\gamma}|,
\end{equation}
where the mass is $\int_{\gamma}|\Omega|$ along a path. We call the
charge $\gamma$ is BPS if and only if the equality holds. Thus, a
charge is BPS if only if its phase of central charge is the same
angle along the path. Therefore, it is reasonable to expect tropical
discs to correspond to BPS charges in physics.
\end{rmk}

There are also some byproduct of Theorem \ref{13}. The tropical discs on $B_{\vartheta}$ corresponds to rays in the scattering diagram $S(\mathfrak{D}_{\vartheta})$ on $B_{\vartheta}$. Although they both live on the same affine manifold $B_{\vartheta}$, one should view $B_{\vartheta}$ as different ones. The tropical discs lives in $B_{\vartheta}$ coming from the complex affine structure of $X_{\vartheta}\rightarrow B$ while the scatter diagram $S(\mathfrak{D}_{\vartheta})$ lives in $B_{\vartheta}$ coming from the symplectic affine structure of $X_{\vartheta+\pi/2}\rightarrow B$. The coincidence of these two affine structures is a statement of mirror symmetry: complex instanton correction of a Calabi-Yau manifold is related to the holomorphic discs on its mirror. From the study of deformation of BPS rays, we also have the following corollary:
\begin{cor}
There is no tropical rational curves (see Definition \ref{322} in
the next section) for a generic $\vartheta$.
\end{cor}
This is reasonable because generic K3 has Picard number $0$ by
Torelli theorem.

\subsection{Generalized Donalson-Thomas Invariants on Elliptic K3 Surfaces} \label{2008}
In this section, we will construct the generalized Donaldson-Thomas invariants on elliptic K3 surfaces with only $I_1$-type singular fibres. Part of the construction might overlap with the one in \cite{KS5}.
From the $S^1_{\vartheta}$-family of scattering diagram, there is a very important concept called wall of marginal stability:

\begin{definition}
For a given $u_0\in B_0$ and $\gamma\in \Gamma_{u_0}$, there exists an open subset $\mathcal{U}\subseteq B_0$ containing $u_0$ such that there is an associated  wall of marginal stability $W^{trop}_{\gamma}\subseteq \mathcal{U}$ given by
  \begin{align*}
    W_{\gamma}^{trop}=\bigcup_{\gamma=\gamma_1+\gamma_2} W^{trop}_{\gamma_1,\gamma_2},
  \end{align*}
    where
\begin{equation*}
    W^{trop}_{\gamma_1,\gamma_2}=\{ u: \exists \gamma_{1,u}, \gamma_{2,u}\in \Gamma_u \text{ coming from BPS rays and } \frac{Z_{\gamma_1}(u)}{Z_{\gamma_2}(u)} \in \mathbb{R}_+  \}
\end{equation*}
\end{definition}

Now we are going to define the generalized Donaldson-Thomas invariants using the following procedure: For any pair $(\gamma,u)$, we first perturb $u$ to be generic nearby one and assume that $u\notin W^{trop}_{\gamma}$. The phase of the central charge $\vartheta=\mbox{Arg}Z_{\gamma}$ indicates which scattering diagram $S(\mathfrak{D}_{\vartheta})$ is going to determine the invariant $\Omega^{trop}(\gamma;u)$. 

Now we want to decompose the symplectormorphism $\theta_{\mathfrak{d}}(u)$ at each point $u$ of a BPS ray $\mathfrak{d}\in S(\mathfrak{D}_{\vartheta})$ into some elementary ones of the form
   \begin{align} \label{1060}
     \mathcal{K}_{\gamma}(u):  \mathbb{C}[\Gamma_{g,u}]  &\hat{\otimes}_{\mathbb{C}}\Lambda_0 \mapsto \mathbb{C}[\Gamma_{g,u}] \hat{\otimes}_{\mathbb{C}}\Lambda_0    \nonumber \\
     & z^{\partial\gamma_u'}\longmapsto z^{\partial \gamma_u'}(1-\sigma(\gamma_u)z^{\partial\gamma_u}T^{Z_{\gamma_u}}),
   \end{align} for $\gamma_u\in \Gamma_u$. Here $\sigma(\gamma_u)=\pm 1$ is called the quadratic refinement satisfying the relation
      \begin{align} \label{2003}
       \sigma(\gamma_{1,u})\sigma(\gamma_{2,u})=(-1)^{\langle \gamma_{1,u},\gamma_{2,u}\rangle}\sigma(\gamma_{1,u}+\gamma_{2,u}),
      \end{align} where the weird $(-1)^*$ is due to the $(-1)^*$ factor in (\ref{1040}).

We will drop the subindex $u$ for shorter notation if the context is clear. Given a path $\phi$ from $u$ to $u'$, one has the parallel transport 
   \begin{align*}
      T^{u',u}_{\phi}:\Gamma_u\rightarrow \Gamma_{u'},
   \end{align*}   
 which extends naturally to 
    \begin{align*}
       T^{u',u}_{\phi}:\mbox{Aut}(\mathbb{C}[\Gamma_{g,u}]  &\hat{\otimes}_{\mathbb{C}}\Lambda_0) \rightarrow \mbox{Aut}(\mathbb{C}[\Gamma_{g,u'}]  \hat{\otimes}_{\mathbb{C}}\Lambda_0)\\
       &\mathcal{K}_{\gamma}(u)\mapsto \mathcal{K}_{T^{u',u}_{\phi}(\gamma)}(u').
    \end{align*}   
   
For this purpose, we will need the following lemmas, which are straight forward algebra computation and we will omit the proof.
\begin{lem} \label{27}
   Let $\gamma_1,\gamma_2\in \Gamma$ be two relative classes. If the intersection pairing $\langle \gamma_1,\gamma_2\rangle=0$ , then the associate Kontsevich-Soibelman transformation (\ref{1060}) commutes Namely, we have $\mathcal{K}_{\gamma_1}\mathcal{K}_{\gamma_2}=\mathcal{K}_{\gamma_2}\mathcal{K}_{\gamma_1}$. In particular, the Kontsevich-Soibelman transformation associate to a pure flavor charge is an identity.
   \end{lem}
\begin{lem}\label{1037}
Fix $\epsilon=0$ or $1$\footnote{
 The sign $\epsilon$ will be related to quadratic refinement in
 \cite{GMN}.}
.
   Let $f=1+a_1x+a_2x^2+\cdots \in \mathbb{Q}[[x]]$, then there is a unique factorization
   \begin{equation*}
       f=\prod_k(1-(-1)^{k\epsilon}x^k)^{kd_k},
   \end{equation*}
  for some $d_k \in \mathbb{Q}$. Moreorver, we have the estimate for size of $d_k$.
\end{lem}

The Lemma \ref{27} guarantees that the Kontsevich-Soibelman transformations in both sides of (\ref{1030}) do not depend on the order. The Lemma \ref{1037} guarantees that one can be uniquely factorized left hand side of (\ref{1030}) into product of $\mathcal{K}_{\gamma}$ with rational exponents.
 \begin{align}\label{1030}
    \prod_{\mathfrak{d}:\mbox{Arg}Z_{\gamma_{\mathfrak{d}}}=\vartheta, u\in \mathfrak{d}\backslash \{o\}}\theta_{\mathfrak{d}}(u)=\prod_{\gamma:\partial \gamma\parallel \mathfrak{d}} \mathcal{K}_{\gamma}^{\Omega^{trop}(\gamma;u)}
 \end{align}
In particular, the generalized Donaldson-Thomas invariant $\Omega^{trop}(\gamma;u)\in \mathbb{Q}$ is the exponent of $\mathcal{K}_{\gamma}$ in the decomposition (\ref{1030}). For a neighborhood $\mathcal{U}$ of $u$, we already define the invariant $\Omega^{trop}(\gamma;u)$ for generic $u\in \mathcal{U}\backslash W_{\gamma}^{trop}$. We will extend $\Omega^{trop}(\gamma;u)$ continuously to each components of $\mathcal{U}\backslash W_{\gamma}^{trop}$. This will finish the definition of the generalized Donaldson-Thomas invariants $\Omega^{trop}(\gamma;u)$. 
\begin{rmk} The construction of $\Omega^{trop}(\gamma)$ can be viewed as an
inverse of the procedure in \cite{L3}.
\end{rmk}
From the definition here, we only know that these generalized Donaldson-Thomas invariants are rational numbers. However, we expect the following integrality conjecture
\begin{conj}The generalized Donaldson-Thomas invariants 
  $\Omega^{trop}(\gamma)\in \mathbb{Z}$, for every $\gamma \in
  \Gamma$.
\end{conj}
One can apply the following M\"obius type transformation to the logarithm of the slab functions. 
\begin{lem}\label{1038}
    Let $c(n)=\sum_{k=1}\frac{d_{n/k}}{k^2}$, and $d_s=0$ if $s$ is not an integer. Then we have
  \begin{equation} \label{1017}
     \sum^{\infty}_{k=1}kd_k\log{(1-(-1)^{k\epsilon}x^k)}=\sum^{\infty}_{n=1}nc(n)u^n,
  \end{equation}where $u=(-1)^{\epsilon}x$.
\end{lem}
Lemma \ref{1038} converts $\Omega^{trop}$ to another rational value invariants $\tilde{\Omega}^{trop}$, which is more directly related to the counting of tropical discs (See Section \ref{1039} in more details). 

The reason why we call these numbers $\Omega^{trop}(\gamma;u)$ generalized Donaldson-Thomas invariants is because they are locally constants from the construction and jump according to a version of Kontsevich-Soibelman wall-crossing formula, which is described below, when $u$ moves across the wall of marginal stability $W_{\gamma}$.

For any strictly convex cone $V\subseteq \mathbb{C}$ with its apex at the
origin, we define
\begin{equation*}
  A_V(u)=\prod_{\gamma:Z_{\gamma}(u)\in V}\mathcal{K}_{\gamma}(u)^{\Omega^{trop}(\gamma;u)},
\end{equation*}
where the product is taken in order of increasing
$\mbox{Arg}Z_{\gamma}(u)$. 
\begin{thm} \label{45}
   The generalized Donaldson-Thomas invariants $\{\Omega^{trop}(\gamma;u)\}$ satisfy the following version of Kontsevich-Soibelman wall-crossing formula: for any path in $B_0$ connecting $u_1$ and $u_2$ which has no point $u$ with $Z_{\gamma}(u)\in \partial V$ and $\Omega(\gamma;u) \neq 0$,
   then $A_V(u)$ and $A_V(u')$ are related by parallel transport in $B_0$ along the path and modulo $T^{\lambda}$, for any $\lambda>0$. 
\end{thm}
\begin{proof}
 Since $\Omega^{trop}$ are locally constant away from the wall of marginal stability, it suffices to prove the statement when path passes through only one wall $W^{trop}_{\gamma_1,\gamma_2}$. We may assume the pass passes through the wall at a generic point $u_0$ with
 $d\frac{Z_{\gamma_1}}{Z_{\gamma_2}}(u_0) \neq 0$ by small perturbation of the path, which preserves the parallel transport. Since there is no $\gamma$ such that $Z_{\gamma}(u_0)\mathbb{R}_{\geq 0} \in \partial V$ with  $\Omega^{trop}(\gamma;u_0)\neq 0$, there is no BPS ray passing through $u_0$ tangent to $W^{trop}_{\gamma_1,\gamma_2}$ at $u_0$. From (\ref{1050}) in the Theorem \ref{13}, we have 
    \begin{align*}
       \theta_{\phi,S(\mathfrak{D}_{\vartheta_0})}^{u_0,\lambda}:=\theta_{\mathfrak{d}_1}(u_0)\circ \cdots \circ\theta_{\mathfrak{d}_s}(u_0)\equiv 0. \pmod{T^{\lambda}}
    \end{align*}We may rearranging the terms 
    \begin{align*}
      \theta_{\mathfrak{d}_1}(u_0)\circ \cdots \circ\theta_{\mathfrak{d}_{s'}}(u_0)=\theta_{\mathfrak{d}_s}^{-1}(u_0)\circ \cdots \circ\theta_{\mathfrak{d}_{s'+1}}^{-1}(u_0)
    \end{align*}such that $u_0+\partial\gamma_{\mathfrak{d}_1},\cdots, u_0+\partial\gamma_{\mathfrak{d}_{s'}}$ are in one side of the wall (say with $u$) and appear in counterclockwise  order, while $u_0-\partial \gamma_{\mathfrak{d}_{s}}, \cdots, u_0-\partial \gamma_{\mathfrak{d}_{s'+1}}$, are on the other side of the wall and appear in counterclockwise order. The Kontsevich-Soibelman wall-crossing formula then follows
       \begin{align*}
         A_V(u_1)=&\prod_{\gamma:Z_{\gamma}(u_1)\in V}\mathcal{K}_{\gamma}(u_1)^{\Omega^{trop}(\gamma;u_1)} \\ 
         \rightarrow & \prod_{\gamma:Z_{\gamma}(u_0)\in V}\big(T^{u_0,u_1}\mathcal{K}_{\gamma}\big)(u_0)^{\Omega^{trop}(\gamma;u_1)} \\ =& \theta_{\mathfrak{d}_1}(u_0)\circ \cdots \circ\theta_{\mathfrak{d}_{s'}}(u_0) \\=&\theta_{\mathfrak{d}_s}^{-1}(u_0)\circ \cdots \circ\theta_{\mathfrak{d}_{s'+1}}^{-1}(u_0)\\
         =& \prod_{\gamma:Z_{\gamma}(u_0)\in V}\big(T^{u_0,u_2}\mathcal{K}_{\gamma}\big)(u_0)^{\Omega^{trop}(\gamma;u_2)}  \\
         \rightarrow & \prod_{\gamma:Z_{\gamma}(u_2)\in V}\mathcal{K}_{\gamma}(u_2)^{\Omega^{trop}(\gamma;u_2)}=A_V(u_2), 
       \end{align*}where the terms connected by arrows means they
     are related by the parallel transport. 
\end{proof}

\begin{rmk} \label{46}
  The smoothness of holomorphic symplectic $2$-form constructed in \cite{GMN} is interpreted as
  Kontsevich-Soibelman wall-crossing formula.
\end{rmk}

\begin{rmk} \label{1049}
Although we don't define the invariants $\Omega^{trop}(\gamma;u)$ when $\partial \gamma=0$, one can still expect the wall crossing structures of $\Omega^{trop}(\gamma;u)$ for charges with $\partial \gamma\neq 0$ and $\partial \gamma =0$ decouple from Lemma \ref{27}. 
\end{rmk}

Notice that the affine structure on $B_{\vartheta}$ and $B_{-\vartheta}$ are actually the same one. Thus we have the following reality condition:
\begin{prop} \label{2010}
For each $u\in B_0$ and $\gamma\in \Gamma$, we have
\begin{align}
\Omega^{trop}(\gamma;u)=\Omega^{trop}(-\gamma;u)
\end{align}
\end{prop}

Following is another direct consequence due to the holomorphicity the central charge.
\begin{prop} \label{31} Let $\gamma_1,\gamma_2$ be two relative classes.
  Assume that $u_{\pm}\in B_0$ be on the different side of $W^{trop}_{\gamma_1+\gamma_2}$ with 
    
   \begin{enumerate}
      \item $\mbox{Arg}Z_{\gamma_1+\gamma_2}(u_+)=\mbox{Arg}Z_{\gamma_1+\gamma_2}(u_-)$.
      \item \begin{equation} \label{902}
                 \frac{\langle \gamma_1, \gamma_2 \rangle
                }{Im\big[Z_{\gamma_1} \bar{Z}_{\gamma_2}
                 \big]}(u_+)>0.
              \end{equation}
    \end{enumerate}
    Then every tropical disc of relative class $\gamma_1+\gamma_2$ end at $u_-$ has a unique natural extension to one with end at $u_+$ by adding the affine straight line segment from $u_+$ to $u_-$. 
   
\end{prop}

Given an elliptic fibration K3 surface $f: X\rightarrow
\mathbb{P}^1$ with holomorphic $(2,0)$-form $\Omega$, then any
$2$-form $\alpha$ on $\mathbb{P}^1$ such that
$\Omega'=\Omega+f^*\alpha$, $\Omega'\wedge \Omega'=0$ gives rise to
another elliptic fibration with same Jacobian. Moreover, any
elliptic fibration with the same Jacobian arises in above
construction. It is obvious that for any $\gamma\in H_2(X,L)$ and
$\tilde{v}$ a lifting of $v\in T\mathbb{P}^1$,
   \begin{align*}
      \int_{\partial \gamma}\iota_{\tilde{v}}f^*\alpha=0.
   \end{align*}
Therefore, changing elliptic fibred K3 surfaces within same Jacobian
doesn't change the affine structure and the scattering diagram. Notice that the construction of the scattering diagram and the definition of the generalized Donaldson-Thomas invariants $\Omega^{trop}$ do not depend on the K\"ahler form of the elliptic fibration. To sum up, we proved
\begin{thm} \label{1004}
  The invariants $\Omega^{trop}(\gamma;u)$ only depend on the Jacobian of
  the elliptic fibration and are independent of the choice of K\"ahler form
   $\omega$.
\end{thm}

\subsection{Generalized Donaldson-Thomas Invariants and Tropical Disc Counting}\label{1039}
In this section, we want to relate the generalized Donaldson-Thomas invariants to counting of tropical discs. First, we define the multiplicity of a tropical curve (with stop) following \cite{M2}:
\begin{definition}\label{1041}
Let $\phi$ be a tropical curve (with stop) has only trivalent interior vertices. The multiplicity of a tropical curve (with stop) $\phi$ is defined by
       \begin{equation}
           Mult(h)=\prod_{v\in G^{[0]}} Mult_v(\phi)
       \end{equation}
\end{definition}
However, neither the generalized Donaldson-Thomas invariant $\Omega^{trop}(\gamma,u)$ or $\tilde{\Omega}^{trop}(\gamma,u)$ are coming from usual weighted count of tropical curve of relative class $\gamma$ and with stop at $u$. We still need the following definition:

\begin{definition}\label{1042}
Given primitive vectors (might be repeated) $m_i\in \mathbb{Z}^2$, and vectors $\mathbf{w}=(\mathbf{w}_1,\cdots,\mathbf{w}_n)$ , 
       $\mathbf{w}_i=(w_{i1},\cdots,w_{il_i})\in \mathbb{Z}^{l_i}_{\geq
       0}$ such that 
      \begin{align*}
       0<w_{i1}\leq w_{i2}\leq \cdots \leq w_{il_i},
      \end{align*}
       we associate a number
  $N^{trop}_{\{m_i\}}(\mathbf{w})$ which counts the weighted count (in the sense of Definition \ref{1041}) of the number of trivalent tropical curves (with
  stop at a generic point) in $\mathbb{Z}^2\otimes \mathbb{R}$ with $l_i$ fixed position unbounded edges in the directions $m_i$ and multiplicities $w_{i,j}$, where $j=1,\cdots, l_i$.
      
\end{definition}

Note that the numbers $N^{trop}_{\{m_i\}}(\mathbf{w})$ do not depend on the generic position of the unbounded edges \cite{GPS}. Now here is the main theorem of this section to illustrate the relation between the generalized Donaldson-Thomas invariants and tropical geometry which is used to inductively compute all the generalized Donaldson-Thomas invariants $\tilde{\Omega}^{trop}(\gamma;u)$:
\begin{thm} \label{63} Assume there are $n$ BPS rays labeled by $\gamma_1,\cdots,\gamma_n$ intersect at a generic point of the wall of marginal stability $W^{trop}_{\gamma}$. Under the same  notation in Definition \ref{1042} and some more notations:
   \begin{enumerate}
    \item let $|\mathbf{w}_i|=\sum^{l_i}_{k=1}w_{ik}$.
    \item The set $\mbox{Aut}(\mathbf{w}_i)$ is the subgroup of the permutation group stabilizing the set $(w_{i1},\cdots,w_{il_i})$, and 
    \item  we denote
             $|\mbox{Aut}(\mathbf{w})|=\prod_i |\mbox{Aut}(\mathbf{w}_i)|$.
   \end{enumerate}
 Then when $u$ crosses the wall of marginal stability $W_{\gamma}^{trop}$, one has the
              following wall-crossing formula for $\tilde{\Omega}^{trop}$:
     \begin{align} \label{37}
        \Delta \tilde{\Omega}^{trop}(d\gamma)=\sum_{\mathbf{w}:\sum|\mathbf{w_i}|\gamma_i=d\gamma}\frac{N^{trop}_{\{\partial\gamma_i\}}(\mathbf{w})}{|Aut(\mathbf{w})|} \bigg(\prod_{1 \leq i \leq n, 1 \leq j \leq l_i} \tilde{\Omega}^{trop}(w_{ij}\gamma_i)
        \bigg).
     \end{align}In particular, the invariant $\tilde{\Omega}^{trop}$ is sum of product of tropical counts in (\ref{37}) at each vertex after infinitesimal
       deformation.
\end{thm}
\begin{proof}
 From the slab function (\ref{1045}) associated to initial BPS rays and Lemma \ref{1038}, we have
 \begin{align}\label{321}
 \tilde{\Omega}^{trop}(d\gamma_e)=\frac{(-1)^{d-1}}{d^2}, 
 \end{align}where
 $\gamma_e$ is the Lefschetz thimble from each singularity and $d\in \mathbb{Z}$. The theorem follows from recursively applying Theorem 2.8 \cite{GPS} at each vertex of the image of tropical discs representing $d\gamma$ and induction on the number of their vertices. Notice that there is no concept of degree in our context. However, any tropical disc $\phi$ with the stop $u\notin \Delta$ and at least one vertex in $B_0$, then $|Z_{\gamma_{\phi}}|>\lambda_0$. In particular, the energy filtration defined in (\ref{1066}) and the degree filtration in Theorem 2.8 \cite{GPS} are equivalent. 
\end{proof}

\begin{cor} \label{2006}
The invariant $\tilde{\Omega}^{trop}(\gamma;u)$ is the weighted count of the number of tropical discs in the Definition \ref{322} with only trivalent vertices. The weight of tropical discs is the weight in Definition \ref{2000} times $\frac{(-1)^{d-1}}{d^2}$ for each edge with weight $d$ and adjacent to a singularity.
\end{cor}
\begin{rmk}
   The sign $(-1)^{d-1}$ in Corollary \ref{2006} seems to give an refinement of the invariants. It may come from the stratification of the wall of marginal stability in symplectic geometry.
\end{rmk}

\begin{ex} \label{65}
   Assume there are two initial BPS rays labeled with relative classes $\gamma_1, \gamma_2$ hitting at $p$ from direction $(-1,0),(0,-1)\in T_pB$. In particular, the intersection point $p$ falls on the wall $W_{\gamma_1,\gamma_2}$ and we want to know how the generalized Donaldson-Thomas invariants jump on one side of the wall to the other side.
     \begin{enumerate}
        \item  In Figure \ref{1051} (a), it seems that there is only one tropical disc. However, there are actually two figures of tropical discs that contribute to the computation of $\Delta\tilde{\Omega}^{trop}(\gamma_1+2\gamma_2)$. Notice that the one in Figure \ref{1051} (c) is a tropical disc in the sense of Definition \ref{322} but not a tropical disc in the usual sense: 
           \begin{align*}
           (1,2)&= 1\cdot (1,0)+2\cdot(0,1) \Rightarrow N^{trop}=1, Aut=2 \leadsto \frac{1}{2}\cdot 1\cdot1^2=\frac{1}{2} \\
                &= 1\cdot (1,0)+1\cdot(0,2) \Rightarrow N^{trop}=2, Aut=1 \leadsto \frac{2}{1}\cdot \frac{-1}{2^2}=\frac{-1}{2}
            \end{align*} Therefore, we have 
            \begin{align*}
             \Delta\tilde{\Omega}^{trop}(\gamma_1+2\gamma_2)=\frac{1}{2}+\frac{-1}{2}=0.
            \end{align*}
            Applying the Lemma \ref{1038}, we have  $\Omega^{trop}(\gamma_1+2\gamma_2)=0$. 
     \begin{figure}\label{1051}
     \begin{center}
     \includegraphics[height=1.5in,width=3in]{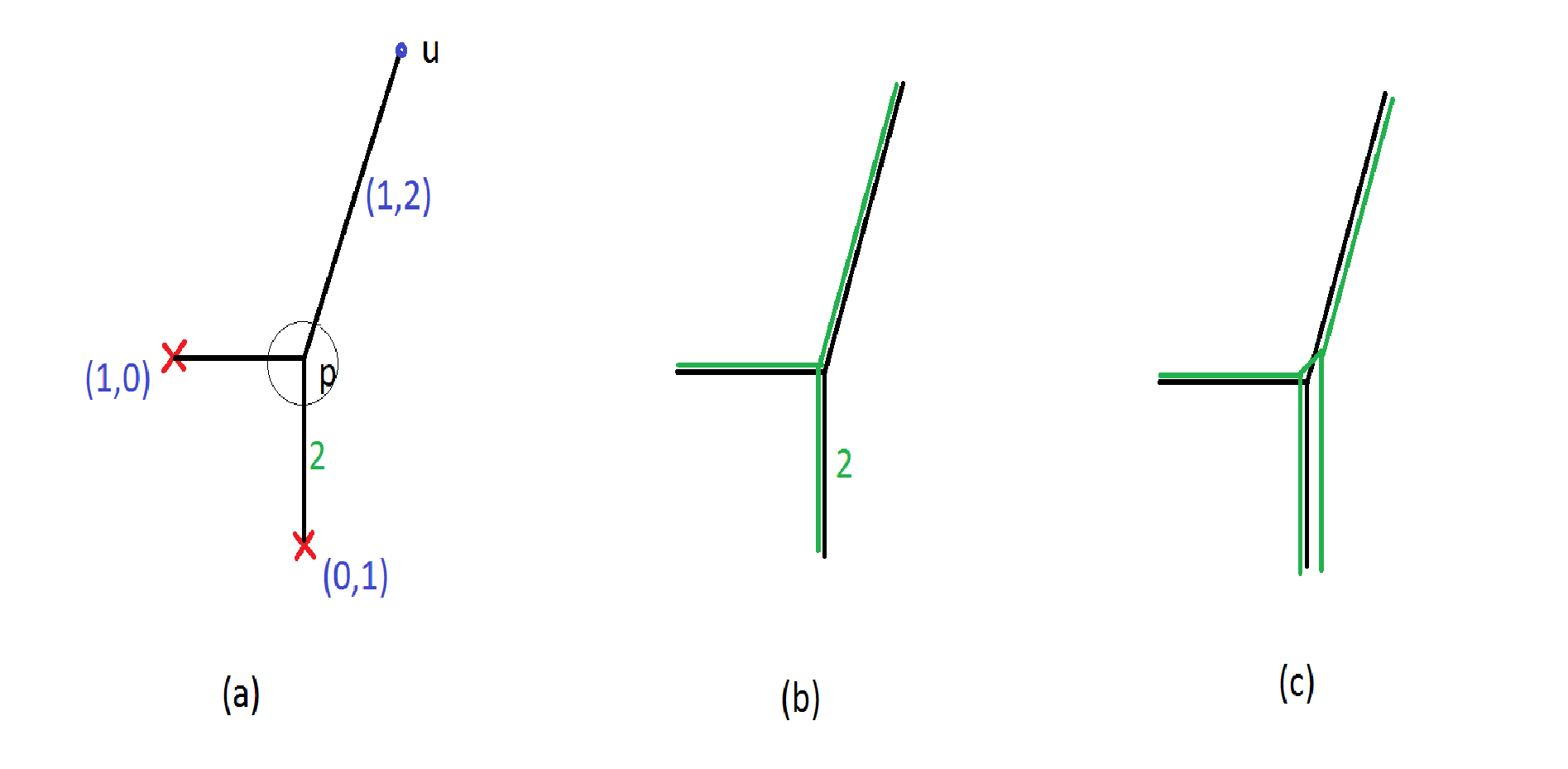}
     \caption{In the Example \ref{65}, there are two  tropical discs (drawn in green) shown in (b) and (c) corresponding to the splitting $(1,2)=(1,0)+(0,2)$ and $(1,2)=(1,0)+2\cdot (0,1)$ respectively. The bounded edges should be viewed as the edges shrink to a point and thus the two tropical discs have the same image.}
     \end{center}
     \end{figure}       
                       
          \item To compute $\Delta\tilde{\Omega}^{trop}(2\gamma_1+2\gamma_2)$, there are four figures of tropical discs contribute to the invariant (See Figure \ref{1053}).
             \begin{align*}
           (2,2)&= 1\cdot (2,0)+1\cdot(0,2) \Rightarrow N^{trop}=4, Aut=1 \leadsto \frac{4}{1}\cdot \frac{-1}{2^2}\frac{-1}{2^2}=\frac{1}{4} \\
                &= 1\cdot (2,0)+2\cdot(0,1) \Rightarrow N^{trop}=4, Aut=2 \leadsto \frac{4}{2}\cdot \frac{-1}{2^2}\cdot 1^2=\frac{-1}{2} \\
                &= 2\cdot (1,0)+1\cdot(0,2) \Rightarrow  \mbox{same as above}=\frac{-1}{2}      \\
                &= 2\cdot (1,0)+2\cdot(0,1) \Rightarrow N^{trop}=2, Aut=2^2 \leadsto \frac{4}{2^2}\cdot 1^2\cdot1^2=\frac{1}{2}.
            \end{align*}Therefore, we have 
             \begin{align*}
            \Delta\tilde{\Omega}^{trop}(2,2)=\frac{1}{4}+\frac{-1}{2}+\frac{-1}{2}++\frac{1}{2}=\frac{-1}{4}
            \end{align*}
             and $\Delta \Omega^{trop}(2\gamma_1+2\gamma_2)=0$ by the Lemma \ref{1038}. 
    \begin{figure}\label{1053}
         \begin{center}
         \includegraphics[height=1.5in,width=3in]{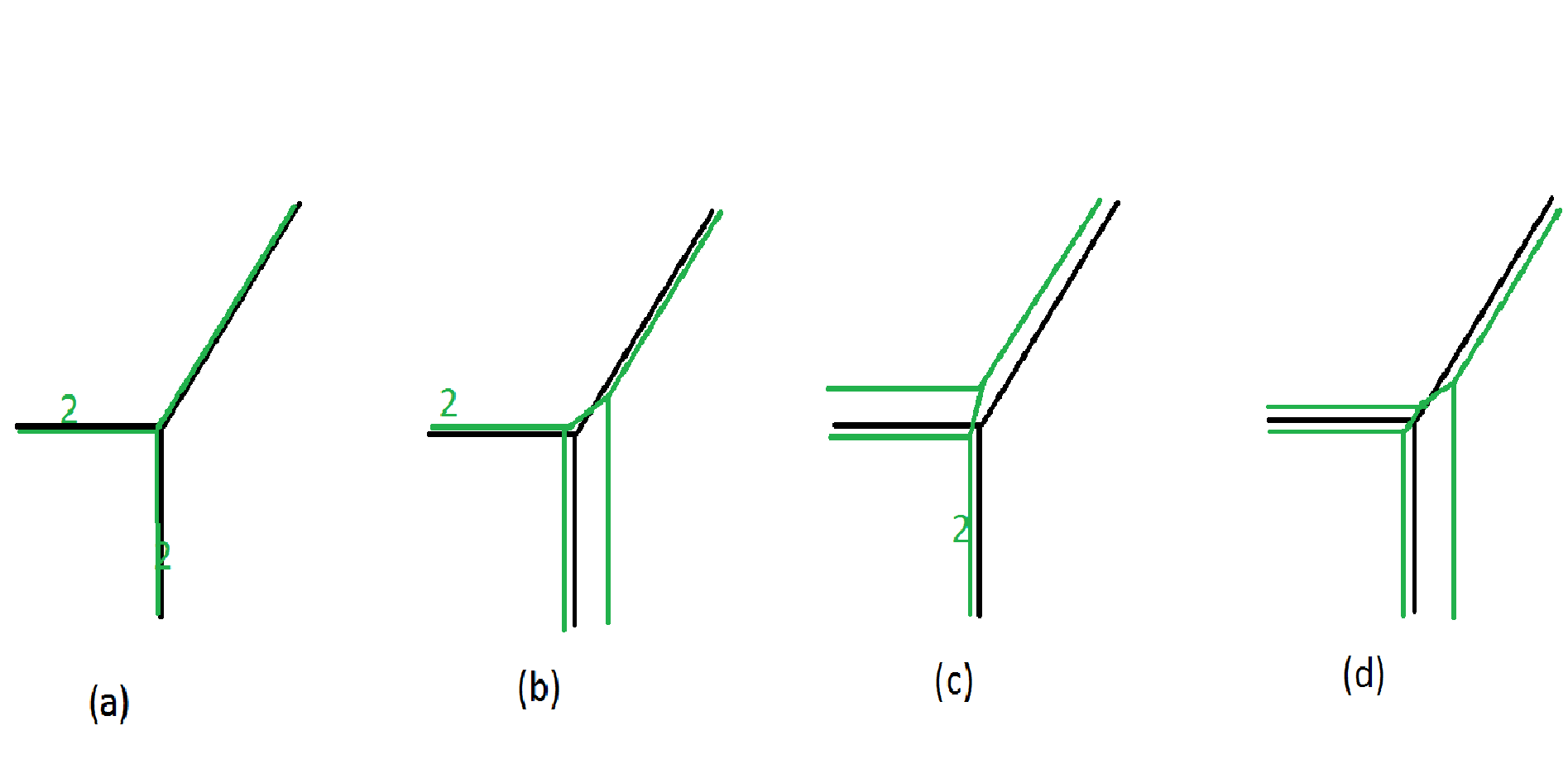}
         \caption{(a)$(2,2)= 1\cdot (2,0)+1\cdot(0,2)$ (b) $(2,2)= 1\cdot (2,0)+2\cdot(0,1)$ (c) $(2,2)= 2\cdot (1,0)+1\cdot(0,2)$ (d) $(2,2)=2\cdot (1,0)+2\cdot(0,1)  $}
         \end{center}
         \end{figure}           
            
     \end{enumerate}
     Actually, using the Theorem \ref{63} one can derive
          \begin{align*}
         \Delta \Omega^{trop}(\gamma)=\begin{cases} 1& \text{if $\gamma= \pm(\gamma_1+\gamma_2),$}\\
                      0& \text{otherwise,}
                                                                \end{cases}
          \end{align*} which is the famous pentagon identity given in \cite{KS2}:
          \begin{align*}
         \mathcal{K}_{\gamma_1}\mathcal{K}_{\gamma_2}=\mathcal{K}_{2}\mathcal{K}_{\gamma_1+\gamma_2}\mathcal{K}_{\gamma_1}.
          \end{align*}
        
\end{ex}

\subsection{Revisit of Definition of Tropical Discs}
In Definition \ref{322}, we restrict the tropical discs can only pass through the singularities via the monodromy invariant direction. However, it might be inevitable to have BPS rays passing through the singularity for all $\vartheta\in S^1$. To keep Corollary \ref{2006} holds for all $\vartheta$ and $\Omega^{trop}$ be locally constant, we may want to generalize the definition of tropical discs on $B_\vartheta$ to be the limit of tropical discs on $B_{\vartheta'}$ in the sense of Definition \ref{322}, $\vartheta'\rightarrow \vartheta$. For instance (under the same notation as in Definition \ref{322}), the interior of an edge $e$ can hit a singularity $p$. Then the affine structure with singularities of $B$ around $\phi(p)$ can locally be modeled by two charts:
    \begin{align*}
      U_1=\mathbb{R}^2\backslash \{(x,0)|x\geq 0\},
      U_2=\mathbb{R}^2\backslash \{(x,0)|x\leq 0\}
    \end{align*}with coordinates 
    \begin{align*}
      \phi_1:U_1\longrightarrow \mathbb{R}^2, \\
            (x,y)\mapsto (x,y)
    \end{align*} and
    \begin{align*}
    \phi_2&:U_2\longrightarrow \mathbb{R}^2 \\
           &  (x,y)  \mapsto \begin{cases} (x+y,y), & \text{if $x\geq 0$}\\
                       (x,y),& \text{otherwise.}
                                                                 \end{cases}
    \end{align*}
    then $\phi(e)$ near $\phi(p)$ can be the limit of affine lines in $U_1$ or $U_2$. 
 In particular, $\phi^{-1}(B_0)$ can be finite union of embedding of affine open intervals.
Notice that the limit of affine lines in $U_1$ can look bend in $U_2$ around the origin and vice versa. One can have more complicated situation such as the limit of a vertex of tropical discs hit the singularity. In any cases, we extend the definition of charges and central charges of the generalized tropical discs by parallel transport. 

At the end of this section, we will have an interesting observation.
In general, it is not an easy task to know whether or not a tropical
curve can be lifted as an actual holomorphic curve. Here we provide
a lifting criteria for certain tropical curves on K3 surfaces which is
independent of the later context.
\begin{thm}
  Assume that  $X_{\vartheta}$ admits a special Lagrangian fibration with a special Lagrangian section and there exists a trivalent tropical curve $\phi$ with multiplicity one on each edge attached to the singularity of affine base of $X_{\vartheta}$, then
  $X_\vartheta$ contains a holomorphic curve in the homology class $[\phi]$.
 \end{thm}
\begin{proof}
  A real integral class $[\phi]$ is of $(1,1)$-type if and only if the pairing with the holomorphic volume form is zero. It is easy to check that $\int_{[\phi]}\Omega_{\pm \vartheta}=0$. From Riemann-Roch theorem, it is easy to see that every integral
  $[\phi]\in H^{1,1}$ if and only if $\int_{[\phi]}\Omega=0$. The latter condition cuts out a hyperplane in the period
  domain and thus meets the twistor $\mathds{P}^1$ (as a quadric in period domain) at exactly $2$ points (actually one with its complex
  conjugate). Thus every real cycle with self-intersection $\geq -2$
  can be realized as holomorphic cycle exactly twice in any twistor
  family.

  Therefore, to prove $X_{\vartheta}$ contains a holomorphic curve in homology class $[\phi]$, it suffices to prove $[\phi]^2\geq -2$.
  Then by tropical perturbative intersection theory, we have     
  \begin{align*}
        [\phi]^2 \geq -\#v_{out}+\sum_{v_{int}}(\mbox{deg}(v_{int})-2)\geq
        -2.
     \end{align*} The last inequality follows easily from induction
     on the number of vertices and edges. Notice that the proof also
     holds for tropical curves with valency $2$ vertices.
\end{proof}

\section{Holomorphic Discs in Elliptic K3 Surfaces}
 In this section, we will study holomorphic discs in K3 surfaces. We first notice  the following observations: 
\begin{prop}\label{600}
Locally the set of special Lagrangian torus fibres bounding holomorphic
discs of the same relative class in $X_{\vartheta}$ all fall above an affine hyperplane on the base affine manifold $B_{\vartheta}$.
\end{prop}
\begin{proof}
   Assume $\{L_t\}_{t\in (-\epsilon,\epsilon)}$ are a family of special Lagrangian torus fibres bounding holomorphic discs in relative class $\gamma_t\in H_2(\underline{X},L_t)$ in $X_{\vartheta}$. Then 
     \begin{align*}
      \int_{\gamma_t}\Omega_{\vartheta}=\int_{\gamma_t}\omega-i\mbox{Re}(e^{-i\vartheta}\Omega)=0. 
     \end{align*}
  In particular, $L_t$ are confined by the equation
     \begin{align*}
       f_{\gamma}:=\mbox{Re}(e^{-i\vartheta}Z_{\gamma})=0,
     \end{align*}which all sit above an affine hyperplane. 
\end{proof}
\begin{rmk}
  From the proof of Proposition \ref{600}, the prescribed affine line is special in the sense that the corresponding central charge $Z_{\gamma}$ has constant phase along $\{f_{\gamma}=0\}$. We will call them special affine lines with respect to phase $\vartheta$.
\end{rmk}

We have support property \cite{KS2}\footnote{The support property is required to define stability
data for a suitable category. It is also needed
to prove the convergence of hyperK\"ahler metric in \cite{GMN} which we don't discuss it here.} for
holomorphic discs because holomorphic cycles are calibrated .
Namely,
\begin{prop}
 There exists $\delta >0$, such that
 \begin{equation*}
    \frac{|Z_{\gamma}|}{\| \gamma \|} > \delta
 \end{equation*}
 for all $\gamma$ which can be represented by holomorphic cycles in $X_{\vartheta}$ for some $\vartheta$.
\end{prop}
\begin{proof}
   We first choose a basis $\{(\alpha_i,\gamma_i)\}$ of $H^2(X_{\vartheta},L)$
   ,where $\alpha_i\in \mathcal{A}^2(X_{\vartheta})$, $\gamma_i\in \mathcal{A}^1(L)$ satisfying $d\alpha_i=0$ and $\alpha_i|_{L}=d\gamma_i$. Then there is a non-degenerate
   pairing given by
   \begin{align*}
     H_2(X_{\vartheta},L) &\times H^2(X_{\vartheta},L) \longrightarrow \mathbb{R} \\
        (\gamma,& \quad (\alpha_i,\gamma_i)) \mapsto
         \int_{\gamma}\alpha_i-\int_{\partial
         \gamma}\gamma_i=\int_{\gamma}\alpha_i-d\tilde{\gamma}_i,
   \end{align*} where $\tilde{\gamma}_i$ are fixed smooth extension
   of $\gamma_i$ to whole $X_{\vartheta}$. If $\gamma$ can be
   represented by a holomorphic cycle (also denoted by $\gamma$) in $X_{\vartheta}$, then we have $\mbox{Vol}(\gamma)=| Z_{\gamma}|$.  Therefore, we have the following inequality
   \begin{align*}
        |\int_{\gamma}(\alpha_i,\gamma_i)| \leq
         \sup_{\substack{v_1,v_2\in T_pX,\\ |v_1\wedge v_2|=1}}\langle \alpha_i-d\tilde{\gamma}_i, v_1\wedge v_2
        \rangle \cdot\mbox{Vol}(\gamma)\leq C\cdot |Z_{\gamma}|,
   \end{align*}where $C$ is a constant independent of $\vartheta$. Summing $i$ through basis of $H^2(X_{\vartheta},L)$,
   the left hand side gives a norm on $H_2(X_{\vartheta},L)$ and we
   prove the support property. 
\end{proof}

\subsection{Holomoprhic Discs in the $S^1$-Family} \label{2009}
Let $(X,\omega,\Omega)$ be a hyperK\"ahler manifold (not necessarily
compact) with K\"ahler form $\omega$ and holomorphic symplectic
$2$-form $\Omega$, then the hyperK\"ahler triple $(\omega,\Omega)$
will give a twistor $\mathbb{P}^1$. Let $L$ be a holomorphic
Lagrangian with respect to $\Omega$, then by Remark \ref{300} there is an $S^1$-family of
complex structures in the twistor family such that $L$ is special
Lagrangian. 

Let $\mathcal{M}_{k+1,\gamma}(\mathfrak{X},L)$ be the moduli space of stable
pesudo-holomorphic discs in the above $S^1$-family \footnote{The $S^1$-family and the moduli space $\mathcal{M}_{k+1,\gamma}(\mathfrak{X},L)$ depends on the choice of the holomorhpic symplectic $2$-form and the Ricci-flat metric $\omega$. However, we will prove that the invariant we defined is independent of these choices when $\partial \gamma \neq 0\in H_1(L)$. Therefore, we will omit the dependence for shorter notations.} with boundary on the fixed
special Lagrangian $L$ with $k+1$ boundary marked points in
counter-clockwise order.

 Let $f:\Sigma\rightarrow X_{\vartheta}\in \mathcal{M}_{0,\gamma}(\mathfrak{X},L)$.
  For each $\alpha \in \mathbb{R}Im\Omega_{\vartheta} \subseteq(\Omega_{\vartheta}^{2,0}
\oplus \Omega_{\vartheta}^{0,2})_{\mathbb{R}}$, Lee \cite{L} associated an
endomorphism $K_{\alpha}:TX_{\vartheta} \rightarrow TX_{\vartheta}$ by
\begin{equation*}
   g(u,K_{\alpha}v)=\alpha(u,v)
\end{equation*}
and considered the following twisted $\bar{\partial}$-equation
\begin{equation} \label{1005}
    \bar{\partial}_{J_{\vartheta}}f=K_{J_{\vartheta}}(f,\alpha)=\frac{1}{2}K_{\alpha}(\partial_{J_{\vartheta}} f \circ
    j).
\end{equation} 
The map that satisfies the above twisted $\bar{\partial}$-equation (\ref{1005}) are
actually $J_{\vartheta,\alpha}$-holomorphic with 
\begin{equation} \label{220}
   J_{\vartheta,\alpha}=\frac{1-|\alpha|^2}{1+|\alpha|^2}J_{\vartheta}-\frac{2}{1+|\alpha|^2}K_{\alpha}
\end{equation}
in the equator $S^1 \subseteq \mathbb{P}^1$ of the twistor family making $L$ special
Lagrangian.

One of the key observation is the following: Assume $f'$ is
$J_{\vartheta,\alpha}$-holomorphic and in the same relative class as $f$, then
by Proposition 1.3 \cite{L},
\begin{align*}
    \int_{(\Sigma',\partial \Sigma')}|\bar{\partial}_{J_{\vartheta}} f'|^2 dv &=\int_{(\Sigma,\partial \Sigma)} g(\bar{\partial}f', K_{J_{\vartheta}}(f',\alpha))dv \\
                                                                                &=\int_{(\Sigma',\partial
                                                                                \Sigma')}f'^*\alpha
                                                                                \\
                                                                &=\int_{(\Sigma,\partial \Sigma)}f^*\alpha+\int_C dF^*\alpha+\int_{C'}F^*\alpha=0,
\end{align*}
where $F$ is any homotopy between $f$ and $f'$. The first term
vanishes because $f$ is $J$-holomorphic. The second term
vanishes because $\Omega_{\vartheta}$ is $d$-closed, and the last term is from
boundary of domain and vanishes because $\alpha|_L=0$. Therefore,
$\bar{\partial}_{J_{\vartheta}}f'\equiv 0$.

\begin{prop} \label{48}
   Given any relative class $\gamma \in H_2(X,L)$, there is at most one complex structure in equator $S^1$ in the twistor family such that
   $\gamma$ has a holomorphic representative.
    In particular, any holomorphic Riemann surface with boundary on the special Lagrangian in a K3 surface is rigid in the $S^1$-family.
\end{prop}
\begin{proof}
    Assume $f$ is both $J_{\vartheta}$ and $J_{\vartheta,\alpha}$-holomorphic and $[Im(f)]=\gamma$, then
     $K_{J}(\partial f \circ j)=0$ or $Im(\partial f) \subseteq
     Ker(K_{\alpha})$. If $[f]\neq 0$, then there exists holomorphic
     $v$ such that $df(v) \neq 0$. For each nonzero $\alpha \in \mathbb{R}\mbox{Im}\Omega_{\vartheta}$, $\alpha$ is a symplectic
$2$-form (here we use the fact $X$ is a hyperK\"ahler). Therefore,
$\alpha(u, df(v))
     \equiv 0$ implies $\alpha=0$. The only possible complex
     structure is $-J_{\vartheta}$ but it can only realize $-[\gamma]$ as a
     holomorphic cycle.
\end{proof}
Therefore, the moduli space for family
$\mathcal{M}_{k+1,\gamma}(\mathfrak{X},L)$ has the same underlying
space as the usual moduli space of holomorphic discs
$\mathcal{M}_{k+1,\gamma}(X_{\vartheta},L)$ for some $\vartheta \in
S^1$. However, we will equip it with different Kuranishi structure
in the next section. Another direct consequence of Proposition \ref{48} is the following:
\begin{cor}\label{9}
   $ImD\bar{\partial} \cap \{K_{J_{\vartheta}}(f,\alpha)| \alpha \in
   \mathbb{R} \mbox{Im}\Omega_{\vartheta} \}= \{0\}$. In particular, There
is a non-trivial map $T_{\vartheta}S^1 \rightarrow L^p(f^*TX\otimes
\Lambda^{0,1})$ induced from the $S^1$-family of hyperK\"ahler
manifolds.
\end{cor}
%\begin{proof}
%  Assuming $f$ is $J$-holomoprhic, then from the Proposition \ref{48} one has
%   \begin{equation*}
%       %\int_{\Sigma}g(D\bar{\partial}V,K_J(f,\alpha))dv=-\int_{\Sigma}g(\bar{\partial}f,K %_J(V,\alpha))dv+\int_{\Sigma}V^*\alpha=0+0=0,
%   \end{equation*}for any $V \in \mathcal{A}(f^*TX, (\partial f)^*TL)$.
%   The first term is well-defined by Holder's inequality and last
%   term is zero because $\alpha|_L=0$.
%\end{proof}

\subsection{Kuranishi Structure for Moduli Space of Holomorphic Discs in the $S^1$-Family}\label{998}
As mentioned in the introduction, the difficulties of defining open Gromov-Witten invariants on K3 surfaces are two-fold. First of all, bubbling phenomenon of pseudo-holomorphic discs may occur and usually the virtual fundamental cycle can not be defined. Secondly, the virtual dimension of the relevant moduli spaces is minus one. In other words, there are no pseudo-holomorphic discs with respect to a generic almost complex structure. Therefore, even if the invariant can be defined, it will be zero and not interesting at all. 

Instead the usual Kuranishi structure defined on the moduli space $\mathcal{M}_{\gamma}(X_{\vartheta},L)$, we will consider the holomorphic discs in the $S^1$-family discussed in the previous section and consider the Kuranishi structure of the family moduli space $\mathcal{M}_{\gamma}(\mathfrak{X},L)$. The construction of the Kuranishi structure can be found in \cite{F1}\cite{FOOO} and we will call it the "reduced" Kuranishi structure since the idea is similar to the reduced Gromov-Witten theory in algebraic geometry. We will put the definition of Kuranishi structures and related terminology in the appendix for reader's reference. Below we will sketch the proof of the construction.

\begin{thm} \label{11} \cite{F1}\cite{FOOO}
There exists a Kuranishi structure admits required properties listed
below:
\begin{enumerate}
  \item It is compatible with the forgetful maps for each $k\geq 1$,
       \begin{equation}
          \mathfrak{forget}_{k,0}:
          \mathcal{M}_{k,\gamma}(\mathfrak{X},L)\rightarrow
          \mathcal{M}_{k-1,\gamma}(\mathfrak{X},L)
       \end{equation}
  \item For each $k \geq 1$, the evaluation map $\{ev_i,ev_{\vartheta}\}:\mathcal{M}_{k,\gamma}(\mathfrak{X},L)\rightarrow L^k\times S^1_{\vartheta}$ are weakly
  submersive. Thus the fibre product of Kuranishi structure in 4.
  make sense.
  \item It is invariant under the cyclic permutation of the boundary
  marked points.
  \item For the decomposition of the boundary of moduli spaces,
  \begin{align} \label{12}
         \partial \mathcal{M}_{k+1,\gamma}(\mathfrak{X},L)= \bigcup_{1\leq i\leq j+1
         \leq k+1} \bigcup_{\substack{\gamma_1+\gamma_2=\gamma,\\ Z_{\gamma_1}/Z_{\gamma_2}\in
         \mathbb{R}_{>0}}}
         \mathcal{M}_{j-1+1,\gamma_1}(\mathfrak{X},L) \notag \\
          \; {}_{(ev_0,ev_{\vartheta})} \! \times_{(ev_i,ev_{\vartheta})}
         \mathcal{M}_{k-j+1,\gamma_2}(\mathfrak{X},L).
       \end{align} and
       \begin{align}\label{7}
             \partial
            \mathcal{M}_{0,\gamma}&(\mathfrak{X},L)=\bigcup_{\tilde{\gamma}:i_*(\tilde{\gamma})=\gamma}
                                                    \big( \mathcal{M}^{cl}_{1,\tilde{\gamma}} \underset{\frak{X}}{\times}   (L\times S^1_{\vartheta})\big)\notag \\
            &\cup \bigcup_{\substack{\gamma_1+\gamma_2=\gamma,\\ Z_{\gamma_1}/Z_{\gamma_2}\in \mathbb{R}_{>0} }}(\mathcal{M}_{1,\gamma_1}(\mathfrak{X},L)\; {}_{(ev_0,ev_{\vartheta})} \! \times_{(ev_0,ev_{\vartheta})}
            \mathcal{M}_{1,\gamma_2}(\mathfrak{X},L))/\mathbb{Z}_2
         \end{align}  
   %  \begin{equation}
   %    \partial \mathcal{M}_{k+1,\gamma}(\mathfrak{X},L)=\bigcup_{1\leq i\leq j+1
   %    \leq k+1} \bigcup_{\substack{\gamma_1+\gamma_2=\gamma, \\ Z_{\gamma_1}/Z_{\gamma_2}\in \mathbb{R}}}
   %    \mathcal{M}_{j-1+1,\gamma_1}(\mathfrak{X},L) \; {}_{(ev_0,ev_{\vartheta})}  \! \times_{(ev_i,ev_{\vartheta})}
   %    \mathcal{M}_{k-j+1,\gamma_2}(\mathfrak{X},L).
   %  \end{equation} and
   %  \begin{equation}
   %        \partial
   %        \mathcal{M}_{0,\gamma}(\mathfrak{X},L)=\bigcup_{\substack{\gamma_1+\gamma_2=\gamma,\\ Z_{\gamma_1}/Z_{\gamma_2}\in \mathbb{R}}}(\mathcal{M}_{1,\gamma_1}(\mathfrak{X},L)\; {}_{(ev_0,ev_{\vartheta})} \! \times_{(ev_0,ev_{\vartheta})}
   %        \mathcal{M}_{1,\gamma_2}(\mathfrak{X},L))/\mathbb{Z}_2.
   %     \end{equation}
     the restriction of the Kuranishi structure to the boundary coincides with the fibre product of the Kuranishi
     structures of the decomposition.

\end{enumerate}
\end{thm}
\begin{proof}
   Step 1: For each point $\bold{p}=f:(\Sigma=\cup
\Sigma_a,\partial\Sigma)\rightarrow (X,L)$ (we include
$\overrightarrow{z}$ in $\Sigma$ for
   simplicity) in the moduli space $\mathcal{M}_{k,\gamma}(\mathfrak{X},L)$, we will construct a Kuranishi neighborhood. We first consider the case when the domain is stable, namely, then
   the automorphism of each component is finite. Let
       \begin{align*}
        W^{1,p}(f^*TX,(\partial f)^*TL)=\{(v_a)\in\oplus_a W^{1,p}(\Sigma_a; f^*TX,&(\partial f)^*TL)| \\  &v_a \mbox{ coincides on nodes}\},
       \end{align*}
   where $v_a \in W^{1,p}(\Sigma_a; f^*TX,(\partial f)^*TL)$ satisfies
   $v_a\in W^{1,p}(\Sigma_a, f^*TX)$ and $v_a|_{\partial\Sigma_a}\in
   W^{1-1/p,p}(\partial \Sigma_a, f^*TL)$. We may choose $p$ large
   enough such that $v_a$ is continuous on $\Sigma_a$.

    We consider the following linearized operator
      \begin{align*}
          D\bar{\partial}_{new} : W^{1,p}(f^*TX,(\partial f)^*&TL)\times \mathbb{R}_{\vartheta} \rightarrow L^p(f^*TX \otimes \Lambda^{0,1}) \\
                                         &(w, \vartheta) \mapsto D\bar{\partial}w+ \vartheta K_J(f,\alpha)
      \end{align*}
  From Corollary \ref{9}, we know that $D\bar{\partial}_{new}(V,\vartheta)=0$ if and
   only if $D\bar{\partial}V=0$ together with $\vartheta=0$. Since
   $D\bar{\partial}_{new}$ is also Fredholm, one can choose
   $E_{\bold{p}}$ such that
       \begin{enumerate}
         \item $E_{\bold{p}}$ is a finite dimensional (complex) subspace of $L^p(\Sigma,w^*TX \otimes \Lambda^{0,1})$.
         \item $\mbox{Im}D\bar{\partial}_{new}+E_{\bold{p}}=L^p(\Sigma,w^*TX \otimes \Lambda^{0,1})$.
         \item Elements of $E_{\bold{p}}$ has support away from special points on $\partial \Sigma$.
         \item $E_{\bold{p}}$ is preserved under
         $\Gamma_{\bold{p}}=Aut({\bold{p}})$.
       \end{enumerate}
   We may enlarge the obstruction bundle $E$ such that
   $(ev_0,ev_{\vartheta})$ is weakly submersive. Given any small
   smooth deformation $f'$ of $f$, we can find a diffeomorphism of
   the domain
      \begin{align*}
         I_{f,f'}:\Sigma \rightarrow \Sigma'
      \end{align*}
  such that $I_{f,f'}$ smoothly depends on the domain of $f'$ and
  $I_{f,f'}$ is identity on a compact set away from marked points.

  We choose a unitary connection on $TX$ such that $L$ is totally
  geodesic and thus we can have a identification of obstruction
  bundle
      \begin{align} \label{101}
         L^{0,p}(\Sigma,f^*(TX\otimes \Lambda^{0,1})\cong
         L^{0,p}(\Sigma',f'^*(TX\otimes \Lambda^{0,1})
      \end{align}
 induced by $I_{f,f'}$ and again denoted by the same notation. Set
 $E_{f'}=I_{f,f'}(E_{\bold{p}})$ and we consider the equation
       \begin{align}\label{100}
          \bar{\partial}f'\equiv 0 \mbox{  mod }E_{f'}
       \end{align}
Let $U_f$ be the solutions of (\ref{100}) and it is a smooth
manifold of dimension $\mbox{dim }E_{\bold{p}}$ by implicit function
theorem. We define the section of Kuranishi chart by
       \begin{align*}
            s(f')=\bar{\partial}f' \in E_{f'}
       \end{align*}

   Notice that $\mathcal{M}_{\gamma}(\mathfrak{X},L)=\mathcal{M}_{\gamma}(X_{\vartheta},L)$ for some $\vartheta \in
   S^1$ as a topological space, therefore the gluing analysis is the same as the standard one. 
Therefore we construct a Kuranishi chart for each element
$\mathbf{p}=[f:((\Sigma,\partial{\Sigma}),\overrightarrow{z})
\rightarrow (X,L)] \in \mathcal{M}_{k,\gamma}(\mathfrak{X},L)$ when
the domain is stable.

Step 2: If any of the component of the domain $\Sigma_a$ is
unstable, we will follow the construction in \cite{F1} and Appendix
\cite{FO}. We first add interior marked points to make $\Sigma_a$
stable and denote it by $\Sigma_a^+$. we add the marked points in
the way that $\Gamma_{\bold{p}}$ acts on additional marked points
freely. Since $f_a$ is non-degenerate, it is immersed at generic
point on $\Sigma_a$ and without loss of generality we can assume
$f_a$ is immersed at additional marked points.

For each additional marked point $p$, we take a $2$-dimensional
submanifold $\mathcal{D}_p\subseteq X$ such that $\mathcal{D}_p$
intersect with the image of $f_a$ transversally at $p$ and
$\mathcal{D}_p=\mathcal{D}_{\gamma\cdot p}$, for each $\gamma\in
\Gamma_{\bold{p}}$. We will denote the holomorphic maps with addtion
marked points by
$\bold{p}^+=[f^+:(\Sigma^+,\partial\Sigma^+)\rightarrow (X,L)]$

Now we may assume the domain with addition marked points
$\Sigma^+\in \mathcal{M}_{0,k+k'}$, where $k'$ is the number of the
additional marked points. Since $\mathcal{M}_{0,k+k'}$ admits an
orbifold structure, we may assume a neighborhood of
$\Sigma^+\in\mathcal{M}_{0,k+k'}$ is parametrized by
$V(\Sigma^+)/\mbox{Aut}(\Sigma^+)$. Follow the same procedure in
Step 1, we construct an Kuranishi chart $V_{\bold{p}^+}$ for
$\bold{p}^+$. Let $ev_{\mbox{add}}: V_{\bold{p}^+}\rightarrow
X^{k'}$ be the evaluation map of added points. Then
   \begin{align*}
      V_{\bold{p}}:=V_{\bold{p}^+} \; {}_{ev_{\mbox{add}}} \! \times \prod_{\substack{p: \mbox{additional}\\ \mbox{marked
      points}}}
      \mathcal{D}_p
   \end{align*}
is a smooth manifold of expected dimension because of the
transversality condition of $\mathcal{D}_p$. We define
$E_{\bold{p}}$ by
   \begin{align*}
      E_{f'}=E_{\bold{p}^+}=\bigoplus_a E(\Sigma^+_a)
   \end{align*} (up to a parallel transport)
and the Kuranishi map $s_{\bold{p}}$ by
   \begin{align*}
     s_{\bold{p}}(f')=\bar{\partial}f' \in E_{f'},
   \end{align*} which is $\Gamma_{\bold{p}}$-equivariant by construction.

Let $s_{\bold{p}}(f')=0$, then $f':(\Sigma^+_{f'},\partial
\Sigma^+_{f'})\rightarrow (X,L)$ is pseudo-holomorphic and we get
$\psi_{\bold{p}}(f'):=\tilde{f}:(\Sigma^+_{f'},\partial
\Sigma^+_{f'})\rightarrow (X,L)$, where the later is induced by $f'$
by forgetting those additional marked points. So far, we construct a
Kuranishi chart with notation changed by
$(V^0_{\bold{p}},E^0_{\bold{p}},\Gamma_{\bold{p}},s^0_{\bold{p}},\psi^0_{\bold{p}})$
for every point $\bold{p}\in\mathcal{M}_{k,\gamma}(\mathfrak{X},L)$.

Step 3: By Gromov compactness theorem, we find a finite cover
    \begin{align*}
      \mathcal{M}_{k,\gamma}(\mathfrak{X},L)=\bigcup_{\bold{p}\in \mathfrak{U}}
      \psi^0_{\bold{p}}((s^0_{\bold{p}})^{-1}(0)/\Gamma_{\bold{p}}).
    \end{align*}
Choose closed subset $W_{\bold{p}}$ of
$\psi^0_{\bold{p}}((s^0_{\bold{p}})^{-1}(0)/\Gamma_{\bold{p}})$ for
each $\bold{p}\in \mathfrak{U}$ such that
    \begin{align*}
      \mathcal{M}_{k,\gamma}(\mathfrak{X},L)=\bigcup_{\bold{p}\in
      \mathfrak{U}} \mbox{Int}W_{\bold{p}}.
    \end{align*}
and we set
   \begin{align*}
      E{\bold{p}}=\bigoplus_{\bold{p}'\in \mathfrak{U}(\bold{p})}E^0_{\bold{p}'}, \mbox{ where }
      \mathfrak{U}(\bold{p})=\{\bold{p}'\in\mathfrak{U}|\bold{p}\in
      W_{\bold{p}'}\}.
   \end{align*}
 This step the closeness of $W_{\bold{p}}$ guarantee the
coordinate change of Kuranishi structure. Thus we construct a
Kuranishi structrue on a fixed moduli space
$\mathcal{M}_{k,\gamma}(\mathfrak{X},L)$ but without compatibility
condition.

Step 4: At the end, we construct the Kuranishi structures on
$\mathcal{M}_{k,\gamma}(\mathfrak{X},L)$ inductively on
$|Z_{\gamma}|$ such that they are compatible with the decomposition
of boundary. Notice that for the case $[\partial \gamma]=0\in H_1(L)$, there is an
additional boundary component of
$\mathcal{M}_{0,\gamma}(\mathfrak{X},L)$
  \begin{align*}
     \mathcal{M}^{cl}_{1,\tilde{\gamma}}(\mathfrak{X})\times_M L,
  \end{align*}
and the proof is also similar.

\end{proof}

\begin{rmk} It is pointed out in \cite{F2} that we choose $E_{\alpha}$ such that its elements has support
  (uniformly) away from the special points for two reasons:
    \begin{enumerate}
      \item We don't have to perturb in a neighborhood of nodal points and thus the gluing analysis is easier.
      \item It is easier to identify the sections of obstruction bundle of each
      component with the sections of obstruction bundle after
      gluing.
    \end{enumerate}
\end{rmk}

\begin{rmk}
  Here we didn't quotient the line bundle given by $K_J(f,\alpha)$
  because it is hard to identify the quotients with those of
  perturbation of $f$.
\end{rmk}

\begin{thm}\label{1} \cite{F1}\cite{FOOO}
 For each $\epsilon$ and $E_0$, there exists a system of continuous
 family of multi-sections $\{\mathfrak{s}_{k,\gamma}\}$ on $\{\mathcal{M}_{k,\gamma}|k\geq 0, |\int_{\gamma}\Omega|<\epsilon\}$ such that
   \begin{enumerate}
      \item It is $\epsilon$-close to the Kuranishi map in $C^0$
      sense.
      \item It is compatible with forgetful map $\mathfrak{forget}$.
      \item It is invariant under the cyclic permutation of the boundary marked points.
      \item All the relevant evaluation maps induce submersion on zero sets of multi-sections.
      \item For the decomposition of the boundary of moduli spaces as in (\ref{12}) and (\ref{7})),
   the restriction of the multi-sections on the boundary of the
     moduli space coincide with the fibre product of the multi-sections from
     decomposition. 
   \end{enumerate}
\end{thm}
\begin{proof} The proof is similar to theorem 5.1 \cite{F1}.
By Gromov compactness, the class achieve minimal area has compact
moduli space consisting only smooth domain. One constructs
multi-sections on the moduli space using appendix A.3.  By
induction, we have constructed multi-sections on both factor of
right hand side of (\ref{7}) for $\gamma$ with $|Z_{\gamma}|<E_0$. For
second terms on the right hand side of (\ref{7}) and right hand side
of (\ref{12}), multi-sections constructed are compatible on the
overlapped part of moduli spaces by lemma 7.2.55 \cite{FOOO}.
Therefore, we can use the fibre product of multi-sections on each
factor to gives multi-sections on the left hand side of
(12),(\ref{7}). For the first term of right hand side, \cite{FO}
constructs such multi-sections. We use lemma 3.14 \cite{FO} to
extend the multi-sections from the boundary to the whole moduli
space and the multi-sections can satisfy (\ref{8}). It is easy to
see that the extension still has the transversal property. The
assumption of energy bound guarantees we only have finitely many
steps so that the $(1)$ of Theorem \ref{1} can be achieved.
\end{proof}
%We will use $[\mathcal{M}]^{vir}$ to denote the zero locus of the
%perturbed multi-sections of corresponding moduli space
%$\mathcal{M}$. In particular, $[\mathcal{M}]^{vir}$ is a smooth
%manifold (with corners).

%\begin{rmk}
%  See 7.2.3 \cite{FOOO} explains why we can not construct
%  multi-section for all moduli space simultaneously.
%\end{rmk}
Notice that the Kuranishi structures and multi-sections given in Theorem \ref{11} and Theorem \ref{1} are not unique. However, there is the following is the cobordism version of Theorem \ref{11} and Theorem \ref{1}.

Assume there is a $1$-parameter family of hyperK\"ahler structure $X_t=(\underline{X},\omega_t,\Omega_t), t\in [0,1]$ such that a submanifold $L$ are holomorphic. Let
  $\mathcal{M}_{k,\gamma}(\mathfrak{X}_t,L)$ denote the moduli space of stable holomorphic discs with boundaries on $L$ in the $S^1$-family complex structures induced by $X_t$. Fix a choice of Kuranishi structures and multi-sections on $\mathcal{M}_{k,\gamma}(\mathfrak{X}_t,L)$, for $t=0,1$, from Theorem \ref{11} and Theorem \ref{1}. 

%Assume $f:(D^2,\partial D^2)\rightarrow (X,L)$ is a holomorphic disc. 
%The new linearized
%$\bar{\partial}$ equations for $1$-parameter family now becomes
%\begin{align*}
%   D\bar{\partial}: W^{1,p}(f^*TX,(\partial f)^*TL)\times
%   \mathbb{R}_{\vartheta} \times \mathbb{R}_t
%   \rightarrow L^p(f^*TX \otimes \Lambda^{0,1})\\
%   (w,\quad \vartheta, \quad t) \qquad \longmapsto D\bar{\partial}w+\vartheta
%   K_J(f,\alpha)+tY
%\end{align*}
%where $Y$ is the tangent of complex moduli along the path.  Using
%similar argument in Theorem \ref{1}, we have

\begin{thm} \cite{F1}\cite{FOOO}\label{64} Given $E_0>0$, there is a system of Kuranishi structures and families of
   multi-sections $\mathfrak{s}_{k,\gamma}$ on
   $\bigcup_{t\in [0,1]}\mathcal{M}_{k,\gamma}(\mathfrak{X}_t,L)$, $k=0,1$ and $\partial \gamma\neq 0$,
   with $|\int_{\gamma}\Omega|<E_0$ satisfying properties below:
   \begin{enumerate}
     \item The multi-sections $\mathfrak{s}_{k,\gamma}$ are transverse to $0$.
     \item The structure is compatible with the forgetful maps.
     \item All the relevant evaluation maps are submersions restricted on the
     zero locus of the perturbed multi-sections.
     \item For the following decomposition of the boundary of moduli spaces,
        \begin{align} \label{61}
             \partial
           \bigg(\bigcup_{t\in [0,1]}\mathcal{M}_{0,\gamma}&(\mathfrak{X}_t,L)\bigg)=\cup
           \big(\mathcal{M}_{0,\gamma}(\mathfrak{X}_1,L)-\mathcal{M}_{0,\gamma}(\mathfrak{X}_0,L)\big) \notag \\
            &\cup
                                                               \bigcup_{\gamma_1+\gamma_2=\gamma}((\bigcup_{t\in [0,1]} \mathcal{M}_{1,\gamma_1}(\mathfrak{X}_t,L))\; {}_{(ev_0,ev_{\vartheta},ev_t)} \! \times_{(ev_0,ev_{\vartheta},ev_t)} (\bigcup_{t\in [0,1]}
                                                               \mathcal{M}_{1,\gamma_2}(\mathfrak{X}_t,L)))/\mathbb{Z}_2 \notag\\
        \end{align}
        the restriction of the multi-sections on the boundary of
        moduli spaces coincides with the fibre products of Kuranishi structures. \item Assume that $t_0$ is in the image of evaluation map $ev_t$ of (\ref{2016}) below. Namely,  
        \begin{align} \label{2016}
                 (\bigcup_{t\in [0,1]} \mathcal{M}_{1,\gamma_1}(\mathfrak{X}_t,L))\; {}_{(ev_0,ev_{\vartheta},ev_t)} \! \times_{(ev_0,ev_{\vartheta},ev_t)} (\bigcup_{t\in [0,1]}
                                                                                \mathcal{M}_{1,\gamma_2}(\mathfrak{X}_t,L))\
                  \end{align} and 
                  \begin{align}\label{2017}
                   \mathcal{M}_{1,\gamma_1}(\mathfrak{X}_{t_0},L)\; {}_{(ev_0,ev_{\vartheta})} \! \times_{(ev_0,ev_{\vartheta})}                                                                              \mathcal{M}_{1,\gamma_2}(\mathfrak{X}_{t_0},L)\
                  \end{align} are actually the same as topological spaces. Then the Kuranishi structures and multi-sections can be chosen to be the same on (\ref{2016}) and (\ref{2017}).

       \item The Kuranishi structures and multi-sections $\mathfrak{s}_{k,\gamma}$ restricted to the given one on $\mathcal{M}_{k,\gamma}(\mathfrak{X}_t,L)$, for $t=0,1$.

   \end{enumerate}
\end{thm}

%\begin{rmk}
%   Another choice is to consider the below moduli space in \cite{FOOO}.
%     \begin{align*}
%        %\mathcal{M}_{k,\gamma}(\{J_{\vartheta,t}\}:\mbox{top}(\vartheta),\mbox{twp}(t))\rig %htarrow [0,1]_t\times L^{k}.
 %    \end{align*} Then the evaluation map $(ev_i,ev_{\vartheta})$
 %    can be made weakly submersive and one can construct a reduced
 %    Kuranishi structures on $\Lambda^*(L\times S^1_{\vartheta})$.
%\end{rmk}
%\begin{rmk}
%  Notice that by the Gromov compactness theorem, the last term in
%  (\ref{61}) is finite.
%\end{rmk}

\subsection{Orientation}
Now we want to orient our moduli space $\mathcal{M}_{\gamma}(\mathfrak{X},L)$ in a coherent way. We first recall the following result on coherent orientation of moduli space of pseudo-holomorphic discs.
\begin{thm}\cite{FOOO} \label{2002}
Given an almost complex manifold $X$ and a totally real submanifold $L$, the moduli space of pseudo holomorphic discs with Lagrangian boundary condition $\mathcal{M}_{\gamma}(X,L)$, $\gamma\in H_2(X,L)$ can be oriented in a coherently provided $L$ is oriented and relatively spin. 
\end{thm}
Let $M^0$ denote the orientation for $M$. The orientation of the moduli space $\mathcal{M}_{\gamma}(\mathfrak{X},L)$ is given by 
  \begin{align}\label{2005}
  \mathcal{M}_{\gamma}(\mathfrak{X},L)^0=(S^1_{\vartheta})^0\times \mathcal{M}_{\gamma}(X_{\vartheta},L)^0.
  \end{align} Therefore, it suffices to orient the later two terms in (\ref{2005}). To apply Theorem \ref{2002} to the pair $(X,L)$, it suffices to define the orientation and a relative spin structure for $\mathcal{L}$. The torus $L$ has a natural orientation since it is a special Lagrangian in $X_{\vartheta}$, $\vartheta=\mbox{Arg}Z_{\gamma}$. The equator $S^1_{\vartheta}\in \mathbb{P}\backslash \{0\}$ is also oriented counterclockwise. Since the tangent bundle of torus is trivial, one can choose the
trivial spin structure for $\mathcal{L}$ and the moduli space $\mathcal{M}_{\gamma}(\mathfrak{X},L)$ is oriented follows from Theorem \ref{2002}.

%\begin{rmk}
%  Each initial ray corresponding to the singularity has same
%  orientation.
%\end{rmk}

\subsection{Reduced $A_{\infty}$ Structure and Floer Theoretic Counting}
From the Kuranishi structure constructed in section \ref{998},
one can define a filtered $A_{\infty}$ structure using De Rham model developed in \cite{F1}\cite{FOOO}
as follows: Let $\rho_k \in \Lambda^*(L\times S^1_{\vartheta})$ be
differential forms on $L\times S^1_{\vartheta}$. For each $k\geq 1$,
we define
\begin{equation*}
m_{k,\gamma}(\rho_1,\cdots,\rho_k)\in \Lambda^*(L\times
S^1_{\vartheta}), \hspace{5mm}
*=\sum_{i=1}^k (deg(\rho_i)-1)+ 1
\end{equation*}
\begin{align*}
  m_{k,\gamma}(\rho_1,\cdots,\rho_k)=Corr_*(\mathcal{M}_{k+1,\gamma}(\mathfrak{X},L); (ev_1, \dots,
  ev_k,ev_{\vartheta}), (ev_0,ev_{\vartheta}))\\(\rho_1 \times \cdots \times
  \rho_k),
\end{align*}
\begin{equation}
 m_{0,\gamma}(1)=Corr_*(\mathcal{M}_{1,\gamma}(\mathfrak{X},L); tri,
 (ev_0,ev_{\vartheta}))(1) \in \Lambda^2(L\times S^1_{\vartheta}).
\end{equation}
For each $\vartheta\in S^1_{\vartheta}$, we define
\begin{equation}
   m_k^{\vartheta}=\sum_{\substack{\gamma \in \pi_2(X,L),\\ \mbox{Arg}Z_{\gamma}=\vartheta} }m_{k,\gamma}T^{Z_{\gamma}}.
\end{equation}
Finally we define the open Gromov-Witten invariants, which we expect to be a symplectic analogue of $\tilde{\Omega}^{trop}$. We will discuss more about its property in the next section.
\begin{definition}
Given an elliptic K3 surface $X\rightarrow B$ with $L_u$ be the torus fibre over $u\in B$ and $\gamma\in H_2(X,L_u)$ such that $\partial \gamma\neq 0\in H_1(L_u)$, we
define
\begin{equation*}
  \tilde{\Omega}^{Floer}(\gamma;u)=\tilde{\Omega}^{Floer}(\gamma;L_u)=Corr_*(\mathcal{M}_{0,\gamma}(\mathfrak{X},L_u);tri, tri)(1)\in
  \mathbb{R}.
\end{equation*}
\end{definition}

\begin{thm} \label{4}Given $E_0>0$, the operators $\{m_{k,\gamma}^{\vartheta}\}_{k\geq 0}$ form an $S^1$-family of  acyclic filtered $A_{\infty}$
algebra structures modulo $T^{E_0}$ on $\Lambda^*(L\times
S^1_{\vartheta})$ with $1$ as a strict unit. Moreover, the structure
is independent of the choices of Kuranishi structures and
multi-sections chosen, up to pseudo-isotopy of inhomogeneous cyclic
filtered $A_{\infty}$ algebras.
\end{thm}
\begin{proof}
This is a standard argument follows \cite{FOOO}. However, the new
Kuranishi structure induces a new $A_{\infty}$ structure and can be
viewed as a new symplectic invariant. We first will construct an
inhomogeneous cyclic filtered $A_{\infty}$ algebra structure modulo
$T^{E_0}$. From the Kuranishi structure and multi-section
constructed in Theorem \ref{1}, we first prove the $A_{\infty}$
relation modulo $T^E$. The $A_{\infty}$ relation is equivalent to
\begin{align*}
  \sum_{\gamma_1+\gamma_2=\gamma}\sum_{k_1+k_2=k+1}\sum_i
  &(-1)^{deg(\rho_1)+\cdots+\deg(\rho_{i-1})+i-1}\\
  & m_{k_1,\gamma_1}^{\vartheta}(\rho_1,\cdots,m_{k_2,\gamma_2}^{\vartheta}(\rho_i,\cdots,\rho_{i+k_2-1}),\cdots,\rho_k)=0
\end{align*}
We may write the sum from left hand side into
\begin{align*}
  m_{1,0}^{\vartheta}m_{k,\gamma}^{\vartheta}&(\rho_1,\cdots,\rho_k)  \\
  +\sum_i(-1)^{deg(\rho_1)+\cdots+\deg(\rho_{i-1})+i-1}m_{k,\gamma}^{\vartheta}(\rho_1,\cdots,m_{1,0}^{\vartheta}(\rho_i),\cdots,\rho_k)\\
  +\sum_{\substack{\gamma_1+\gamma_2=\gamma,k_1+k_2=k+1 \\ \gamma_1 \neq 0 \text{ or } k_1\neq 1 \\ \gamma_2 \neq 0 \text{ or } k_2 \neq 1}}\sum_i
  (-1)^{deg(\rho_1)+\cdots+\deg(\rho_{i-1})+i-1}\\
   m_{k_1,\gamma_1}^{\vartheta}(\rho_1,\cdots,m_{k_2,\gamma_2}^{\vartheta}(\rho_i,\cdots,\rho_{i+k_2-1}),\cdots,\rho_k)
\end{align*}
Then the $A_{\infty}$ comes from the applying the Stoke's theorem \ref{2} and composition law \ref{3} to (\ref{12}).

Here the signs involved depend on the orientation of the relevant moduli spaces. Recall that $\mathcal{X}$ is the twistor space of $X$ with $X_0,X_{\infty}$ deleted. Let $\mathcal{L}$ be the $3$-torus which is the product of the equator $S^1_{\vartheta}\subseteq \mathbb{P}^1$ and the Lagrangian torus $L$. The complex manifold $\mathcal{X}$ might not be symplectic and $\mathcal{L}$ may not be a Lagrangian in $\mathcal{X}$. However, $\mathcal{L}$ is still a totally real $3$-torus. One can study holomorphic discs in $\mathcal{X}$ with boundaries on $\mathcal{L}$ and the corresponding linearized Cauchy-Riemann operators are Fredholm. From the maximal principle. every holomorphic discs in $\mathcal{X}$ with boundary on $\mathcal{L}$ will factor through the fibre of $\mathcal{X}\rightarrow \mathbb{P}^1$, thus gives rise to an element in $\mathcal{M}_{\gamma}(\mathfrak{X},L)$ for some $\gamma\in H_2(X,L)\subseteq H_2(\mathcal{X},\mathcal{L})$. Notice that the equator $S^1_{\vartheta}\in \mathbb{P}^1\backslash \{\infty\}$ admits an natural orientation and thus $\mathcal{L}$ is orientated. If we choose the trivial spin structure on the $3$-torus $\mathcal{L}$, the orientation of $\mathcal{M}_{\gamma}(\mathcal{X},\mathcal{L})$ from Theorem \ref{2002} will be the same as the orientation of $\mathcal{M}_{\gamma}(\mathfrak{X},L)$. For $0<E_0<E_1$, one can extend the $A_{\infty}$ relation from
modulo $T^{E_0}$ to modulo $T^{E_1}$ by Theorem 8.1 \cite{F1}.

The cyclic symmetry follows from the cyclic symmetry of the
perturbed multi-sections:
    \begin{align*}
      \langle m_{k+1,\gamma}&(\rho_1,\cdots,\rho_k),\rho_0 \rangle \\
           =&Corr_*(\mathcal{M}_{k+1,\gamma}(\mathfrak{X},L);(ev_1,\cdots,ev_k,ev_0),
           tri)(\rho_1\times \cdots \rho_k \times \rho_0)
    \end{align*}

To prove $1$ is a strict unit: for $\gamma \neq 0$ and $k\leq 1$, we
want to claim
   \begin{equation*}\label{5}
     m_{k,\gamma}(\rho_1,\cdots,\rho_{i-1},1,\rho_{i+1},\cdots,\rho_k)=0
   \end{equation*}
 This is because of the compatibility of forgetful map
 $\mathfrak{forget}:\mathcal{M}_k(\mathfrak{X},L) \rightarrow \mathcal{M}_{k-1}(\mathfrak{X},L)$
Let $V$ be the vector tangent to the fibre of the forgetful map
$\mathfrak{forget}$, then
      \begin{equation*}
        \iota_V ((f^s)_{\alpha}^*\rho \wedge \omega)=0,
      \end{equation*}where $\rho=\rho_1\times \cdots \times
      \rho_{i-1}\times 1\times \rho_{i+1}\times \cdots \rho_k$.
      Therefore each components of right hand side of \ref{5} vanish and $1$ is a strict unit.

\begin{rmk}
   Here we need to include the factor $S^1_{\vartheta}$ to make
   $m_k$ degree $1$ operators. Otherwise, $m_k$ will be degree $0$ and do not form $A_{\infty}$-structure.
\end{rmk}

\begin{prop}
  The $A_{\infty}$-structure $(\Lambda^*(L\times
  S^1_{\vartheta}),\{m_{k,\gamma}\})$ constructed in Proposition
  \ref{4} is independent of choice of the Kuranishi structure and the
  family of multi-sections chosen, up to pseudo-isotopy of
  inhomogeneous cyclic filtered $A_{\infty}$ algebra modulo $T^{E_0}$.
\end{prop}
\begin{proof}
  The proof is similar to proposition 4.1 \cite{F2}
\end{proof}
Finally we apply lemma 4.2 \cite{F2} to extend the inhomogeneous
cyclic filtered $A_{\infty}$ algebra structure and finish the proof
of Theorem \ref{4}.
\end{proof}

%\begin{rmk} \label{312}
%  Let $\mathcal{X}$ be the twistor space of $X$ with two fibres admit elliptic %fibrations discarded. From the expression (\ref{38}),
%  $\Omega(\zeta)\wedge \frac{d\zeta}{\zeta}$ is a nowhere vanishing holomorphic %$(3,0)$-form on $\mathcal{X}$. Notice that every holomorphic disc
%  $(D^2,\partial D^2)\rightarrow (\mathcal{X},L\times S^1_{\vartheta})$ factor %through a fibre by maximal principle. Theorem \ref{4} is equivalent to the
%  Kuranishi structure of the real $3$-torus $L\times S^1_{\vartheta}$ in %$\mathcal{X}$ constructed in \cite{FOOO}.
%\end{rmk}

%\begin{rmk}
%  One can generalized theorem to the case includes interior marked
%  points and study open-closed duality.
%\end{rmk}

\subsection{Open Gromov-Witten Invariants}
From now on, we will assume $X$ is an elliptic K3 surface and
$L=L_u$ is a torus fibre for our purpose. After hyperK\"ahler
rotation, $X$ admits an $S^1_{\vartheta}$-family of special
Lagrangian torus fibration. Similar to the situation in tropical case discussed in Section \ref{2008}, the $S^1$-family of complex structures in Section \ref{2009} also motivates the concept of wall of marginal stability in symplectic geometry. However, the similar notion is more complicated.

\begin{definition}
   Given local section $\gamma$ of $\bigcup_u H_2(X,L_u)$, we define locally
     \begin{align*}
     W''_{\gamma_1,\gamma_2}=\{u\in
     B_0|\mbox{Arg}Z_{\gamma_1}=\mbox{Arg}Z_{\gamma_2},
     Z_{\gamma_1}Z_{\gamma_2}\neq 0\, \mbox{ and }\gamma_1,\gamma_2 \mbox{ are not
     colinear}\},
      \end{align*} and
     $W''_{\gamma}=\bigcup_{\gamma=\gamma_1+\gamma_2}W''_{\gamma_1,\gamma_2}$.
\end{definition}
Because the central charge $Z_{\gamma}$ is holomorphic, each
$W''_{\gamma_1,\gamma_2}$ forms a real analytic Zariski closed
subset of real codimension one on $B_0$ if $\partial\gamma_1$ and $\partial \gamma_2$ are not parallel. Indeed, from Cauchy-Riemann equation, the singular points $z_0$ of $W_{\gamma_1,\gamma_2}''$ satisfies $d\mbox{Arg}\frac{Z_{\gamma_1}}{Z_{\gamma_2}}(z_0)=0$. If $z_0\in B_0$, then there are gradient flow line of $|Z_{\gamma_i}|^2$, $i=1,2$ intersecting transversally at $z_0$. Let $v$ be the tangent of the gradient flow line, then $d\mbox{Arg}\frac{Z_{\gamma_1}}{Z_{\gamma_2}}(v)\neq 0$ at $z_0$. If $z_0\in \Delta$, then one can find a coordinate $z$ around $z_0$ such that $z(z_0)=0$. In this coordinate, we may assume the central charges are of the form 
\begin{align*}
Z_{\gamma_1}&=C_1z \\
Z_{\gamma_2}&=C_2(\frac{1}{2\pi i}z\log{z}+C'),
\end{align*} where $C_i\in \mathbb{Z}$ and $C'\in \mathbb{C}^*$. Direct computation shows that $W_{\gamma_1,\gamma_2}$ is two smooth curves glue smoothly at $z_0$.

Notice that the preimage of
$\mathbb{R}e^{i\vartheta}$ is a sub $\mathbb{Z}$-module and thus a
sublattice. In particular, the union in the definition of
$W''_{\gamma}$ is finite. Therefore, $W''_{\gamma}$ is locally a
real analytic Zariski closed subset of the base unless $\gamma$ is not primitive. If a
relative class $\gamma\in H_2(X,L)$ can be represented as a
holomorphic cycle, the phase of central charge $\int_{\gamma}\Omega$
will indicate which complex structure $J_{\vartheta}$ makes $\gamma$
holomorphic. Indeed, assume $\gamma_u \in H_2(\underline{X},L)$ can be realized as the image of a holomorphic map in $X_{\vartheta}$. Then 
  \begin{align*}
     Z_{\gamma}(u)=e^{i\vartheta}\int_{\gamma}e^{-i\vartheta}\Omega=e^{i\vartheta}i \int_{\gamma}\mbox{Im}(e^{-i\vartheta}\Omega).
  \end{align*} The last identity is because $\int_{\gamma}\Omega_{\vartheta}=0$.
In particular, we have $\mbox{Arg}Z_{\gamma}(u)=\vartheta+\pi/2$. In other words, the phase of central charge $\mbox{Arg}Z_{\gamma}$ a priori determines the obstruction for the complex structure on the equator to support $\gamma$ as a holomorphic cycle. When the relative class $\gamma$ supports a holomorphic cycle, the central charge is nothing but
    \begin{align}\label{736}
     Z_{\gamma}= \mbox{(symplectic area) }e^{i(\vartheta+\pi/2)}
    \end{align} from the Proposition \ref{49}.
In particular, the moduli space for family
$\mathcal{M}_{k,\gamma}(\mathfrak{X},L)$ has the same underlying
space as the usual moduli space of holomorphic discs
$\mathcal{M}_{k,\gamma}(X_{\vartheta},L)$, where  $\vartheta=\mbox{Arg}Z_{\gamma}-\pi/2$. We might drop the subindex $\vartheta$ when the target is clear. However, the tangent-obstruction theory (or the Kuranishi structures) on $\mathcal{M}_{\gamma}(\mathfrak{X},L)$ and $\mathcal{M}_{\gamma}(X_{\vartheta},L)$ are different. This motivates the following motivation:
\begin{definition}
  Given local section $\gamma$ of $\Gamma=\bigcup_u H_2(X,L_u)$, we define locally
  \begin{align} \label{230}
    W'_{\gamma}&=\bigcup_{\gamma_1+\gamma_2=\gamma}W'_{\gamma_1,\gamma_2}
    \notag
    \\
     &=\{u \in B \big| \gamma=\gamma_1+\gamma_2 \text{, where $\gamma_1$ and $\gamma_2$ are represented by holomorphic }  \notag \\
     & \qquad \qquad  \text{   discs with boundary on $L_u$ in $X_{\vartheta}$, for some $\vartheta$.} \}
  \end{align}
\end{definition}
  Notice that by Gromov compactness theorem, %(Corollary 5.5 \cite{F4}) % 
  we have $W'_{\gamma_1,\gamma_2}\subseteq
  W''_{\gamma_1,\gamma_2}$ as a closed subset in standard topology on $B_0$ and the expression in (\ref{230}) is a finite union.
However, $W'_{\gamma}$ might not be real codimension one because of
appearance of holomorphic discs with respect to non-generic (almost)
complex structures. Also, $W'_{\gamma_1,\gamma_2}$ might depend on
the choice of Ricci-flat $\omega$ while the real codimension one
$W''_{\gamma_1,\gamma_2}$ are not. 

Now assume that the moduli space $\mathcal{M}_{\gamma_u}(\mathfrak{X},L_u)$ has  non-empty real codimension one boundary for some $\gamma_u\in H_2(X,L_u)$. Namely, we have 
  \begin{align*}
  \gamma_u=\gamma_{1,u}+\gamma_{2,u} \in H_2(X,L_u)
  \end{align*} and $\mathcal{M}_{\gamma_{i,u}}(X_{\vartheta},L_u)$ are non-empty, for $i=1,2$ and $\vartheta\in S^1$. In particular, we have $Z_{\gamma_1}(u)Z_{\gamma_2}(u)\neq 0$ from (\ref{736}) and 
     \begin{align*}
      \mbox{Arg}Z_{\gamma}(u)=\mbox{Arg}Z_{\gamma_1}(u)=\mbox{Arg}_{\gamma_2}(u)=\vartheta+\pi/2.
     \end{align*}
     The interesting implication is that we may not always have bubbling phenomenon of the moduli space $\mathcal{M}_{\gamma}(\mathfrak{X},L_u)$ unless the torus fibre $L_u$ sits over the locus characterized by
   \begin{align} \label{700}
    \mbox{Arg}Z_{\gamma_1}=\mbox{Arg}Z_{\gamma_2}.
   \end{align} 
Assume that $Z_{\gamma_1}$ is not a multiple of $Z_{\gamma_2}$. Since the central charges are holomorphic functions, the equation (\ref{700}) locally is pluriharmonic and defines a real analytic pseudoconvex hypersurface. In particular, the mean value property of (pluri)harmonic functions implies that locally this hypersurface divides the base into chambers. If $Z_{\gamma_1}=kZ_{\gamma_2}$, then $Z_{\gamma_1-k\gamma_2}=0$. In particular, $dZ_{\gamma_1-k\gamma_2}=0$ together with Lemma \ref{34} implies 
  \begin{align*}
\partial \gamma_1-k\partial \gamma_2=0\in H_1(L,\mathbb{Z})\cong \mathbb{Z}^2
  \end{align*}
 Thus, from the exact sequence (\ref{1002}) there exists positive integers $k_1=kk_2,k_2$, and $\partial\gamma'\in H_1(L,\mathbb{Z})$, such that we have
 \begin{align*}
 \partial\gamma_i=k_i\partial \gamma' \in H_1(L,\mathbb{Z})), \mbox{ $i=1,2$}.
 \end{align*} and 
 \begin{align*}
     \partial\gamma=\partial\gamma_1+\partial\gamma_2=(k_1+k_2)\partial\gamma'
 \end{align*}is not primitive. To sum up, we have proved the following theorem:
   
  \begin{thm}\label{811} Assume that
      \begin{enumerate}
            \item The relative class $\gamma$, with $\partial \gamma \neq 0\footnote{For the case $\partial \gamma=0$, one also has to consider the situation when a rational curve with a point on $L$ appears as real codimension one boundary of the moduli space and thus the invariant can depend on the choice of K\"ahler form of the elliptic fibration. We will leave this case in the forthcoming paper \cite{L7}.}\in H_1(L)$ being primitive.             
            \item The fibre $L_u$ does not sit over the closed subset $W'_{\gamma}$. Notice that there are only finitely many possible decompositions $\gamma=\gamma_1+\gamma_2$ such that $W'_{\gamma_1,\gamma_2}$ is non-empty for a fix choice of $[\omega]$ by Gromov compactness theorem.
          \end{enumerate}
     Then the moduli space $\mathcal{M}_{k,\gamma}(\mathfrak{X},L_u)$ has no boundary. 
  \end{thm}     

\begin{rmk}
Assuming the holomorphic structure of the K3 surface (preserving the elliptic fibration structure) is chosen generically, then the first assumption and be loosen to $\gamma\in H_2(X,L)$ is generic and the theorem still holds.
\end{rmk}

\begin{lem} \label{70}
   Under then same assumption in Theorem \ref{811}, then  $\tilde{\Omega}^{Floer}(\gamma;u)$ is
   well-defined.
\end{lem}
\begin{proof}
  Assume there are two different Kuranishi structures, and give rise to $\tilde{\Omega}^{Floer}(\gamma;u)$ and $\tilde{\Omega}'^{Floer}(\gamma;u)$ respectively. Consider the trivial family hpyerK\"ahler structures and apply Theorem \ref{64}.
    \begin{align}\label{102}
    \partial(&[0,1]\times \mathcal{M}_{0,\gamma}(\mathfrak{X},L_u))=(\{0,1\}\times
    \mathcal{M}_{0,\gamma}(\mathfrak{X},L_u)) \notag \\ 
                                                   &\cup
                                                   \bigcup_{\gamma_1+\gamma_2=\gamma}(([0,1]\times \mathcal{M}_{1,\gamma_1}(\mathfrak{X},L_u))\; {}_{(ev_0,ev_{\vartheta},ev_t)} \! \times_{(ev_0,ev_{\vartheta},ev_t)} ([0,1]\times
                                                   \mathcal{M}_{1,\gamma_2}(\mathfrak{X},L_u)))/\mathbb{Z}_2 \notag \\
    \end{align}
Applying the Stokes theorem (Proposition \ref{2}) to both sides of (\ref{102}), we get
   \begin{align*}
    0=&  \tilde{\Omega}^{Floer}(\gamma;u)-\tilde{\Omega}'^{Floer}(\gamma;u)  \\
     +& \mbox{Corr}_*(\bigcup_{\gamma_1+\gamma_2=\gamma}(([0,1]\times \mathcal{M}_{1,\gamma_1}(\mathfrak{X},L_u))\; {}_{(ev_0,ev_{\vartheta},ev_t)} \! \times_{(ev_0,ev_{\vartheta},ev_t)} ([0,1]\times
                                                        \mathcal{M}_{1,\gamma_2}(\mathfrak{X},L_u)))/\mathbb{Z}_2);tri,tri)(1),
   \end{align*}where the last term vanishes because the last term in (\ref{102}) is empty by assumption and the lemma follows.

\end{proof}

\begin{rmk}If we view $X$ as an elliptic fibred K3 surface, then naively we are counting special Lagrangian discs with boundaries on elliptic fibres. Holomorphic discs in $X_{\vartheta}$ corresponds to special Lagrangian discs with phase $\vartheta$ in $X$. It is conjectured that special Lagrangians are stable objects in the Fukaya category. Changing the special Lagrangian boundary conditions is expected to
 be mirror to changing the stability conditions of stability
 condition in Donaldson-Thomas theory. Since our central
 charge is constraint by $\langle dZ,dZ\rangle=0$, the affine
 base $B$ can be viewed as a complex isotropic submanifold of the
 corresponding stability manifold.
\end{rmk}

Now given any two points $u_0, u_1\in B_0$. We choose a $1$-parameter
family of fibration preserving diffeomorphisms $\phi_t$ such that
$\phi_t(L_0)=L_t$, for $t\in [0,1]$. By pulling back the K\"ahler
forms and complex structures to $L_0$, we get a family of hyperK\"ahler structures $(\underline{X},\phi_t^*\omega,\phi^*_t\Omega)$ such that $L=L_0$ is always an elliptic curve.

\begin{prop}\label{2012}
  If $\tilde{\Omega}^{Floer}(\gamma;u)$ is well-defined, then $\tilde{\Omega}^{Floer}(\gamma;u')$ is well-defined for nearby $u'$. Moreover, the invariant $\tilde{\Omega}^{Floer}(\gamma;u)$ is locally a constant. Namely,
    \begin{align*}
     \tilde{\Omega}^{Floer}(\gamma;u)=\tilde{\Omega}^{Floer}(\gamma;u'),
    \end{align*}for $u'$ in a neighborhood of $u$.
\end{prop}
\begin{proof}
  The union in (\ref{230}) is a finite union by Gromov's compactness theorem. Therefore, $u\notin W'_{\gamma}$ implies $u'\notin W'_{\gamma}$ for all nearby $u'$. In particular, there exists a path connecting $u$ and a nearby $u'$ such that it does not intersect $W'_{\gamma}$.  Applying Theorem \ref{64} to the $1$-parameter family $(\underline{X},\phi_t^*\omega,\phi_t^*\Omega)$, there exists Kuranishi structures and multi-sections with compatibility
    \begin{align}\label{2013}
   \partial(&\bigcup_{t\in [0,1]}\mathcal{M}_{0,\gamma}(\mathfrak{X}_t,L))=(\mathcal{M}_{0,\gamma}(\mathfrak{X}_1,L)-\mathcal{M}_{0,\gamma}(\mathfrak{X}_0,L)) \notag \\ 
                                                      &\cup
                                                      \bigcup_{\gamma_1+\gamma_2=\gamma}((\bigcup_{t\in [0,1]}\mathcal{M}_{0,\gamma_1}(\mathfrak{X}_t,L))\; {}_{(ev_0,ev_{\vartheta},ev_t)} \! \times_{(ev_0,ev_{\vartheta},ev_t)} (\bigcup_{t\in [0,1]}\mathcal{M}_{0,\gamma_2}(\mathfrak{X}_t,L))/\mathbb{Z}_2 \notag \\
    \end{align}
  Apply the Proposition \ref{2} to both sides of (\ref{2013}), we have
  \begin{align*}
   \tilde{\Omega}^{Floer}(\gamma;u)=\tilde{\Omega}^{Floer}(\gamma;L)=\tilde{\Omega}^{Floer}({\phi_1}_*\gamma;\phi_1(L))=\tilde{\Omega}^{Floer}(\gamma;u')
  \end{align*}

\end{proof}

In the definition of the invariant $\tilde{\Omega}^{Floer}(\gamma;u)$, there
is a choice of the Ricci-flat metric and the holomorphic symplectic $2$-form to form the twistor family.
However, we have the following
\begin{cor} \label{216}
Let $(\omega,\Omega),(\omega',\Omega')$ be hyperK\"ahler triples with $\Omega$ and $\Omega'$ are holomorphic symplectic $2$-forms of the fixed elliptic K3 surface $X$. Assume that the
corresponding invariants $\tilde{\Omega}(\gamma;u)$ and
$\tilde{\Omega}'(\gamma;u)$ are well-defined. Then
    \begin{align*}
       \tilde{\Omega}(\gamma;u)=\tilde{\Omega}'(\gamma;u)
    \end{align*}
\end{cor}
\begin{proof}
   Notice that $W''_{\gamma}$ is independent of the choices of $(\omega,\Omega)$ or $(\omega',\Omega')$.
   First we assume that $u\notin W''_{\gamma}$ then the corollary follows by the similar cobordism argument as in Lemma
   \ref{70}. In particular, the corollary holds on an real
   analytic Za1riski open subset of $B_0$. For $u\in
   W''_{\gamma}\backslash W'_{\gamma}$, then by Theorem \ref{2012}
   there exists $u'\notin W''_{\gamma}$ near $u$ such that
     \begin{align*}
        \tilde{\Omega}(\gamma;u)=\tilde{\Omega}(\gamma;u')=\tilde{\Omega}'(\gamma;u')=\tilde{\Omega}'(\gamma;u).
     \end{align*}
\end{proof}

%\begin{prop}
%  The invariants $\Omega(\gamma;u)$ independent of the choice
%  of the holomorphic $(2,0)$-form of the elliptic fibration. 
%\end{prop}
%\begin{proof}
  %If we change the elliptic fibration to its conjugate complex strucuture, then
  %both $S^1_{\vartheta}$ and $L$ change orientation. All the other ingredients %involve orientation remain the
  %same. This indicates the first part of the proposition. For the
  %later part, we first look at the case when the holomorphic $(2,0)$-form %$\Omega$ replaced by $-\Omega$. In this case, the complex structure of the %special Lagrangian fibration
  %changed to its conjugate. Thus the The sign may
  %change on $H^0(C,F)$ and $H^1(C,F)$ but cancel out by Riemann-Roch
  %formula. Notice that the orientation of $S^1_{\vartheta}$ is
  %unchanged because both $\alpha=\mbox{Im}\Omega$ and $j$ change sign and %$K_J(f,\alpha)=\frac{1}{2}K_{\alpha}(df\circ
  %j)$ is unchanged. 
%  Any two choices of holomorphic
%  volume form will give rise to pseudo-isotopy of Kuranishi
%  structures and the proof is the same as Lemma \ref{70}.
%\end{proof}
%\begin{rmk}
%  Considering $S^1$-family not only raises the virtual dimension of
%  moduli space to $0$ but after projecting to a fixed phase still
%  get the reduced counting.
%\end{rmk}

Next we want to study the situation when given two points $u_0, u_1\in B_0$ and the path connecting them intersect with $W'_{\gamma}$. We can still use the same argument in Proposition \ref{2012} to compute the difference of the invariant
 $\tilde{\Omega}^{Floer}(\gamma;u_1)-\tilde{\Omega}^{Floer}(\gamma;u_0)$. However, the last term in (\ref{2013}) may have non-trivial contribution to the difference and the invariant may jump. Therefore, we want to compute 
   \begin{align}\label{2015}
   \mbox{Corr}_*(\bigcup_{\gamma_1+\gamma_2=\gamma}(([0,1]\times \mathcal{M}_{1,\gamma_1}(\mathfrak{X},L_u))\; {}_{(ev_0,ev_{\vartheta},ev_t)} \! \times_{(ev_0,ev_{\vartheta},ev_t)} ([0,1]\times
                                                           \mathcal{M}_{1,\gamma_2}(\mathfrak{X},L_u)))/\mathbb{Z}_2);tri,tri)(1).
   \end{align}
First, we notice the following lemma:
 \begin{lem} \label{50}
     Let $X_i\cong B_i\times S^1$, where $B_i$ are manifolds with a compact support volume form $\omega_i$. Assume $f_i:X_i \rightarrow T^2$ are submersions
     such that $[f_1(\{pt\} \times S^1)].[f_2(\{pt\}\times S^1)]=k$, then $\int_X \omega_1
     \wedge \omega_2=k(\int_{B_1}\omega_1)(\int_{B_2}\omega_2)$, where $X=X_1\times_{T^2} X_2$.
  \end{lem}
  \begin{proof}
    Consider the map $\pi:X\rightarrow B_1\times B_2$. $X$
    defined by $f_1(x)=f_2(y)$, is closed and $\pi$ is proper. Easy
    computation shows that
    \begin{align*}
    \int_X\omega_1\wedge \omega_2=(\mbox{deg}\pi)\cdot\int_{B_1\times
    B_2}\omega_1 \wedge \omega_2=\pm k \cdot
    (\int_{B_1}\omega_1)(\int_{B_2}\omega_2)
    \end{align*}
  \end{proof}
The compatibility of forgetful maps says that if $V_{\alpha_k}$ are local Kuranishi chart for
  $\mathcal{M}_{0,\gamma_k}$, for $k=1,2$, then one can take
  $V'_{\alpha_k}=V_{\alpha}\times S^1$ as Kuranishi chart of
  $\mathcal{M}_{1,\gamma_k}$. Here $\mathcal{M}_{k,\gamma}$ can be either $\mathcal{M}_{k,\gamma}(\mathfrak{X},L)$ or $\bigcup_{t\in [0,1]}\mathcal{M}_{k,\gamma}(\mathfrak{X}_t,L)$. Thus locally we can choose
  $\mathfrak{s}_{\alpha_k,i,j}^{-1}(0)$ to be $X_i$ in above lemma. $\int_{B_k}\omega_k$ is the local contribution
  to $m_{-1,\gamma_k}$ and $\int_X\omega_1\wedge \omega_2$ corresponding local contribution for left
  hand side of (\ref{27}). Notice that the factor $k=\pm\langle \gamma_1,\gamma_2\rangle$ implies that (2015) has no contribution if $\langle\gamma_1,\gamma_2\rangle=-0$. Therefore, the jumping phenomenon for the invariants $\tilde{\Omega}^{Floer}(\gamma;u)$ for charges $\partial \gamma=0$ and $\partial \gamma\neq 0$ decouple, which is parallel to Proposition \ref{1049} in tropical geometry.
    This also coincide with that fact that in the expansion formula of symplectic holomorphic $2$-form constructed in (\cite{GMN}), and the scattering diagram constructed in Theorem \ref{13}
    are decoupled from those $\gamma$ with $[\partial \gamma]=0$. Moreover,
  \begin{prop}
  The invariant $\tilde{\Omega}^{Floer}(\gamma;u)$ is well-defined if $\gamma$ can not be decomposed to $\gamma=\gamma_1+\gamma_2$ such that $\gamma_1,\gamma_2$ can both be represented as holomorphic cycles in $X_{\vartheta}$, for some $\vartheta$ and $\partial \gamma_1,\partial \gamma_2$ are not parallel.
  In particular, the multiple cover contribution is well-defined.
  \end{prop}

 Therefore, we may just consider the case when $\langle\partial \gamma_1,\partial \gamma_2\rangle \neq 0$. Assume the path connecting $u_0, u_1$ only intersect with $W''_{\gamma_1,\gamma_2}$ transversally at a point $u_{t_0}$. Then from Theorem \ref{64}, one can replace (\ref{2015}) by 
   \begin{align} \label{2018}
                       \mbox{Corr}_*(\mathcal{M}_{1,\gamma_1}(\mathfrak{X}_{t_0},L)\; {}_{(ev_0,ev_{\vartheta})} \! \times_{(ev_0,ev_{\vartheta})}                          \mathcal{M}_{1,\gamma_2}(\mathfrak{X}_{t_0},L)\ ;tri,tri)(1)
   \end{align}
Again using the compatibility of forgetful map, we have the following propositions:
\begin{thm}\label{32}Assume $\tilde{\Omega}(\gamma_i;u_{t_0})$ are defined and $\gamma_i$ are primitive. Then
  \begin{align}
      (\ref{2018})=\pm\langle \gamma_1,\gamma_2 \rangle \tilde{\Omega}(\gamma_1;u_{t_0})\tilde{\Omega}(\gamma_2;u_{t_0}).     
  \end{align}
\end{thm}

%\begin{proof}
%    It suffices to prove the following statement:
%     \begin{lem}
%       Let $X_i=B_i\times S^1$, $i=1,2$, where $B_i$ are manifolds
%       with differential forms $\omega_i$ of
%       $\mbox{deg}(\omega_i)=\mbox{dim}B_i-1$. Assume
%       $f_i:X_i\rightarrow T^2\times S^1\times \mathds{R}_t$
%       are submersions and $[f_1(\{pt\}\times S^1)]\cdot [f_2(\{pt\}\times S^1)]=0$. Let $X=X_1\times_{(f_1,f_2)} X_2$, then $\int_X \omega_1\wedge
%       \omega_2=0$.
%     \end{lem}
%\end{proof}

\begin{prop} \label{310}
Let $\{L_t\}$ be a $1$-parameter family of torus fibres such that
passing through $W''_{\gamma}$ exactly once and transversally. Assume
$\gamma=\sum_{i}k_i\gamma_i$, where $k_i\in \mathbb{N}$ and $\gamma_i$ primitive. If
  \begin{align*}
 \tilde{\Omega}^{Floer}&(k\gamma_{i_0})=0,\mbox{ for all $k\leq k_i$ and some $i_0$, then}\\
& \Delta\tilde{\Omega}^{Floer}(\gamma)=0
  \end{align*}
\end{prop}

%\begin{rmk}
%   Assume $T$ and $T'$ are two decorated trees with
%   $\mathcal{M}_{T}(\mathfrak{X},L)\subseteq
%   \mathcal{M}_{T'}(\mathfrak{X},L)$. The way we construct the
%   Kuranishi structure, the zero locus of multi-section of two
%   boundary components are disjoint by dimension counting. Indeed,
%   let $\bold{p}\in \mathcal{M}_{1,\alpha}\times_L
%   \mathcal{M}_{1,\gamma}$, follows the notation of Kuranishi structure in the appendix, we have
%      \begin{align*}
%         \mbox{dim}\mathfrak{s}_{\alpha}^{-1}(0)=\mbox{dim}W_{\alpha}=\mbox{dim}E_{\alpha}
%      \end{align*}
%     Therefore, the zero locus of perturbed multi-section on
%     $\mathcal{M}_{1,\alpha}\times_L \mathcal{M}_{1,\gamma}$ has
%     dimension $\mbox{dim}E_{\alpha}+\mbox{dim}E_{\gamma}$. On the
%     other hand, the zero locus of perturbed multi-section in a Kuranishi chart of $\bold{p}\in
%     \mathcal{M}_{0,\alpha+\gamma}$ has same dimension because of we
%     identify the obstruction bundle by parallel transport, see
%     (\ref{101}). If they are not disjoint, it will contradicts with
%     the dimension count.
%\end{rmk}

Therefore, Proposition \ref{310} motivates us to define the wall of
marginal stability in symplectic geometry:
\begin{definition}For $\gamma_1,\gamma_2$ primitive relative classes and $k,l\in \mathbb{N}$, we define
  \begin{align*}
      W_{k\gamma_1,l\gamma_2}=\{u\in B_0| \mbox{Arg}Z_{\gamma_1}=\mbox{Arg}Z_{\gamma_2}, \tilde{\Omega}^{Floer}(k'\gamma_1)
      \tilde{\Omega}^{Floer}(l'\gamma_2)\neq 0 \\
           \mbox{ for some $k'\leq k,l'\leq l$}
      \}.
  \end{align*}
  The wall of marginal stability of $\gamma$ for holomorphic discs counting is
  given by
  \begin{align*}
    W_{\gamma}&=\bigcup_{\substack{\gamma=k\gamma_1+l\gamma_2,\\ \gamma_1,\gamma_2 \mbox{ primitive}}}W_{k\gamma_1,l\gamma_2}
  \end{align*}
\end{definition}It is easy to see that $W_{\gamma}$ is an
real dimension one open subset of $W'_{\gamma}$. Finally, the invariants $\tilde{\Omega}^{Floer}$ also satisfies the reality condition, which is an analogue of Proposition \ref{2010} in tropical geometry.
\begin{prop}If $\tilde{\Omega}^{Floer}(\gamma;u)$ is well-defined,
then $\tilde{\Omega}^{Floer}(-\gamma;u)$ is also well-defined.
Moreover, we have reality condition
  \begin{align*}
   \tilde{\Omega}^{Floer}(\gamma;u)=\tilde{\Omega}^{Floer}(-\gamma;u)
  \end{align*}
\end{prop}
\begin{proof}
  Notice that the two relevant moduli spaces $\mathcal{M}_{\gamma}(\mathfrak{X},L)$ and $\mathcal{M}_{-\gamma}(\mathfrak{X},L)$ are the same as topological spaces and same Kuranishi structures. Thus, if one invariant is well-defined so is the other one. It suffices to check the orientation of the two moduli spaces are the same to make sure the signs are the same. The complex structure on $X_{\vartheta}$ and $X_{-\vartheta}$ are complex conjugate to each other. Therefore, the usual moduli spaces $\mathcal{M}_{\gamma}(X_{\vartheta},L)$ and $\mathcal{M}_{-\gamma}(X_{-\vartheta},L)$ have the same orientation, where $\vartheta=\mbox{Arg}Z_{\gamma}$. Also, the orientation of $S^1_{\vartheta}$ for $\mathcal{M}_{\gamma}(\mathfrak{X},L)$ and $\mathcal{M}_{-\gamma}(\mathfrak{X},L)$ are the same and the proposition follows. 
\end{proof}

\subsection{Disc Contribution from Local} \label{1001}
\subsubsection{Local Model: the Ooguri-Vafa Space} \label{801}
All the definitions and arguments in the previous section apply to any hyperK\"ahler surfaces (not necessarily compact) whenever the Gromov compactness holds.  The Ooguri-Vafa space is an elliptic fibration over a unit disc such that the only singular fibre is a nodal curve (or an $I_1$-type singluar fibre) over the origin. The singular central fibre breaks the $T^2$-symmetry into only $S^1$-symmetry. Using Gibbons-Hawkings ansatz, Ooguri and Vafa\cite{OV} wrote down Ricci-flat metrics with a $S^1$ symmetry and thus the Ooguri-Vafa space is hyperK\"ahler. Conversely, all Ricci-flat metrics with an $S^1$-symmetry are in this form. 
From the discussion earlier, there exists an $S^1$-family of integral affine structures on $B_0$. The central fibre is of $I_1$-type implies that the monodromy of affine structure (see the section 6.1) around the singularity is conjugate to $\bigl(
\begin{smallmatrix}
  1 & 1\\
  0 & 1
\end{smallmatrix} \bigr)$. Thus there exists an unique affine line $l_{\vartheta}$ passing through the singularity in the monodromy invariant direction.
      
Following the maximal principle trick in \cite{A}\cite{C}, if we fix $\vartheta$, the only simple holomorphic disc is in the relative class of Lefschetz thimble $\gamma_e$ and with its boundary on torus over $l_{\vartheta}$. Topologically, this holomorphic disc is the union of vanishing cycles over $l_{\vartheta}$. This reflects the fact the standard moduli space of holomorphic discs has virtual dimension minus one and thus generic torus fibre would not bound any holomorphic discs. However, when the $\vartheta$ goes around $S^1$, the affine line $l_{\vartheta}$ will rotate $2\pi$ and every point on the base will be exactly swept once by $l_{\vartheta}$. In other words, every torus fibre bounds a unique simple holomorphic disc (up to orientation) in Ooguri-Vafa space but with respect to a different complex structure. See Figure 1 below. 
\begin{figure}[htb]
\begin{center}
\includegraphics[height=3in,width=6in]{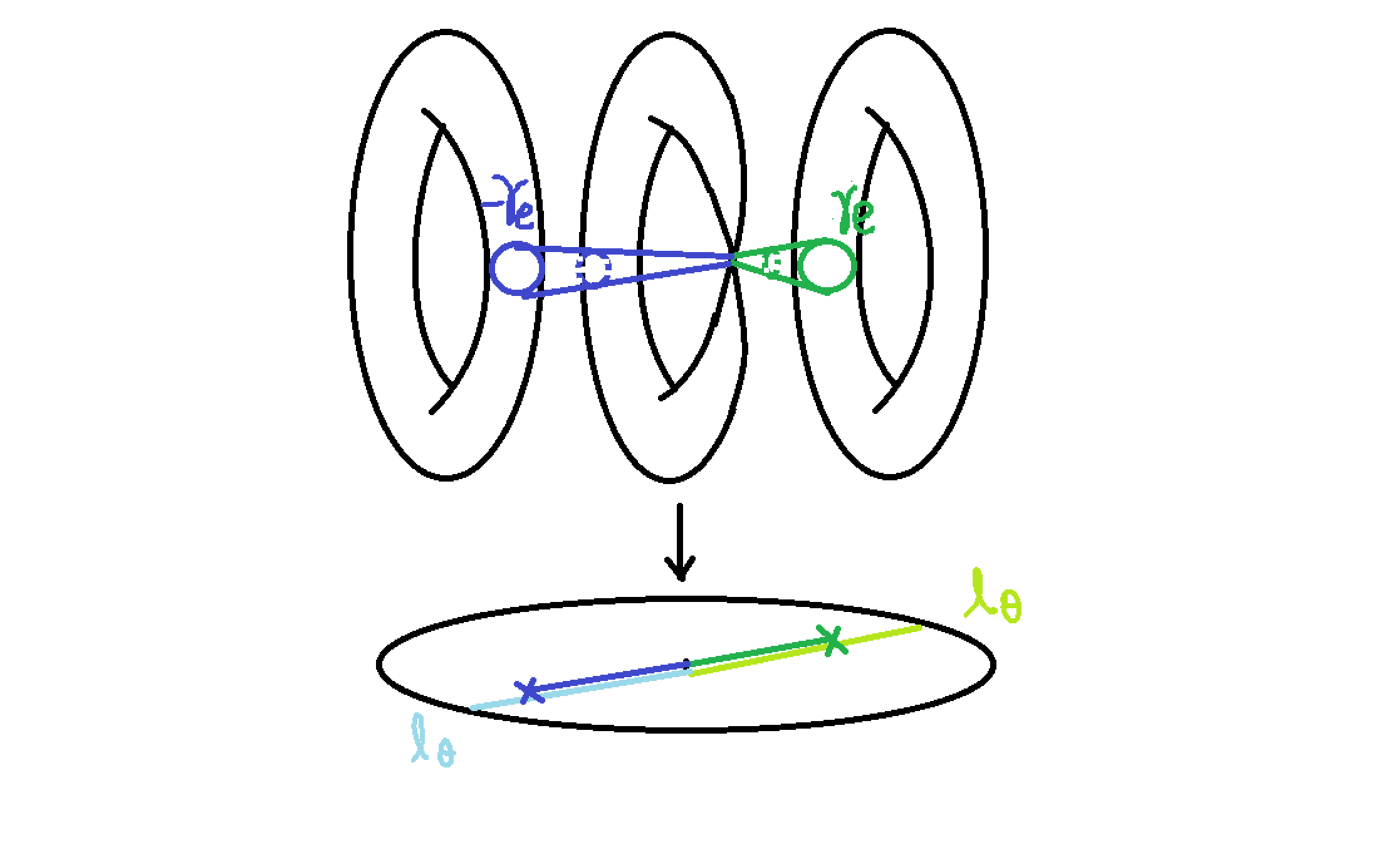}
\caption{The Ooguri-Vafa space and its unique simple holomorphic disc.}
\end{center}
\end{figure}
 
The unique simple holomorphic disc in the Ooguri-Vafa space is Fredholm regular in the $S^1$-family. Thus, it is straightforward to show that $\tilde{\Omega}^{Floer}(\pm\gamma_e)=1$. Moreover, we compute the multiple cover formula in Section \ref{1016},
   \begin{align*}
         \tilde{\Omega}^{Floer}(\gamma,u)=\begin{cases} \frac{(-1)^{d-1}}{d^2}& \text{, if $\gamma=  d\gamma_e$, $d\in \mathbb{Z}$}\\
                0& \text{, otherwise}.
                                                          \end{cases}
        \end{align*}
   
\begin{rmk}
  Notice that the invariant computed above indeed satisfies Conjecture \ref{222}. Moreover, it suggests that
     \begin{align*}
      \Omega^{Floer}(\gamma,u)=\begin{cases} 1& \text{if $\gamma= \pm\gamma_e,$}\\
             0& \text{otherwise,}
                                                       \end{cases}
     \end{align*}which coincides with the BPS counting for the Ooguri-Vafa space \cite{GMN}.
\end{rmk}        

\subsubsection{Existence of Initial Discs in Elliptic K3 Surfaces}
From the discussion in Section \ref{801}, there is a unique holomorphic disc with
boundary on every special Lagrangian torus over monodromy invariant
direction in the Ooguir-Vafa space. Thus it gives rise to a special Lagrangian disc after
hyperK\"ahler rotation. To prove there exists such a holomorphic
discs near the $I_1$-type singular fibre of K3 surface with special
Lagrangian fibration, we need some understanding of geometry of K3 surfaces
near large complex limit. Gross and Wilson construct an approximate
metric by gluing Ooguri-Vafa metric with the semi-flat metric for
elliptic K3 surfaces when the K\"ahler class goes large (along a
straight line in K\"ahler cone). Moreover, they derive some estimate
for the approximate metric:
\begin{thm}(\cite{GW})
  Let $X$ be an elliptic K3 surface with holomorphic volume form $\Omega$. There exists an approximate metric $\omega_{\epsilon}$ which equal to the twisted Ooguri-Vafa metric near the singular fibres and equal to the semi-flat metric \cite{GVY} away from singular fibres and with fibre size
  $\epsilon$. Moreover, if
  $F_{\epsilon}=\log{\big(\frac{\Omega\wedge\bar{\Omega/2}}{\omega^2_{\epsilon}}\big)}$,
  then the solution $u_{\epsilon}$ of the Monge-Amp\`ere equation
    \begin{align*}
       \mbox{det}(g_{i\bar{j}}+i\partial_i
       \bar{\partial}_{\bar{j}}u_{\epsilon})=e^{F_{\epsilon}}
    \end{align*}
 satisfies the follow exponential decay property
    \begin{align*}
      \parallel F_{\epsilon}\parallel_{C^k}\leq D_1
      e^{-D_2/\epsilon}
    \end{align*} for some positive constant $D_1$, $D_2$ depending on $k$.
\end{thm}
Heuristically, Gross-Wilson proved that the Ricci-flat metric
converges to semi-flat metric when the K3 surface goes to the large
complex limit. Moreover, the K3 surface collapses to an affine manifold with singularities. This is the modified Strominger-Yau-Zaslow conjecture proposed \cite{GW}\cite{KS4}\cite{SYZ}. Therefore, We will use this asymptotic behavior of metric in Corollary \ref{800} as the working definition of large complex limit of K3 surfaces. However, the estimate is not accurate enough in the sense that it is not clear if it includes the contribution of instanton correction from the initial rays (and thus the higher order correction).

Using this estimate and the third order estimates in \cite{Y1}, we have the following corollary:
\begin{cor} \label{800}
    Let $\omega$ be the unique Ricci-flat metric in the class $[\omega_{\epsilon}]$, then the difference with the approximate metric $\omega_{\epsilon}$ is 
      \begin{align*}
        \parallel \omega-\omega_{\epsilon}\parallel_{C^1} \leq D e^{-D'/\epsilon}.
      \end{align*}
\end{cor}

\begin{rmk}
In the paper \cite{GW}, they derived all the $C^{k,\alpha}$-estimate for all $k$ but away from the singular fibres. It is not known to the author at the moment if we can get better global estimate of the difference between the approximate metric $\omega_{\epsilon}$ and true Ricci-flat metric $\omega$.
\end{rmk}

To prove that the fiber $T_y$ indeeds bounds a holomorphic disc corresponding to the initial BPS ray, we will take the following indirect method:
First we will view the $X$ as an elliptic K3 and construct an special Lagrangian disc with boundary on $T_y$ but with respect to the approximate metric $\omega_{\epsilon}$. Then we will deformed it to a special Lagrangian discs with respect to the true Ricci-flat metric and thus it becomes a holomorphic discs after hyperK\"ahler rotation.

\begin{thm} \label{60}
   Let $(X,\omega,\Omega)$ be an elliptic K3 surface with a holomorphic section \footnote{If we only want to know the existence of initial discs near the singular fibre then one can drop this assumption.}, where $\omega$ is the Ricci-flat metric and $\Omega$ is a nowhere vanishing holomorphic volume form. Let $u$ be a point on the above affine ray starting at the singular point $p$. Assume there is no other singular point of affine structure on the affine segment between $u$ and $p$. Then there exists $\epsilon_0=\epsilon_0(u)>0$ such that there exists an immersed holomorphic disc in the relative class $\gamma_e$ and boundary on $L_u$ in $X_{\vartheta}'$. Here $X_{\vartheta}'$ is a K3 surface with special Lagrangian with $\int_{L_u}\omega'<\epsilon_0$. 
 \end{thm}
\begin{proof}
Without loss of generality, we will construct a special Lagrangian disc of phase $\vartheta=0$. Assume $y$ is on the BPS ray emanated from a singular point. Take a
neighborhood $\mathcal{U}$ of this segment from the singularity to
$y$ and lift the special Lagrangian fibration to the universal cover
of $\mathcal{U}$. The affine coordinate on $\mathcal{U}$ is $(y_1,
-Im\int \tau dy)$ when $|y| < \epsilon$ or $|y|> \epsilon$. Here we
assume the monodromy invariant direction is characterized by
$y_1=0$. 

We first review the construction of $\omega_{\epsilon}$ in \cite{GW}.  From \cite{GW}, there exist a local holomophic section $\sigma$ and a
real function $\phi$ such that
\begin{equation*}
    \omega_{SF}-T_{\sigma}^*\omega_{OV}=i \partial \bar{\partial}\phi
\end{equation*}
Note that the local $S^1$-action (given by $t\rightarrow t+a$) still live on the torus fibration over $\mathcal{U}$, therefore, we can average $\phi$ with respect to the local $S^1$-action and assume it is $S^1$-invariant.
\begin{align*}
 \omega_{\epsilon} &=\omega_{SF}+i\partial \bar{\partial}(\psi(|y|^2)\psi) \\
                                 &=(1-\psi(|y|^2))\omega_{SF}+\psi(|y|^2)T_{\sigma}^*\omega_{OV}    \\
                                   & \quad -i(\psi'(|y|^2)\bar{y}dy\wedge \bar{\partial}\phi+ \psi'(|y|^2)y\partial \phi \wedge d\bar{y}+ \psi''(|y|^2)|y|^2\phi dy \wedge d\bar{y}).
\end{align*}
\begin{rmk}
   There is actually another term should be added to $\omega_{\epsilon}$ to make it in the right cohomology class.
\end{rmk}
The natural candidate for the approximate special disc is given by
$y_1=u=0$, so we want to prove that $\omega_{\epsilon}$ restricted on it to
zero. It is easy to see that the first, second and last term
restrict to zero. Then third and fourth term together restrict to
the disc is $-2y_2dy_2 \wedge Re\partial \phi =-y_2 dy_2 \wedge
d\phi$. Notice that $d\phi (\frac{\partial}{\partial t})=0$, since $\phi$ is
real and is chosen to be $S^1$-invariant. Thus, we get a smooth special Lagrangian
disc with respect to approximate metric $\omega_{new}$ and
boundary on an elliptic torus fibre.

The following proof is the standard deformation theory of special
Lagrangian submanifolds (with boundaries) modified from \cite{B}.
Start with a smooth special Lagrangian disc $f:(D^2,
\partial D^2) \rightarrow ((X, \omega_{\epsilon}, \Omega),L)$ we consider the
a family of embeddings
\begin{align*}
\Phi_V: (D^2,\partial D^2) \rightarrow ((X,\omega_t,\Omega_t),L)\\
\phi_V(x)= \exp_{f(x)}(V(x)),
\end{align*}
where $\omega_t=\omega_{\epsilon}+t(\omega-\omega_{\epsilon})$,
$\Omega_t=e^{i\theta(t)}\Omega$. Also, $\exp$ is specially
constructed and $V$ should satisfy the Neumann boundary condition to
keep $\Phi_V(\partial D^2) \subseteq L$.

It is natural to write down
\begin{align} 
    F(V,t):& W^{1,p}(N_{L/X}) \times \mathbb{R} \rightarrow  X \nonumber\\ 
      F(V,t)(x)=& \bigg( (\exp_xV(x))^*\omega_t, (\exp_xV(x))^*Im(e^{i\theta(t)}\Omega_t) \bigg)
\end{align}
The embedding $\Phi_V(D^2,\partial D^2)$ gives a special Lagrangian
disc in the K3 surface $(X,\omega,\Omega)$ if and only if $F(V,1)=0$. To solve the equation, we first compute its linearization
\begin{equation} \label{1031}
    dF_{(V,t)}(W,s)=(\exp V)^*\bigg(d\iota_V\omega_t+s\frac{d\omega_u}{du}|_t, d\iota_V Im\Omega_t+ s\frac{d Im(e^{i\theta(u)}\Omega_u)}{du}|_t  \bigg).
\end{equation}
In particular, the linearization at $(0,0)$
\begin{equation}
    dF_{(0,0)}(W,s)= \big(d\eta+ s\frac{d\omega_t}{dt}|_0, d*(\psi \eta)+ s \psi \theta'(0) vol_{g_a} \big),
\end{equation}
where $\eta=\iota_W\omega_{\epsilon}$ is a $1$-form and $Re\Omega|_{D^2}=\psi
vol_{g_a}$. Notice that $*(\psi\eta)=*_{\psi}\eta$, where $*_{\psi}$
is the Hodge star operator for the another metric $g_{\psi}$.
Note that $F$ will factor through $C^{0,\alpha}(d\Omega^1(D^2))
\times C^{0,\alpha}(d\Omega^3(D^2))$ because $[(\exp
V)^*\omega_t]=[\omega_{\epsilon}]=0$ and $[(\exp
V)^*Im\Omega_t]=[Im\Omega]=0$.
\begin{prop}
   The linearized operator
   \begin{align*}
   dF_{(0,0)}: C^{1,\alpha}(N_{L/X})^N \times \mathbb{R} \rightarrow C^{0,\alpha}(d\Omega^1(D^2)) \times C^{0,\alpha}(d\Omega^1(D^2))
   \end{align*}\footnote{
     Here we need the $C^{2,\alpha}$-estimate of $u_{\epsilon}$.}
    is surjective, if $\theta(t)$ chosen suitably.
\end{prop}
\begin{proof}
  Let $N$ be a normal vector field of $\partial D^2$ and $\xi \in C^{1,\alpha}(\Omega^1(D^2))$ and $\gamma \in C^{1,\alpha}(\Omega^1(D^2))$. We are trying to solve the system
     \begin{align}
          d\eta&=d\xi+ \frac{d\omega_t}{dt}|_0   \\
          d*_{\psi}\eta&=d\psi+ \frac{d Im\Omega_t}{dt}|_0  \\
          \eta(N)&=0
     \end{align}
  Hodge theory for a manifolds with boundary (p.123 \cite{S1}) shows that this system of equations can be solved with H\"older regularity if and only if
     \begin{enumerate}
          \item  $d \big( d\xi+ \frac{d\omega_t}{dt}|_0 \big)=0= d\big( d\psi+ \frac{d Im\Omega_t}{dt}|_0 \big)$
          \item  $\big( d\psi+ \frac{d Im\Omega_t}{dt}|_0 \big)(E_1,E_2)|_{\partial D^2}=0$ for any vector $E_1, E_2$ tagent to $\partial D^2$.
          \item  $\int_{D^2}( d\xi+ \frac{d\omega_t}{dt}|_0
          \big)*_{\psi}\lambda=0$, for every $\psi$-harmonic form
          $\lambda$ of $D^2$ satisfying Neumann boundary condition.
          \item $\int_{D^2} *_{\psi}\big( d\psi+ \frac{d Im\Omega_t}{dt}|_0
          \big)*_{\psi}\kappa=0$, for every $\psi$-harmonic form
          $\kappa$ of $D^2$ satisfying Neumann boundary condition.
      \end{enumerate}
 Notice that result quoted in \cite{S1} are stated for differential forms with Sobolev regularity but extend to H\"older regularity by standard bootstrapping
 argument for elliptic operators.
 Therefore, $dF$ is surjective if there is no cohomological obstruction from $\big( \omega_t, Im(e^{i\theta(t)}\Omega_t) \big)$, which can be achieved if we choose $\theta(t)$ such that
     \begin{equation*}
         \int_{D^2}d\psi+ \theta'(0) \int_{D^2}\psi Vol_{g_a}=0.
     \end{equation*}
\end{proof}
For the injectivity of $dF_{(0,0)}$, we consider the following system
of equations
\begin{align*}
    d\eta+ s(\omega-\omega_{\epsilon})=0 \\
    d*_{\psi}\eta + s\psi \theta'(0) vol_{g_a}=0 \\
    \eta(N)=0.
\end{align*}
Integrating the second equation over the disc $D^2$, we have
\begin{align*}
    s\psi \theta'(0) Vol(D^2) &= -\int_{D^2}d*_{\psi}\eta \\
                                            &= -\int_{\partial D^2} *_{\psi}\eta =0.
\end{align*}
The last equality holds because the tangent component of $*_{\psi}\eta$ equals to the Hodge star of the normal component of $\eta$, which is zero. Thus the solutions of the system are exactly harmonic
$1$-forms with Neumann boundary condition on $D^2$ (See Section 6
\cite{S1}), which has the same dimension as $b^1(D^2)=0$. However,
the distance from Ooguri-Vafa space to a neighborhood of $I_1$-type
singular fibre of K3 surface is non-trivial. Therefore, we need the
following quantitative implicit function theorem. 
\begin{prop}\label{30}\cite{Z}
   Assume that $\mathcal{B}_1, \mathcal{B}_2$ are Banach spaces and $F:\mathcal{B}_1 \times \mathbb{R} \rightarrow \mathcal{B}_2$ is a map with continuous Frechet derivative. If we have the following conditions:
      \begin{enumerate}
          \item $\frac{\partial F}{\partial V}(0,0): \mathcal{B}_1 \rightarrow \mathcal{B}_2$ is invertible and $\parallel \frac{\partial F}{\partial V}(0,0)^{-1} \parallel \leq C$.
          \item There  exists $r_0 > r>0, t_0>0$ such that for every $(V,t)\in U_{\mathcal{B}_1}(r_0) \times [0,t_0]$,
               \begin{equation}
                   \parallel \frac{\partial F}{\partial V}(V,t)-\frac{\partial F}{\partial V}(0,0) \parallel \leq \frac{1}{2C}, \hspace{2mm} \parallel F(0,t)\parallel _{\mathcal{B}_1} \leq \frac{r}{2C}.
                \end{equation}
       \end{enumerate}
   Then there exists a unique $C^1$-path $V(t)$ in $U_{\mathcal{B}_1}(r)$ for each $t\in [0,t_0]$ such that $F(V(t),t)=0$.
\end{prop}

We will apply Proposition \ref{30} to the linearization (\ref{1031}). So we still need effective bounds for
$\parallel \frac{\partial F}{\partial V}(0,0)^{-1} \parallel$,
$\parallel \frac{\partial F}{\partial V}(V,t)-\frac{\partial
F}{\partial V}(0,0) \parallel$ and $\parallel F(0,t)\parallel$. The first term has the estimate
\begin{align*}
\parallel \frac{\partial F}{\partial V}(0,0)^{-1} \parallel \leq C\lambda_1 \leq
\frac{C'}{diam(X,\omega_{\epsilon})^2}=O(\epsilon),
\end{align*} where $\lambda_1$ here is the first eigenvalue of the Laplace operator associated with $\omega_{\epsilon}$. The second inequality is first eigenvalue estimate and the equality comes from the estimates in \cite{GW}. For the later one, we have
$\parallel \frac{\partial F}{\partial V}(V,t)-\frac{\partial
F}{\partial V}(0,0) \parallel$, $\parallel F(0,t)\parallel
$ are of order $O(C_1e^{-D_2/\epsilon})$ from
the estimates in \cite{GW}.

\end{proof} 
 
\begin{definition}
 The holomorphic disc constructed Theorem \ref{60} is called the initial disc, which corresponds to an initial BPS ray in Section \ref{1032}.
\end{definition} 

There is another way to prove that the regularity of the initial holomorphic
discs by applying the automatic transversality \cite{W2} to K3 surfaces

\begin{prop}
 Assume $X$ is a K3 surface and $L$ is a special Lagrangian, then the
bundle pair $(TX,TL)$ always has Maslov index is $0$. Let $f:(D^2,
\partial D^2) \rightarrow (X,L)$ to be a holomorphic disc with
boundary on $L$. Let $\bold{D}^N_f$ and $\bold{D}^T_f$ be the normal
and tangent splitting of the usual linearized Cauchy-Riemann
operator of $f$, then
%\begin{enumerate}
%   \item $\bold{D}^N_f$ is injective.
%   \item 
   $f$ is regular (in $S^1$-family) if and only if it is immersed or equivalently,
         $(f^*TX,f^*TL)_{\mathbb{C}}\cong \mathcal{O}_{\mathbb{P}^1}(-2)\oplus \mathcal{O}_{\mathbb{P}^1}(2)$.
%\end{enumerate}
\end{prop}
\begin{proof} 
%The first part is follows from the fact that
%$\mbox{Ind}(\bold{D}^N_f)$ is less than zero and by Proposition 2.2 (i)
%\cite{W2}. The second part is 
This is a direct consequence of Theorem 3 in
\cite{W2}
\end{proof}

%\begin{rmk}
%  The author do not know any direct modification of the proof of
%  automatic transversality for $S^1$-family due to the fact that the
%  line bundle generated by $K_J(f,\alpha)$ doesn't have a good
%  splitting with respect to $\bold{D}^N$ and $\bold{D}^T$.
%\end{rmk}

We are first interested in the holomorphic discs with small symplectic area. 
\begin{lem} \label{360}
   Let $X$ be a K3 surface and a Lagrangian fibration $X\rightarrow B$. Given any
   $\delta>0$, then there exists $\epsilon(\delta)>0$ such that any
   holomorphic disc with non-trivial Lagrangian boundary condition satisfies 
     \begin{enumerate}
      \item symplectic area less than $\epsilon(\delta)$ 
      \item non-trivial boundary homology class 
     \end{enumerate}
     should contained in an $\delta$-neighborhood of a singular fibre.
\end{lem}
\begin{proof}
   This is a consequence of the gradient estimate of harmonic maps
   (ex. theorem 2.1 \cite{Y2}). Indeed, a small symplectic area holomorphic disc with boundary on a smooth torus fibre $L$ area will fall in a neighborhood of
   $L$. However, a (topological) trivial $\mathbb{R}^k\times T^l$
   fibration cannot bound any disc with non-trivial boundary
   homology class.
\end{proof}

The following is another well-known folklore theorem. 
\begin{thm} \label{43}
   Let $\gamma_e$ be the Lefschetz thimble, then
  \begin{align*}
   \tilde{\Omega}^{Floer}(\gamma_{e}; u)=1,
  \end{align*}
   for $u$ closed enough to a singularity of affine structure. Moreover, for $u$ close enough to the singularity,
   $\gamma_e$ is the class achieves minimum energy with $\tilde{\Omega}^{Floer}(\gamma)\neq 0$.
\end{thm}
\begin{proof}
Assume $X$ is a K3 surface with special Lagrangian fibration around
large complex limit. We will replace $X$ by the
preimage of a $\epsilon$-neighborhood of singular point on the base,
with the topology same as Ooguri-Vafa space. Notice that we still
have Gromov's compactness theorem for holomorphic discs with small area in $X$ because of Lemma \ref{360}.

First view the Ooguri-Vafa space as an elliptic fibration. Assume
$\omega_{K3}=\omega_{OV}+i\partial \bar{\partial}\phi$, where $\phi$
is a smooth function.
\begin{lem}
There is a path of hyperK\"ahler triples $X_t=(\underline{X}, \omega_t,
\Omega_t), t\in [0,1]$ connecting the restriction of K3 and the Ooguri-Vafa
space, keeping the elliptic fibration structure.
\end{lem}
\begin{proof}It suffices to prove that the existence and uniqueness of solution to complex Monge-Amp\`ere equation with Dirichlet boundary condition (\ref{1006}) and the solution is smoothly depends on the boundary condition. Take
$u_t=t\phi$ and by the estimate\footnote{Actually we only need the $C^0$-estimate of the difference between approximate metric and Ricci-flat metric here} in Corollary \ref{800},  there exists a
non-negative constant $\epsilon_t$, such that $\epsilon_t=0 $ for
$t\in \{0,1\}$, and $u_t$ are subsolutions of the below equation
\begin{align}\label{1006}
   \begin{cases} (\omega_{OV}+i\partial \bar{\partial}u_t)^2=
   (\frac{1}{2}-\epsilon_t)\Omega\wedge\bar{\Omega} \\
   u_t|_{\partial X}=t\phi  \end{cases}
\end{align}
Therefore, the Theorem 1.1 in \cite{GL} provides a solution for the
Dirichlet problem in (\ref{1006}). Next we want to prove uniqueness of the solution the equation (\ref{1006}). Assume there are two
solutions $\psi$,$\phi$ then we have
   \begin{align*}
      \mbox{det}(g_{i\bar{j}}+\phi_{i\bar{j}})=\mbox{det}(g_{i\bar{j}}+\psi_{i\bar{j}})
   \end{align*} or can rewrite it as
   \begin{align*}
      \mbox{det}(g_{i\bar{j}}+\phi_{i\bar{j}}+(\psi_{i\bar{j}}-\phi_{i\bar{j}}))\mbox{det}(g_{i\bar{j}}+\phi_{i\bar{j}})^{-1}=1
   \end{align*}
By arithmetic-geometric mean inequality,
    \begin{align*}
      \frac{1}{2}[2+\Delta'(\psi-\phi)]\geq 1,
    \end{align*} where $\Delta'$ is the Laplacian of metric
    $(g_{i\bar{j}}+\phi_{i\bar{j}})$. Therefore, $\psi-\phi$ is
    subharmonic with respect to $\Delta'$. Since $\psi$ and $\phi$
    are smooth functions, we may add a constant and assume
    $\psi-\phi \geq 0$. Then
    \begin{align*}
       0=\int_X
       \Delta'(\psi-\phi)^2=2\int_X(\psi-\phi)\Delta'(\psi-\phi)+2\int_X|\nabla'(\psi-\phi)|^2
    \end{align*}
   All the terms arises from integration by part vanishes because $\psi$ and
   $\phi$ satisfy the same Dirichlet boundary condition. The first term of right hand side is non-negative implies
   $\nabla'(\psi-\phi)=0$ or $\phi$, $\psi$ differ by a constant.
\end{proof}

The ambient space $\underline{X}$ is diffeomerphic to 
  \begin{align*}
    X':=\mathbb{C}^2 \backslash \{xy-\epsilon=0\}.
  \end{align*}
There exists a $T^2$-fibration on $X'$ with one $I_1$-singular fibre given by 
  \begin{align*}
     & X'\longrightarrow \mathbb{R}^2  \\
      &  (x,y) \mapsto (|x|^2-|y|^2,\log{|xy-\epsilon|}).
  \end{align*} Let $L'$ be any smooth torus fibre. In particular, one have 
   \begin{align*}
     H_2(X,L_u)\cong H_2(X',L') \cong \mathbb{Z},
   \end{align*} which is generated by the Lefschetz thimble. Now the theorem follows from the cobordism argument similar to Proposition \ref{2012}.

\end{proof}

Similar argument also applies to multiple cover of the initial discs:
\begin{thm}\label{1033}
 Let $\gamma_e$ be the Lefschetz thimble and $d\in \mathbb{Z}$, then
   there exists a neighborhood $\mathcal{U}$ might depending on $d$ such that for $u\in \mathcal{U}$,  there are no relative classes $\gamma$ with $\partial\gamma\neq 0$, $\tilde{\Omega}^{Floer}(\gamma;u)\neq 0$ and $|Z_{\gamma}|<|Z_{d\gamma_e}|$ unless $\gamma=d'\gamma_e$, $|d'|<|d|$.
\end{thm}
We will compute the multiple cover formula for the invariants $\tilde{\Omega}^{Floer}(\gamma_e;u)$ when $u$ is near the singularity of the affine structure in the next section.
\subsection{Multiple Cover Formula of the Initial Holomorphic Discs} \label{1016}
To prove the correspondence theorem, one ingredient is multiple cover
formula of the holomorphic discs discussed in Theorem \ref{43}. Use
argument similar to Theorem \ref{43}, it suffices to compute the
multiple cover contribution from the following local model below:
$X=T^*\mathbb{P}^1$ and $L\cong \mathbb{R}^1\times S^1$ is the fixed
locus of an anti-symplectic, anti-holomorphic involution $\iota$. Using localization, the multiple cover of disc
invariants in certain cases are found in \cite{KS3}\cite{GZ}\cite{PSW}.

We first construct an special Lagrangian fibration with respect to
Iguchi-Hanson metric with $L$ one of the fibre. Let
$X=T\mathbb{P}^1$ be the blow-up of $\mathbb{C}^2/\mathbb{Z}_2$ at
the origin. Let $(y,\lambda),(x,\mu)$ be the coordinate chart on
$X$, where $x=1/y$ and $\mu=\lambda y^2$. Then there is an natural
$S^1$-action on $X$ preserving $\lambda y=\mu x$. The Iguchi-Hanson
metric $\omega_{EH}$ is invariant under this $S^1$-action and thus
   \begin{align*}
      X & \longrightarrow \mathbb{R}^2 \\
     (y,\lambda)&\mapsto (\mu_{S^1}, \mbox{Re}(\lambda y))
   \end{align*}
is a special Lagrangian fibration with smooth fibres homeomorphic to
$\mathbb{R}\times S^1$ \cite{G2}. Let $\iota: (y,\lambda)\mapsto
(\bar{y},\bar{\lambda})$ be an involution which is both anti-holomorphic and
anti-symplectic. The fixed locus of $\iota$ is a special
Lagrangian fibre $L$. Moreover, by maximal principle there exists a
unique simple holomorphic disc (up to reflection $\iota$) in $X$
with boundary on $L$. This is a local model near the singularity of
Ooguri-Vafa space. Locally around the singularity of the fibration,
Iguchi-Hanson metric and Ooguri-Vafa metric are both $S^1$ invariant
and thus arise from Gibbons-Hawking ansatz (section 2.6 \cite{GW}).
%\begin{rmk}
%The statement in \cite{GW} is only for the situation the fibration structure is %trivial. One can solve for the potential function $V$ for Iguchi-Hanson metric %locally and it will admit no monodromy since the Iguchi-Hanson metric is globally %defined.
%\end{rmk}
Therefore, there is a family of hyperK\"ahler triple from
Gibbons-Hawking ansatz connecting these two spaces. Similar to the Ooguri-Vafa example, there is a unique embedded holomorphic disc in
each space by maximal principle. Therefore, pseudo-isotopy of Kuranishi structure along
this family guarantees we can just compute the multiple cover
formula for moduli space
$\mathcal{M}_d=\mathcal{M}_{d\gamma}(\mathfrak{X},L)$ of discs in
$X=T^*\mathbb{P}^1$ (with $S^1$-family of complex structures by
hyperK\"ahler rotation) with boundary on $L$ and image only mapping
to certain side of the equator.

Let $\gamma$ be the unique (up to sign) relative class bounding
simple holomorphic disc and $\tilde{\gamma}$ is the homology class
of zero section $\mathbb{P}^1$ in $X$. Since the invariant
$\tilde{\Omega}^{Floer}(d\gamma)$ is independent of the choice of
Kuranishi structures, we will choose the Kuranishi structure for
computation purpose as follows: The obstruction bundle $F_d$ is an
orbi-bundle over an orbifold $M_d$ and the Kuranishi map is just the
zero section. For each point $[f]\in \mathcal{M}_d$, it has a
corresponding point in $M_d$ again denoted by $[f]$ and
   \begin{align*}
     F_d|_{[f]}=H^1(D^2,\partial D^2;f^*TX,f^*TL)/\mathbb{R},
   \end{align*}where the quotient $\mathbb{R}$-factor is induced by
   the $S^1$-family of complex structures by hyperK\"ahler rotation.
Also the tangent space of $M_d$ at $[f]$ is
   \begin{align*}
      TM_d|_{[f]}=H^0(D^2,\partial D^2;f^*TX,f^*TL)/Aut(D^2).
   \end{align*}
 Notice that
$\mathcal{M}_d$ is bijective with the moduli space
$\tilde{\mathcal{M}}_d$ of real rational curves of degree $d$ in a
twistor family of $X$. (Notice that this point is quite different
from the situation of \cite{PSW}) Therefore, we double the Kuranishi
structure on $\mathcal{M}_d$ and equip $\tilde{\mathcal{M}}_d$ with
an $\iota$-equivariant Kuranishi structure under this
identification. We may choose the perturbed multi-section on
$\tilde{M}_d$ to be $\iota$-invariant. Therefore,
 \begin{align}\label{320}
    \tilde{\Omega}^{Floer}(d\gamma):&=Corr_*(\mathcal{M}_d;tri,tri)(1)\notag \\
                            &=\frac{1}{2}Corr_*(\tilde{\mathcal{M}}_d;tri,tri)(1)
 \end{align}
The Kuranishi structure on $\tilde{\mathcal{M}}_d$ is a smooth
closed orbifold $\tilde{M}_d$ of real dimension $2d-2$ with an
orbibundle $\tilde{F}_d$ and Kuranishi map is the zero section. In
particular, (\ref{320}) is just the top Chern class of the
orbibundle $\tilde{F}_d$. We will use localization to compute this
top Chern class. Using the fact that $L$ is the fixed locus of
$\iota$, which is both anti-holomorphic and anti-symplectic, one
concludes that maps in $\mathcal{M}_d$ is a $d$-fold cover from
 trees of disks to the zero section $\mathbb{P}^1$ with boundary in $S^1$.
Then there is only two torus fixed point in the moduli space
 $\tilde{\mathcal{M}}_d$, namely,
 the doubling of $f_d$ given by $z \mapsto z^d$ composed with the only embedded
 disc (we will again denote it by $f_d$) and its reflection under the involution $\iota$.

Conversely, the involution $\iota$ can double the holomorphic disc into a rational curve $\mathbb{P}^1$. Moreover, the bundle pair $(f_d^*TX, f^*dTL)$ can doubled to a vector bundle over the $\mathbb{P}^1$ which will be denoted by $(f_d^*TX,f_d^*TL)_{\mathbb{C}}$.

 Let $(a)$ denote the complex line with the $U(1)$-action of weight $a$;
 while $(0)_{\mathbb{R}}$ denotes the real line with the trivial $U(1)$-action.
It is easy to see that the doubling of bundle
\begin{align*}
(f_d^*TX,f_d^*TL)_{\mathbb{C}}\cong \mathcal{O}_{\mathbb{P}^1}(-2).
\end{align*}
 Then straightforward computation shows
\begin{align*}
 &\mbox{Aut}(D^2) = (1/d) + (0)_{\mathbb{R}}, \hspace{50mm} \mbox{ (real dimension $3$)}\\
 &H^0( D^2,\partial D^2; f_d^*TX, f_d^*TL ) =\oplus_{j=1}^d(j/d) + (0)_{\mathbb{R}}.   \quad \quad \mbox{ (real dimension $2d+1$)}
\end{align*}
 So together we have
  \begin{align*}
   T\tilde{M}_d|_{[f_d]}=\oplus_{j=2}^d(j/d).
  \end{align*}
On the other hand,
\begin{align*}
 H^1(D^2,\partial D^2; f_d^*TX, f_d^*TL) = \oplus_{j=1}^{d-1}(-j/d) + (0)_\mathbb{\mathbb{R}} \qquad \mbox{ (real dimension
 $2d-1$)}.
\end{align*}
 So the fibre of obstruction bundle $\tilde{F}_d$ at $f_d$ is given by
 \begin{align*}
   \tilde{F}_d|_{[f_d]} =\oplus_{j=1}^{d-1}(-j/d)
  \end{align*}
because the counting in $S^1$-family is equivalent to changing the
obstruction bundle by $(0)_{\mathbb{R}}$. Finally we get

\begin{align*}
   \tilde{\Omega}^{Floer}(d\gamma)=\frac{1}{2}\int_{\tilde{M}_d}e_{U(1)}(\tilde{F}_d)= \frac{1}{|Aut(f_d)|} \frac{e_{U(1)}(\tilde{F}_d|_{[f_d]})}{e_{U(1)}(T\tilde{M}_d|_{[f_d]})} =\frac{(-1)^{d-1}}{d^2},
\end{align*}which coincides with the multiple cover formula of
$\tilde{\Omega}^{trop}$ for initial discs (\ref{321}). Together with Theorem \ref{1033}, we
prove the multiple cover formula for initial discs:
\begin{thm}\label{301}
  Let $\gamma_e$ be the relative class of Lefschetz thimble around an $I_1$-type singular fibre, then given any $d_0\in \mathbb{N}$, there exists a non-empty neighborhood $\mathcal{U}$ of the singularity such that for each $u\in \mathcal{U}$, we have 
    \begin{align*}
     \tilde{\Omega}^{Floer}(d\gamma_{e}; u)=\frac{(-1)^{d-1}}{d^2}, \mbox{ for every integer $d$, $|d|\leq d_0$}.
    \end{align*}
     Moreover, for $u$ close enough to the singularity,
     $\pm\gamma_e$ are the only classes achieve minimum energy with $\tilde{\Omega}^{Floer}(\gamma)\neq 0$.
\end{thm}

\begin{rmk}
The author conjecture that the neighborhood $\mathcal{U}$ in Theorem \ref{1033} and Theorem \ref{301} can be chosen independent of $d_0\in \mathbb{Z}$.
\end{rmk}

\begin{rmk}
  It is not enough to show that the initial data $\Omega^{Floer}(\gamma,u)=1$ for $u$
  near the singularity from the observation that there is a unique
  simple disc. More than that, one also needs a correct multiple cover formula Theorem \ref{301}.
\end{rmk}

\begin{rmk}
  For general $I_n$ type singular fibres, we can deform the elliptic K3 surfaces (via a family of elliptic K3 surfaces) into the one with only $I_1$-type singular fibres. Each $I_n$-type singular fibre will deform to $n$ copies of $I_1$-type singular fibres in a neighborhood with parallel initial rays. Thus, from Theorem \ref{301}, the initial data associated to $I_n$-type singular fibres is 
   \begin{align} \label{1333}
   \tilde{\Omega}^{Floer}(d\gamma_e)=\frac{n(-1)^{d-1}}{d^2}. 
   \end{align}

%   For $n=2$, one can deform the elliptic K3 surface and split an $I_2$-singular %fibre into two $I_1$-singular fibres. In this case, the monodromy invariant %direction from each of the singularity pass through the other and (\ref{1333}) %follows. However, the local model metric from Gibbons-Hawking ansatz has %$\mathbb{Z}_n$-quotient singularity and our argument fails for general $n\in %\mathbb{Z}$.
\end{rmk}

\subsection{Corresponding Theorem}
The classical way of constructing tropical discs is taking certain
adiabatic limit of the images of the holomorphic discs under the
fibration. However, this method usually involves hard analysis and
we don't know much about the Calabi-Yau metric. Therefore, we
introduce here another point of view of tropical discs from the
locus of Lagrangian fibres bounding the prescribed class of
holomorphic discs.

We want to prove that all holomorphic discs in this $S^1$-family are all from "scattering"-
gluing of these (multiples of) initial discs. Assume $\gamma_u\in H_2(X,L_u)$ is represented as a holomorphic disc
with boundary on $L_u$ such that
$\tilde{\Omega}^{Floer}(\gamma;u)\neq 0$. By changing $u$ to a generic nearby point, we may assume $\mbox{Arg}Z_{\gamma_u}=0$ and is generic. There is an affine half line $l$ emanating from $y$
on the base such that $Z_{\gamma_t}$ is a
decreasing function of $t\in l$, where $ \gamma_t$ is the
parallel transport of $\partial \gamma_u$ along $l$. 
The function $Z_{\gamma_t}$ is a strictly decreasing function along $l$ without lower bound. Therefore, there is some point $u'\in l$ such that
$Z_{\gamma_{u'}}$. Thus, there are two cases:
  \begin{enumerate}
    \item If $\tilde{\Omega}^{Floer}(\gamma;u)=\tilde{\Omega}^{Floer}(\gamma;u')$ then $L_{u'}$ is a singular fibre by Lemma \ref{360}.  In particular,
    if $L_{u'}$ is of $I_1$-type singular fibre then $\gamma_{u'}$ is represented by multiple cover of the unique area
    minimizing holomorphic disc and $\partial \gamma_u$ is the parallel of $\gamma_{u'}$ along
    $l$.
    \item If $\tilde{\Omega}^{Floer}(\gamma;u)\neq \tilde{\Omega}^{Floer}(\gamma;u')$ then from the Proposition \ref{310} and Theorem \ref{64}, there exists
    $\gamma_{n,u'}$ in the same phase with $\gamma_{u'}$ such that
    $\sum_n k_n\gamma_{n,u'}=\gamma_{u'}$ and $\tilde{\Omega}^{Floer}(\gamma_n;u')\neq 0$. Then we replace $\gamma$ by $\gamma_{i,u'}$ and repeat the same
    processes. The procedure will stop at finite time. Indeed, if there are infinitely many jumps of invariants before $Z_{\gamma_t}$ becomes negative, then there is a limit points $u_0$ where the invariants jump. If $u\in B_0$, then it contradicts to Gromov compactness theorem. If $u\in \Delta$, then it contradicts to the minimality of the initial disc (Theorem \ref{43}). By induction, every holomorphic
    disc with nontrivial invariant give rise to a tropical disc,
    which is formed by the union of the affine segments on the base.
    In particular, the balancing conditions are guaranteed by the
    conservation of charges $\sum_n k_n\gamma_{n,u'}=\gamma_{u'}$ at
    each vertex $u'$.
  \end{enumerate} To sum up, we proved the following theorem by this
  attractor flow mechanism \cite{DM} of holomorphic discs.

\begin{thm} \label{47}
  Let $X$ be an elliptic K3 surface (singular fibres not necessarily of $I_1$-type). For every relative class $\gamma\in H_2(X,L_u)$ with
  $\tilde{\Omega}^{Floer}(\gamma;u)\neq 0$ can be associated an image of a
  tropical disc $\phi$ such that $[\phi]=\gamma$. Moreover, the
  symplectic area of the holomorphic disc is just the total affine
  length of the corresponding tropical disc.
\end{thm}

\begin{figure}
\begin{center}
\includegraphics[height=3in,width=6in]{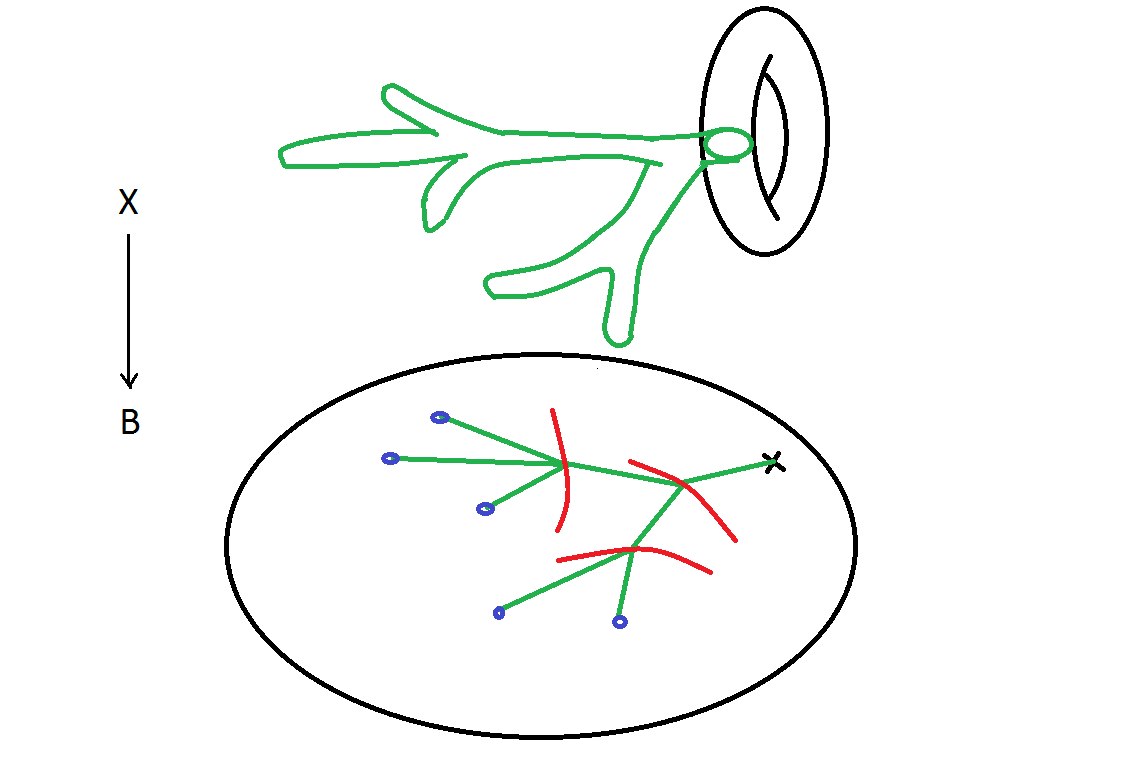}
\caption{Holomorphic discs with non-trivial invariant and its corresponding tropical disc. The red curves are walls of marginal stability and blue little circles are singularity of the affine structure}
\end{center}
\end{figure}

Theorem \ref{47} also helps to understand the topology of holomorphic discs in K3 surfaces with special Lagrangian fibration with only $I_1$-type singularities in the following sense:
\begin{cor} All the holomorphic discs with non-trivial boundary class and with non-trivial open Gromov-Witten invariants in $X$ topologically are coming from "scattering" (gluing) of discs coming frommn singularities.
\end{cor}

Below we will demonstrate a non-trivial example of wall-crossing
phenomenon of invariants $\tilde{\Omega}^{Floer}(\gamma;u)$.
\begin{ex} \label{900}
  Assume there are two initial rays emanating from two $I_1$-type singularities of phase $\vartheta_0$ intersect
  at $p\in B_0$. From Theorem \ref{60}, there are two initial holomorphic
  discs of relative classes $\gamma_1,\gamma_2$ corresponding to the initial rays which are Fredholm. Moreover, the local model provided in the proof of
  Theorem \ref{60} indicates that they intersect transversally in $L_p$. From automatic
  transversality of K3 surfaces, these two discs cannot be smoothed out in
  $L_p$. To prove that these two discs will smooth when changing the Lagrangian boundary condition, First pick two point $p_1,p_2$ near $p$ but on the
  different side of wall of marginal stability
  $W_{\gamma_1,\gamma_2}$. Let $\psi:(-\epsilon, 1+\epsilon)$ be a path on $B_0$ such that $\psi(0)=p_1$, $\psi(1)=p_2$ and intersect $W_{\gamma_1,\gamma_2}$
  once transversally at $p$. Recall $\mathcal{X}$ is the total space of twistor space of $X$ with two fibres with elliptic fibration threw
  away. Then $L_u\times S^1_\vartheta$ is a totally real torus in
  $\mathcal{X}$. Now consider an complex manifold $\mathcal{X}\times \mathbb{C}$
  with a totally real submanifold
    \begin{align*}
       \mathcal{L}=\bigcup_t (L_{\psi(t)}\times S^1_{\vartheta}).
    \end{align*}
  By our assumption, there are two regular holomorphic discs in
  $\mathcal{X}$ with boundaries in
  $L_p\times\{\vartheta_0\}\subseteq \mathcal{L}$ of relative classes again we denoted by $\gamma_1$, $\gamma_2$. The tangent of
  evaluation maps for both discs are two dimensional and
  transversal from the $C^1$-estimate in Corollary \ref{800}. By Theorem 4.1.2 \cite{BC}, these two discs can be smoothed out
into simple regular discs in $\mathcal{L}$ and the union of initial
holomorhpic discs are indeed the codimension one of the boundary of
the usual moduli space of holomorphic discs
$\mathcal{M}_{0,\gamma_1+\gamma_2}(\mathcal{X},\mathcal{L})$. By maximal principle twice, each of the
holomorphic disc falls in
$\mathcal{M}_{\gamma_1+\gamma_2}(\mathfrak{X},L_{\psi(t)})$ for some
$t$. In particular,
   \begin{align*}
     \mathcal{M}_{1,\gamma_1}(\mathfrak{X},L_p)\times_{L\times
     S^1_{\vartheta}}\mathcal{M}_{1,\gamma_2}(\mathfrak{X},L_p)\subseteq
     \mathcal{M}_{\gamma_1+\gamma_2}(\mathfrak{X},\{L_t\})
   \end{align*} as codimension one boundary. Therefore,
   \begin{align*}
      \Delta\tilde{\Omega}^{Floer}(\gamma_1+\gamma_2)&=\pm Corr_*(\mathcal{M}_{1,\gamma_1}(\mathfrak{X},L_p)\times_{L\times S^1_{\vartheta}}\mathcal{M}_{1,\gamma_2}(\mathfrak{X},L_p);tri,tri)(1)\\
      &=\pm\langle\gamma_1,\gamma_2\rangle
      \tilde{\Omega}^{Floer}(\gamma_1,p)\tilde{\Omega}^{Floer}(\gamma_2,p),
   \end{align*}
by Theorem \ref{60}, Proposition \ref{2} and Theorem \ref{32}.
Assume moreover, that the two $I_1$-type singularities on the base
are closed enough to each other. Using the same argument in the
proof of Theorem \ref{47}, the difference of the invariant appear
the side of the wall of marginal stability satisfying (\ref{31}) and
thus determined the sign.
\end{ex}

\begin{rmk}
  One can view the phenomenon in Example \ref{900} as two special Lagrangian discs smooth out along the boundaries only when their boundaries are constrained by fibres on the side of wall of marginal stability characterized by (\ref{902}). Similar phenomenon is studied by Joyce \cite{J1} for smoothing of two (conical) special Lagrangians without boundaries at a transverse intersection point in Calabi-Yau $3$-folds. Consider a family of Calabi-Yau manifolds $\{X_s\}$ and there are two special Lagrangians (of the same phase) $L_1, L_2$ in $X_0$ intersecting transversally at a point. In particular, $X_0$ falls on the locus characterized by 
    \begin{align}\label{903} 
        \mbox{Arg}Z_{[L_1]}=\mbox{Arg}Z_{[L_2]},
    \end{align} which is parallel to Proposition \ref{31}. Here, the central charge is defined by 
    \begin{align}\label{904}
      Z_{[L]}(s)=\int_{[L]}\Omega_s,
    \end{align}where $[L]\in H_3(X_s)$ and $\Omega_s \in H^{(3,0)}(X_s)\backslash \{0\}$. Notice that while (\ref{904}) is not well-defined but the locus characterized by (\ref{903}) is well-defined. Then Joyce proved that the immersed (conical) special Lagragian in the class $[L_1]+[L_2]$ by gluing the Lawlor neck only appears on the side of the wall where 
       \begin{align*}
          \frac{\mbox{Im}(Z_{[L_1]}\bar{Z}_{[L_2]})}{[L_1].[L_2]}>0.
       \end{align*} In general, if $L_1$ and $L_2$ have the all their intersection transversal then the number of immersed (conical) special Lagrangian in the class $[L_1]+[L_2]$ (from gluing the Lawlor neck) on different sides of the wall is differed by $|[L_1].[L_2]|$.
    
\end{rmk}
Notice that the jump of the invariant is exactly the same as
described in Theorem \ref{63}. We expect the wall-crossing formula
(\ref{37}) also holds for $\tilde{\Omega}^{Floer}(\gamma;u)$.
Together with Theorem \ref{301}, it will imply the following
corresponding of counting holomorphic discs and tropical discs:
\begin{conj}\label{35}
   For elliptic K3 surface with $24$ $I_1$-type singular fibres, $\tilde{\Omega}^{Floer}(\gamma;u)=\tilde{\Omega}^{trop}(\gamma;u)$ for every $u\in B_0$. In particular, the open Gromov-Witten invariant $\tilde{\Omega}^{Floer}(\gamma;u)\in
   \mathbb{Q}$.
\end{conj}

Therefore, Conjecture \ref{35} and Theorem \ref{301} together
provide an algorithm to compute all the invariants
$\tilde{\Omega}^{Floer}(\gamma,u)$ on elliptic K3 surfaces with $24$
$I_1$-type singular fibres. For more explicit computation of the
invariants $\tilde{\Omega}^{Floer}(\gamma;u)$, one need to study the affine
geometry induced by elliptic fibration to describe the position of
possible walls of marginal stability. The explicit jump from
wall-crossing formula may be computed from the study of quiver
representation moduli spaces \cite{R1} or proving the generating
functions of these invariants satisfy certain modularity. We close
this section with following conjecture which indicates the existence
of a Gopakumar-Vafa type invariants $\Omega^{Floer}$.
\begin{conj} \label{222}
  There exists $\{\Omega^{Floer}(\gamma;u)\in \mathbb{Z}| \gamma \in H_2(X,L_u)\}$
  such that
     \begin{equation*}
         \tilde{\Omega}^{Floer}(\gamma;u)=\sum_{d>
         0}\pm d^{-2}\Omega^{Floer}(\gamma/d;u).
     \end{equation*}
\end{conj}

\subsection{Other Local Models}
  The proof of Theorem \ref{47} doesn't depend on the type of
  singular fibre we have. However, different kind of singular fibres
  might impose different initial data instead of (\ref{321}).
  \begin{ex}($I_0^*$-type singular fibre) Let $E_{\tau_i}\equiv \mathbb{C}/\langle 1,\tau_i\rangle$ be elliptic curves for $i=1,2$ and Let $X_0$ be the quotient of
$E_{\tau_1}\times E_{\tau_2}$ by the involution $\iota: z_i\mapsto
-z_i$. Let $X$ be the blow-up of $X_0$ at the $16$ ordinary double
points and $X$ is called the Kummer K3 surface. The Kummer K3
surface admits an elliptic fibration with $4$ singular fibres of
$I^*_0$-type. Since the elliptic fibration is isotrivial, the
$S^1$-family of affine structures reduce to the same one, which
coincides with the $\mathbb{Z}_2$-quotient of the standard affine
structure on elliptic curve $E_{\tau_1}$.

It is pointed out in \cite{F3} that given a pair $(k,l)\in
\mathbb{Q}^2$, consider a map $f:\{|s|<1\}\times \mathbb{R}_t
\rightarrow (X_0,E_{(k,l)})$ by
\begin{align*}
   z_1=(k+il)s, \quad z_2=(k+il)t
\end{align*}
When $\parallel (k,l)\parallel \ll 1$, then the image of $f$ will
falls in a small neighborhood of the $I_0^*$-type singular fibre. It
is straight forward to see that the map $f$ lifts to $X$ and the
image is still a smooth disc. Moreover the pull-back of
Eguchi-Hanson metric $\omega$ and imaginary part of the holomorphic volume form of $X$ is zero.
Therefore, by the similar technique of deformation of special
Lagrangian discs and estimates in \cite{L1}, it can be deformed to a
smooth special Lagrangian discs with respect to true Ricci-flat
metric and thus gives rise to a holomorphic disc from hyperK\"ahler
rotation trick. In particular, they become regular holomorphic discs
after hyperK\"ahler rotation. We expect these are all the
holomorphic discs near the $I_0^*$-type singular fibre.
\end{ex}

\appendix
\flushbottom
\include{AppendixA}
\flushbottom
\section{Kuranishi Structures and related terminologies}
This Appendix is a review of standard things of Kuranishi spaces we
need in this article. Except Proposition \ref{67} doesn't appears in
existing literature, all the rest can be found in
\cite{F1}\cite{FOOO}\cite{FOOO4} with more details.
\subsection{Kuranishi Structure and Good Coordinate}
\begin{definition}(Kuranishi neighborhood) Let $M$ be a Haudorff
topological space. A Kuranishi neighborhood of $p\in M$ is a
$5$-tuple $(V_p,E_p,\Gamma_p, \psi_p, s_p)$ such that
\begin{enumerate}
   \item $V_p$ is a smooth manifold (with corners) and $E_p$ is a
   smooth vector bundle over $V_p$.
   \item $\Gamma_p$ is a finite group acting on $E_p \rightarrow
   V_p$.
   \item $s_p$ is a $\Gamma_p$-invariant continuous section of
   $E_p$.
   \item $\psi_p:s^{-1}(0)\rightarrow M$ is continuous and induced
   homeomorphism between $s^{-1}(0)/\Gamma_p$ and a neighborhood of
   $p\in M$.
\end{enumerate}
\end{definition}
For $x\in V_p$, we denote the isotropy subgroup at $x$ by
      \begin{align*}
      (\Gamma_p)_x=\{\gamma\in \Gamma_p|\gamma x=x\}.
      \end{align*}
Let $o_p\in V_p$ such that $s(o_p)=0$ and $\psi_p([o_p])=p$. We will
assume $o_p$ is fixed by all elements of $\Gamma_p$ and thus the
unique point of $V_p$ mapping to $p$ via $\psi_p$.

\begin{definition}(Kuranishi structure)
 A Kuranishi structure on a Hausdorff topological space $M$ is an
 assignment of Kuranishi neighborhood \\$(V_p,E_p,\Gamma_p,s_p,
 \psi_p)$ for each $p\in M$ and a $4$-tuple
 $(V_{pq},\hat{\phi}_{pq},\phi_{pq},h_{pq})$ to each pair $(p,q)$
 where $p\in M$ and $q\in \psi_p(s^{-1}_p(0)/\Gamma_p)$ satisfying:
 \begin{enumerate}
   \item $V_{pq}$ is an open subset of $V_q$ containing $o_q$.
   \item $h_{pq}$ is an injective homomorphism from $\Gamma_q$ to $\Gamma_p$ such that restricts to an isomorphism
   $(\Gamma_q)_x\rightarrow (\Gamma_p)_{\phi_{pq}(x)}$ for any $x\in V_{pq}$
   \item $\phi_{pq}:V_{pq} \rightarrow V_p$ is an $h_{pq}$-equivariant embedding such that it descends to
   injective map
   $\underline{\phi}_{pq}:U_{pq}=V_{pq}/\Gamma_p\rightarrow
   U_q=V_q/\Gamma_q$.
   \item $\hat{\phi}_{pq}:E|_{V_{pq}}\rightarrow E_p$ is an $h_{pq}$-equivariant embedding of
   vector bundles covering $\phi_{pq}$.
   \item $\hat{\phi}_{pq}\circ s_q= s_p \circ \phi_{pq}$
   \item $\psi_q=\psi_p \circ \underline{\phi}_{pq}$ on
   $(s^{-1}_q(0)\cap V_{pq})/\Gamma_q$.
   \item If $r\in \psi_q(s^{-1}_q \cap V_{pq})$, then
   $\hat{\phi}_{pq} \circ \hat{\phi}_{qr}=\hat{\phi}_{pr}$ in a
   neighborhood of $\psi^{-1}(0)$.
   \item The virtual dimension  $\text{dim}V_p-\text{dim}E_p$ of the Kuranishi structure is
   independent of $p$.
   \item (Cocycle condition) If $r\in \psi_q((V_{pq}\cap
   s_q^{-1}(0))/\Gamma_p)$, $q\in \psi_p(s_p^{-1}(0)/\Gamma_p)$,
   then there exists $\gamma^{\alpha}_{pqr}\in \Gamma_p$ for each
   connected component $(\phi_{qr}^{-1}(V_{pq})\cap V_{qr}\cap
   V_{pr})_{\alpha}$ of $\phi_{qr}^{-1}(V_{pq})\cap V_{qr}\cap
   V_{pr}$ such that
     \begin{align*}
    h_{pq}\circ h_{qr}=\gamma^{\alpha}_{pqr}\cdot h_{pr}\cdot
    (\gamma^{\alpha}_{pqr})^{-1}, \\ \phi_{pq}\circ
    \phi_{qr}=\gamma^{\alpha}_{pqr}\cdot \phi_{pr}, \hspace{3mm}
    \hat{\phi}_{pq}\circ \hat{\phi}_{qr}=\gamma^{\alpha}_{pqr}\cdot
    \hat{\phi}_{pr}.
     \end{align*}
 \end{enumerate}
\end{definition}
We call a Hausdorff topological space equipped with Kuranishi
structure a Kuranishi space. Moreover, we will ask
$\{V_{\alpha},E_{\alpha},\Gamma_{\alpha},s_{\alpha},\psi_{\alpha}|\alpha
\in \mathfrak{U}\}$ is a good coordinate system in the following
sense:
  We have a partial order $<$ on $\mathfrak{U}$, such that $\alpha_1
  \leq \alpha_2$ or $\alpha_2 \leq \alpha_1$ for $\alpha_1,\alpha_2
  \in \mathfrak{U}$ if
      \begin{align*}
         \psi_{\alpha_1}(s^{-1}_{\alpha_1}(0)/ \Gamma_{\alpha_1})\cap
         \psi_{\alpha_2}(s^{-1}_{\alpha_2}(0)/\Gamma_{\alpha_2})
         \neq \O
      \end{align*}
Assume that $\alpha_1 < \alpha_2$, then we have an injective
homomorphism $h_{\alpha_1\alpha_2}:\Gamma_{\alpha_1}\rightarrow
\Gamma_{\alpha_2}$, $\Gamma_{\alpha_2}$-invariant open set
$V_{\alpha_1\alpha_2}\subseteq V_{\alpha_1}$ such that there is an
$h_{\alpha_1\alpha_2}$-equivariant embedding of open set
$\phi_{\alpha_1\alpha_2}:V_{\alpha_1\alpha_2}\rightarrow
V_{\alpha_2}$, and $h_{\alpha_1\alpha_2}$-equivariant bundle map
$\hat{\phi}_{\alpha_1\alpha_2}:E_{\alpha_1}|_{V_{\alpha_1\alpha_2}}\rightarrow
E_{\alpha_2}$ cover $\phi_{\alpha_1\alpha_2}$ such that analogue of
above 5,6,7 are satisfied (more details see \cite{FOOO4}).
\begin{rmk}
  \cite{FOOO4} section 7 guarantees the existence of good coordinate
  for any compact Kuranishi space.
\end{rmk}

   Assume $X_i=(V^i, E^i, \Gamma^i, \psi^i, s^i)$ Kuranishi
   structure and $f_i: X_i \rightarrow Y_i$ strongly continuous and
   weakly submersive. Let $Y=\prod Y_i$, $W$ some manifold with
   corners and $f: W \rightarrow Y$ a smooth map. We can construct
   a Kuranishi structure on
     \begin{align*}
       Z=\prod_i X_i\times_Y W
     \end{align*}
   by taking
     \begin{align*}
       V_=\prod_i V^i \times_Y W, \hspace{5mm} E=\prod_i E^i, \hspace{5mm}\Gamma=\prod_i \Gamma^i
     \end{align*} Since $f_i$ are submersions, $V$ is a smooth
     manifold. It is easy to define the section $s$ and the homeomorphism $\psi$ in a natural way.
\begin{definition}(Fibre product Kuranishi structure) Let $X_i$ have Kuranishi structures. Let $f_i:X_i
\rightarrow Y$ to be strongly continuous and weakly submersive. We
define the Kuranishi structure on $Z=X_1 \times_Y X_2$ by identify
$Z=(X_1\times X_2)\times _{Y^2}Y$, where $Y\rightarrow Y^2,
y\rightarrow (y,y)$.
\end{definition}

\subsection{Partition of Unity}
 Fix $\epsilon>0$ sufficiently small and $\chi^{\delta}:\mathbb{R}\rightarrow [0,1]$ smooth
 function such that
     \begin{equation}
        \chi^{\epsilon}(s)=\begin{cases} 0, & \mbox{if } s>\epsilon  \\
                                                 1, & \mbox{if } s<\frac{\epsilon}{2}  \end{cases}  \\
     \end{equation}
For each $x\in V_{\alpha}$, we put
    \begin{align*}
        \mathfrak{U}_{x,+}&=\{\alpha_+|x\in V_{\alpha\alpha_+}, \hspace{1mm} \alpha
        < \alpha_+ \} \\
        \mathfrak{U}_{x,-}&=\{\alpha_-|[x \hspace{3mm}
        mod \hspace{1mm}\Gamma_{\alpha}]\in
        U_{\epsilon}(V_{\alpha_-\alpha}/\Gamma_{\alpha_-}),
        \hspace{1mm}
        \alpha_-< \alpha \}.
    \end{align*}
For $\alpha_- \in \mathfrak{U}_{x,-}$, we take $x_{\alpha_-}\in
N_{V_{\alpha_-\alpha}}V_{\alpha}$ such that
$\mbox{Exp}(x_{\alpha_-})=x$. The following definition generalize the usual notion of partition of unity on a Kuranishi space. Thus, one can start to talk about integration on Kuranishi spaces. 
\begin{definition}
   A system $\{\chi_{\alpha}|\alpha \in \mathfrak{U}\}$ of
   $\Gamma_{\alpha}$-equivariant smooth functions
   $\chi_{\alpha}:V_{\alpha}\rightarrow [0,1]$ with compact support
   is a partition of unity subordinate to the given good (Kuranishi)
   coordinate if for each $x\in V_{\alpha}$,
      \begin{align}
         \chi_{\alpha}(x)+\sum_{\alpha_-\in
         \mathfrak{U}_{x,-}}\chi^{\epsilon}(\parallel x_{\alpha_-}\parallel)\chi_{\alpha_-}(\mbox{Pr}_{\alpha_-\alpha}(x_{\alpha_-}))+\sum_{\alpha_+\in
         \mathfrak{U}_{x,+}}\chi_{\alpha_+}(\phi_{\alpha\alpha_+}(x))=1.
      \end{align}
\end{definition}

\subsection{Multi-Sections and Compatible Perturbations}
 Let
 $(V_{\alpha},E_{\alpha},\Gamma_{\alpha},s_{\alpha},\psi_{\alpha})$
 be a Kuranihis chart of $\mathcal{M}$ and $x\in V_{\alpha}$. Set
 $\mathcal{S}^l(E_{\alpha,x})$ be the $l$-fold symmetric product of
 $E_{\alpha,x}$. There is a natural map
     \begin{align*}
       tm_m:&\mathcal{S}^l(E_{\alpha,x})\rightarrow
       \mathcal{S}^{lm}(E_{\alpha,x}) \\
       &[a_1,\dots,a_l] \mapsto
       [\underbrace{a_1,\dots,a_1}_{\mbox{$m$ copies}},\dots,\underbrace{a_l,\dots,a_l}_{\mbox{$m$ copies}}]
     \end{align*}

A smooth multi-section $s$ of the orbibundle $E_{\alpha}\rightarrow
V_{\alpha}$ is a set of data $s_i$, such that $s_i(x)\in
\mathcal{S}^{l_i}(E_{\alpha,x})$ for an open cover
$\{U_{\alpha,i}\}$ of $V_{\alpha}$ satisfying
   \begin{enumerate}
      \item $U_{\alpha,i}$ are $\Gamma_{\alpha}$-invariant and $s_i$ is $\Gamma_{\alpha}$-equivariant.
      \item If $x\in U_{\alpha,i}\cap U_{\alpha,j}$, then
             \begin{equation}
                tm_{l_j}(s_i(x))=tm_{l_i}(s_j(x))\in
                \mathcal{S}^{l_il_j}(E_{\alpha,\gamma x})
             \end{equation}
      \item There exists local smooth lifting $\tilde{s}$.
   \end{enumerate}

We identify two multi-section $(\{U_i\},\{s_i\},\{l_i\})$,
$(\{U'_i\},\{s'_i\},\{l'_i\})$ if
   \begin{align*}
      tm_{l_j}(s_i(x))=tm_{l_i}(s'_j(x))\in
      \mathcal{S}^{l_il'_j}(E_{\alpha,\gamma x})\hspace{5mm}\mbox{on $U_i\cap U'_j$.}
   \end{align*}
To add up two multi-sections $s^{(1)}$, $s^{(2)}$ together, we first
refine the associated open cover if necessarily such that they
coincides and same automorphism on each open cover. Then we can
define
    \begin{align*}
       +:&\hspace{10mm}\mathcal{S}^{l_1}(E)\times \mathcal{S}^{l_2}(E) \longrightarrow \mathcal{S}^{l_1l_2}(E) \\
        &([a_1,\cdots, a_{l_1}],[b_1,\cdots,b_{l_2}]) \mapsto [a_i+b_j:i=1,\cdots, l_1,b=1,\cdots, l_2]
    \end{align*}
It is easy to check that $+$ is well-defined, associative and
commutative. However, it only has a monoid structure. Another thing
worth mention is that although $C^0(\mathcal{M})$ acts on the sets of
multi-sections, we don't have $(f+g)s=fs+gs$, for $f,g\in
C^0(\mathcal{M})$!

 To make the integration along fibre well-defined, we introduce the
  auxiliary manifold $W_{\alpha}$ which is a finite dimensional smooth
  oriented manifold. We consider the the pull-back bundle
      \begin{equation*}
         \pi^*E_{\alpha} \rightarrow W_{\alpha} \times V_{\alpha}
      \end{equation*}
  and the action of $\Gamma_{\alpha}$ acts on $W_{\alpha}$ is trivial.
\begin{definition}(perturbed multi-section)
   \begin{enumerate}
      \item A $W_{\alpha}$-parametrized family $\mathfrak{s}_{\alpha}$ of
      multi-section $s_{\alpha}$ is a multi-section of $\pi^*E_{\alpha}$.
      \item Fix a metric on the bundle $E$. We say $\mathfrak{s}_{\alpha}$ is
      $\epsilon$-closed to $s$ if each branch $\mathfrak{s}_{\alpha,i,j}$,
      we have
         \begin{equation*}
            |\mathfrak{s}_{\alpha,i,j}(w,\cdots)-s_{\alpha}(\cdots)|_{C^0}<\epsilon
         \end{equation*} in a neighborhood of $x$, for each
         $(w,x)\in W_{\alpha}\times V_{\alpha}$.
      \item $\mathfrak{s}_{\alpha}$ is transverse to $0$ if every branch $\mathfrak{s}_{\alpha,i,j}$ is transverse to $0$.
      \item With above properties, $f_{\alpha}|_{\mathfrak{s}_{\alpha}^{-1}(0)}$ is a
      submersion if restriction to zero locus of each branch is a
      submersion.
   \end{enumerate}

\end{definition}

\begin{lem}
   Assume $f_{\alpha}:V_{\alpha}\rightarrow M$ is a submersion, then there exists $W_{\alpha}$
   such that for any $\epsilon$ there exists a $W_{\alpha}$-parametrized
   family of multi-sections $\mathfrak{s}_{\alpha}$ satisfying:
   \begin{enumerate}
    \item $\mathfrak{s}_{\alpha}$ is $\epsilon$-closed to $s_{\alpha}$.
    \item $\mathfrak{s}_{\alpha}$ is transverse to zero section.
    \item $f|_{\mathfrak{s}_{\alpha}^{-1}(0)}$ is a submersion.
    \item $\mathfrak{s}_{\alpha}(v,0)=s_{\alpha}(0)$.
   \end{enumerate}
 Moreover, if a given
   $\mathfrak{s}_{\alpha}$ satisfies the condition on a neighborhood of
   $\Gamma_{\alpha}$-invariant compact subset in $V$, then we may extend it
   to $V_{\alpha}$.
\end{lem}
\begin{proof}
    We first choose $W_{\alpha}$ to be a vector space with dimension large
    enough such that
       \begin{align*}
          Sur_{\alpha}:W_{\alpha}\times V_{\alpha}\rightarrow
          E_{\alpha}
       \end{align*} is a surjective bundle map (not necessarily
       $\Gamma_{\alpha}$-equivariant). Set
       \begin{align*}
         \mathfrak{s}_{\alpha}^{(1)}(w,x)=Sur_{\alpha}(w,x)+s_{\alpha}(x)
       \end{align*}and
       \begin{align*}
          \mathfrak{s}_{\alpha}^{(2)}(w,x)=[\gamma_1\mathfrak{s}_{\alpha}^{(1)}(w,x),\cdots,
          \gamma_g\mathfrak{s}_{\alpha}^{(1)}(w,x)],
       \end{align*} where $\Gamma_{\alpha}=\{\gamma_1,\cdots,\gamma_g\}$.
       $\mathfrak{s}_{\alpha}^{(2)}$ defines a multisection on $W_{\alpha}\times V_{\alpha}$
       which is transverse to $0$ because the extra dimension from the auxiliary $W_{\alpha}$. Finally,
       $\big(\mathfrak{s}_{\alpha}^{(2)}\big)^{-1}(0)\rightarrow V_{\alpha}$ is
       submersive implies that $f|_{(\mathfrak{s}_{\alpha}^{(2)})^{-1}(0)}$ is a
       submersion.
\end{proof}

\begin{thm} \label{999}
\cite{FOOO}There exists a system of multi-sections
$\mathfrak{s}_{k+1,\beta}$ on \\
$\mathcal{M}_{k+1,\beta}(X,L)$ such that
\begin{enumerate}
  \item They are transverse to $0$.
  \item $ev_0$ induces submersion on the zero sets of
  $\mathfrak{s}_{k+1,\beta}$.
  \item The multi-section is preserved by cyclic permutations of boundary points.
  \item The multi-section $\mathfrak{s}_{k+1,\beta}$ is the
  pull-back of the multi-section $\mathfrak{s}_{k,\beta}$ by the
  forgetful map.
  \item The restriction of multi-section to the boundary
     % \begin{equation}
     %  \partial \mathcal{M}_{k+1,\beta}(X,L)=\bigcup_{1\leq i\leq j+1
     %  \leq k+1} \bigcup_{\beta_1+\beta_2=\beta}
     %  \mathcal{M}_{j-1+1,\beta_1}(X,L) \; {}_{ev_0} \! \times_{ev_i}
     %  \mathcal{M}_{k-j+1,\beta_2}(X,L).
     %\end{equation}
     is compatible.
\end{enumerate}
\end{thm}
%\begin{rmk}
%   We don't need the geometry of the moduli space to construct
%   perturbed multi-sections but need the Kuranishi structure itself.
%\end{rmk}

\begin{rmk}
  We can take $ev_0$ to be weakly submersive and by cyclic symmetry
  each $ev_i$ is weakly submersive. However, if we ask the choice of
  multi-section is compatible with the forgetful map then the map
  $(ev_1,\cdots, ev_k)$ can not be weakly submersive anymore by
  trivial dimensional counting argument. However, $ev_0$ is weakly submersive already
  can pullback differential forms and define the de Rham model.
\end{rmk}

Moreover, we want the family of multi-section
$\mathfrak{s}_{k,\beta}$ satisfies the compatibility with respect to
the good coordinate. Let $\alpha_1<\alpha_2$: Choose an
$\Gamma_{\alpha_2}$-invariant metric on $V_{\alpha_2}$ and consider
the exponential map
   \begin{align*}
      \mbox{Exp}_{\alpha_1\alpha_2}:B_{\epsilon}N_{\alpha_1\alpha_2}V_{\alpha_2}
      \rightarrow V_{\alpha_2}
   \end{align*}
and denotes the image by
$U_{\epsilon}(V_{\alpha_1\alpha_2}/\Gamma_{\alpha_1})\subseteq
V_{\alpha_2}/\Gamma_{\alpha_2}$. We extend the orbibundle
$E_{\alpha_1}$ to
$U_{\epsilon}(V_{\alpha_1\alpha_2}/\Gamma_{\alpha_1})$ by pullback
of projection
   \begin{align*}
      \mbox{Pr}_{V_{\alpha_1\alpha_2}}:U_{\epsilon}(V_{\alpha_1\alpha_2}/\Gamma_1)
      \rightarrow V_{\alpha_1\alpha_2}/\Gamma_{\alpha_1}
   \end{align*}
and also extend the embedding $E_{\alpha_1}\rightarrow
\hat{\phi}_{\alpha_1\alpha_2}^*E_{\alpha_2}E_{\alpha_2}$ to
$U_{\epsilon}(V_{\alpha_1\alpha_2}/\Gamma_{\alpha_1})$. Fix a
$\Gamma_{\alpha}$-invariant invariant inner product on $E_{\alpha}$
and we have a bundle isomorphism
    \begin{align*}
       E_{\alpha_2}\cong E_{\alpha_1}\oplus
       \frac{\hat{\phi}_{\alpha_1\alpha_2}^*E_{\alpha_2}}{E_{\alpha_1}}
    \end{align*} over $U_{\epsilon}(V_{\alpha_1\alpha_2}/\Gamma_1)$.
One might need to use implicit function theorem and tangent bundle
to modify $\mbox{Pr}_{\alpha_1\alpha_2}$ such that
   \begin{align*}
      ds_{\alpha_2}(\tilde{y}\hspace{3mm} mod
      \hspace{1mm}TV_{\alpha_1}) \equiv
      s_{\alpha_2}(y)\hspace{3mm}mod \hspace{1mm}E_{\alpha_1},
   \end{align*} if $y=\mbox{Exp}_{\alpha_1\alpha_2}(\tilde{y})\in
   U_{\epsilon}(V_{\alpha_1\alpha_2}/\Gamma_{\alpha_1})$.

\begin{definition}
  If $\alpha_1<\alpha_2$, then $\mathfrak{s}_{\alpha_2}$ is compatible with $\mathfrak{s}_{\alpha_1}$ if for each
  $y=\mbox{Exp}_{\alpha_1,\alpha_2}(\tilde{y})\in
  U_{\epsilon}(V_{\alpha_1\alpha_2}/\Gamma_{\alpha_1})$, we have
     \begin{align} \label{8}
        \mathfrak{s}_{\alpha_2}(\tilde{y})=\mathfrak{s}_{\alpha_1}(Pr(\tilde{y}))\oplus
        ds_{\alpha_2}(\tilde{y} \hspace{3mm} mod \hspace{1mm}
        TV_{\alpha_1}).
     \end{align} via isomorphism $E_{\alpha_2}\cong
     E_{\alpha_1}\oplus
     \frac{\phi_{\alpha_1\alpha_2}^*TV_{\alpha_2}}{TV_{\alpha_1}}$
     assuming the moduli space is oriented.
\end{definition}

\subsection{Smooth Correspondence} Let $\mathcal{M}$ with Kuranishi
structure $(V_{\alpha},E_{\alpha},
\Gamma_{\alpha},s_{\alpha},\psi_{\alpha})$,
$f^s:\mathcal{M}\rightarrow N_s$ strongly continuous and
$f^t:\mathcal{M}\rightarrow N_t$ strongly continuous and weakly
submersive. Here we assume that $N_s$, $N_t$ are both smooth manifolds
(possibly with boundaries or corners). We will define the smooth
correspondence
   \begin{align*}
      Corr_*(\mathcal{M};f^s,f^t):\Lambda^d(N_s)\rightarrow
      \Lambda^{d+dimN_t-dimN_s}(N_t),
   \end{align*} where $\Lambda^k(X)$ denotes the space of smooth $k$-forms on the manifold $X$. 
 We first take a compatible continuous family of multi-sections
 $\mathfrak{s}_{\alpha}=\{\mathfrak{s}_{\alpha,i,j}|j=1,\cdots
 l_i\}$ satisfies the lemma and $\tilde{\mathfrak{s}}_{\alpha,i,j}$
 is the local smooth lifting. Let $\rho \in \Lambda(N_s)$. Consider
 a branch $\tilde{\mathfrak{s}}_{\alpha,i,j}$ as a section of $E_{\alpha}$
 over $U_{\alpha,i}\times W_{\alpha}$. We choose a volume form
 $\omega_{\alpha}$ on $W_{\alpha}$ with total mass $1$ and support on an $\epsilon$-neighborhood of $0\in W_{\alpha}$ and partition
 of unity $\chi_i$ for open cover $\{U_{\alpha,i}\}_i$. Then
     \begin{equation*}
       \frac{1}{\#\Gamma_{\alpha}}\sum_i\sum_{j=1}^{l_i} (f^t_{\alpha}\circ \pi_{\alpha}|_{\tilde{\mathfrak{s}}_{\alpha,i,j}^{-1}(0)})_!\frac{1}{l_i}(\chi_i\chi_{\alpha}(f^s_{\alpha})^*\rho
       \wedge \omega_{\alpha})|_{\tilde{\mathfrak{s}}_{\alpha,i,j}^{-1}(0)}
     \end{equation*}
defines the $U_{\alpha,i}$ part of the smooth correspondence
$Corr_*(\mathcal{M};f^s,f^t)(\rho)$ and we use partition unity
$\chi_{\alpha}$ for summing various $\chi_{\alpha}$ to glue them
together.

\begin{rmk}
   The definition of smooth correspondence $Corr_*(\mathcal{M};f^s,f^t)$
   only depends on the Kuranishi structure
   $(V_{\alpha},E_{\alpha},\Gamma_{\alpha},s_{\alpha},\psi_{\alpha})$,
   the auxiliary $(W_{\alpha},\omega_{\alpha})$, the perturbation
   $\mathfrak{s}_{\alpha}$, $f_{\alpha}$ but not depends
   on other choices.
\end{rmk}
\begin{rmk}
  Although one may not be able to exclude the case
  $\mathcal{M}_{\beta}$ has infinitely many components. However,
  $Corr_*(\mathcal{M}_{\beta};tri,tri)$ is always finite.
\end{rmk}

We have the following version of Stoke's theorem composition formula for smooth correspondences.
\begin{prop} \label{2} 
  \begin{equation}
  d \circ Corr_*(\mathcal{M};f^s,f^t)-Corr_*(\mathcal{M};f^s,f^t)\circ
  d=Corr_*(\partial \mathcal{M};f^s,f^t).
  \end{equation}
\end{prop}

\begin{prop} \label{3}(composition formula) Let $\rho_i \in
\Lambda(N_s^i)$, $i=1,2$ with the following diagram,
\begin{equation*}
   \xymatrix{\mathcal{M}=\mathcal{M}_1 \times \mathcal{M}_2 \ar[d] \ar[rr] & & \mathcal{M}_1 \ar[d]^{f^{1,t}} \ar[r]^{f^{1,s}} & N_s^1  \\
              \mathcal{M}_2 \ar[d]^{f^{2,t}} \ar[r] & N_s^2 \times N_t^1 \ar[r] & N_t^1 \\
              N_t^2  }
\end{equation*}
then we have
  \begin{align}
    &Corr_*(\mathcal{M};f^s,f^t)(\rho_1 \times \rho_2) \nonumber \\
    =&Corr_*(\mathcal{M}_2;f^{2,s},f^{2,t})(Corr_*(\mathcal{M}_1;f^{1,s},f^{1,t})(\rho_1)\times
    \rho_2)
  \end{align}
\end{prop}

At the end, we have an open analogue of divisor axiom though we
never use it.
\begin{prop} \label{67}Assume $\beta$ is primitive and $L$ doesn't fall on $W'_{\beta}$.
Let $\rho \in \Lambda^1(L\times S^1_{\vartheta})$, $d\rho=0$, and
$\mathfrak{forget}:\mathcal{M}_{1,\beta}(\mathfrak{X},L)\rightarrow
\mathcal{M}_{0,\beta}(\mathfrak{X},L)$, then
   \begin{align*}
      Corr_*(\mathcal{M}_{1,\beta}(L);(ev_0,ev_{\vartheta}),\mathfrak{forget})(\rho)=\bigg(\int_{\partial\beta}\rho \bigg) \cdot m_{-1,\beta}(L)
   \end{align*}
\end{prop}
\begin{proof}
\begin{align*}
         &\int_{L\times S^1_{\vartheta}}m_{0,\beta}\wedge \rho\\
         =&\int_{L\times S^1_{\vartheta}} Corr_*(\mathcal{M}_{1,\beta}(L);tri,(ev_0,ev_{\vartheta}))(1)\wedge \rho \\
         =&Corr_*(\mathcal{M}_{1,\beta}(L);(ev_0,ev_{\vartheta}),tri)(\rho)         \\
         =&Corr_*(\mathcal{M}_{1,\beta}(L);(ev_0,ev_{\vartheta}),tri\circ \mathfrak{forget})(\rho)         \\
         =&Corr_*(\mathcal{M}_{0,\beta}(L);id,tri)(Corr_*(\mathcal{M}_{1,\beta}(L_t);(ev_0,ev_{\vartheta}),\mathfrak{forget})(\rho)) \\
         =&Corr_*(\mathcal{M}_{0,\beta}(L);tri,tri)(1)\cdot \int_{\partial\beta} \rho\\
         =&m_{-1,\beta} \cdot  \int_{\partial\beta} \rho
\end{align*} The second equality is from the projection formula of integration along fibres. The fourth equality is a bit subtle.
The smooth correspondence is originally defined only when the target
is a smooth manifold. However, the compatibility of forgetful map
guarantees the definition extends to this case.
\end{proof}

Department of Mathematics, Stanford University\\
E-mail address: yslin221@stanford.edu

\end{document}